\documentclass[a4paper,12pt]{article}

\usepackage{amsthm}
\usepackage{latexsym}
\usepackage{dsfont}
\usepackage{bbm}
\usepackage{amssymb}
\usepackage{amsmath}
\usepackage{graphicx}

\numberwithin{equation}{section}

\theoremstyle{plain}
\newtheorem{Theorem}{Theorem}[section]
\newtheorem{Lemma}[Theorem]{Lemma}

\newtheorem{Prop}[Theorem]{Proposition}
\theoremstyle{remark}
\newtheorem{Rem}[Theorem]{Remark}
\theoremstyle{definition}
\newtheorem{Exa}[Theorem]{Example}

\DeclareMathOperator{\N}{\mathbb{N}}
\DeclareMathOperator{\Z}{\mathbb{Z}}
\DeclareMathOperator{\Q}{\mathbb{Q}}
\DeclareMathOperator{\R}{\mathbb{R}}

\renewcommand{\i}{\mathrm{i}}

\DeclareMathOperator{\Prob}{\mathbb{P}}
\DeclareMathOperator{\Var}{\mathrm{Var}}
\DeclareMathOperator{\E}{\mathbb{E}}
\DeclareMathOperator{\A}{\mathcal{A}}
\DeclareMathOperator{\B}{\mathcal{B}}

\DeclareMathOperator{\Exp}{\mathrm{Exp}}

\DeclareMathOperator{\1}{\mathbbm{1}}

\DeclareMathOperator{\dx}{\mathrm{d} \mathit{x}}
\DeclareMathOperator{\dy}{\mathrm{d} \mathit{y}}
\DeclareMathOperator{\ds}{\mathrm{d} \mathit{s}}
\DeclareMathOperator{\dt}{\mathrm{d} \mathit{t}}
\DeclareMathOperator{\du}{\mathrm{d} \mathit{u}}
\DeclareMathOperator{\dv}{\mathrm{d} \mathit{v}}

\DeclareMathOperator{\sign}{\mathrm{sign}}

\title{Limit theorems for renewal shot noise processes with decreasing response functions}
\date{\today}

\author{Alexander Iksanov\footnote{Faculty of Cybernetics, National T.\
Shevchenko University of Kyiv, 01033 Kyiv, Ukraine, e-mail:
iksan@univ.kiev.ua},\ Alexander Marynych\footnote{Faculty of
Cybernetics, National T.\
Shevchenko University of Kyiv, 01033 Kyiv, Ukraine, e-mail: marynych@unicyb.kiev.ua}   \\
and Matthias Meiners\footnote{ Institut f\"{u}r Mathematische Statistik,
Westf\"{a}lische Wilhelms-Universit\"{a}t M\"{u}nster, 48149
M\"{u}nster, Germany, e-mail: mmeiners@math.uni-muenster.de.}}

\begin{document}

\thispagestyle{empty}
\maketitle

\begin{abstract}
We consider shot noise processes $(X(t))_{t \geq 0}$ with
deterministic response function $h$ and the shots occurring at the
renewal epochs $0= S_0 < S_1 < S_2 \ldots$ of a zero-delayed
renewal process. We prove convergence of the finite-dimensional
distributions of $(X(ut))_{u \geq 0}$ as $t \to \infty$ in
different regimes. If the response function $h$ is directly
Riemann integrable, then the finite-dimensional distributions of
$(X(ut))_{u \geq 0}$ converge weakly as $t \to \infty$. Neither
scaling nor centering are needed in this case. If the response
function is eventually decreasing, non-integrable with an
integrable power, then, after suitable shifting, the
finite-dimensional distributions of the process converge. Again,
no scaling is needed. In both cases, the limit is identified. If
the distribution of $S_1$ is in the domain of attraction of an
$\alpha$-stable law and the response function is regularly varying
at $\infty$ with index $-\beta$ (with $0\leq \beta < 1/\alpha$ or
$0\leq \beta \leq \alpha$, depending on whether $\E S_1 < \infty$
or $\E S_1 = \infty$), then scaling is needed to obtain weak
convergence of the finite-dimensional distributions of $(X(ut))_{u
\geq 0}$. The limits are fractionally integrated stable L\'evy
motions if $\E S_1 < \infty$ and fractionally integrated inverse
stable subordinators if $\E S_1 = \infty$.
\end{abstract}

\hphantom{\emph{2010 Mathematics Subject Classification:}}      Secondary:  60G55               


\newpage
\tableofcontents

\section{Introduction}  \label{sec:intro}
Continuing the line of research initiated in \cite{Iksanov:2012}, in the present paper, we are
investigating the convergence of the finite-dimensional distributions of renewal shot noise processes.
Some work on convergence of renewal shot noise processes has already been done by other authors \cite{Iglehart:1973,Miller:1974,Rice:1977}.
However, their results do not intersect with those obtained here.
The special case of Poisson shot noise has received more attention, see \textit{e.g.}\
\cite{Heinrich+Schmidt:1985,Klueppelberg+Mikosch:1995a,Klueppelberg+Mikosch+Schaerf:2003,Lane:1984,Rosinski:2001}.

Initially, shot noise processes were introduced to model the current induced by a stream of electrons arriving at the anode of a vacuum tube \cite{Schottky:18}.
Since their first appearance in the literature,
shot noise processes have been used to model rainfall \cite{Rodriguez-Iturbe+Cox+Isham:1987,Waymire+Gupta:1981},
stream- and riverflows \cite{Lawrance+ Kottegoda:1977,Weiss:1977}, earthquake occurences \cite{Vere-Jones:1970},
computer failures \cite{Lewis:1964}, traffic noise \cite{Marcus:1975},
delay in claim settlement in insurance \cite{Klueppelberg+Mikosch:1995a,Klueppelberg+Mikosch:1995b},
and several processes in finance \cite{Samorodnitsky:1996},
to name but a few.
The recent paper \cite{Alsmeyer+Iksanov+Meiners:2012} offers a list of further references.

We now start with the mathematical setup.
Let $\xi_1, \xi_2, \ldots$ be a sequence of independent copies of a positive random variable $\xi$.
The distribution of $\xi$ is denoted by $F$.
By $(S_k)_{k \in \N_0}$
(where $\N_0 := \N \cup \{0\}$)
we denote the random walk with initial position $S_0 := 0$ and increments $S_k - S_{k-1} = \xi_k$, $k \in \N$.
The corresponding renewal counting measure is denoted by $N$, that is,
\begin{equation*}
N ~=~ \sum_{k \geq 0} \delta_{S_k},
\end{equation*}
where $\delta_x$ denotes the Dirac distribution concentrated at $x$.
We write $N(t)$ for $N([0,t])$, $t \geq 0$.
By $U$ we denote the intensity measure corresponding to $N$.
Hence, $U(B) := \E N(B)$ for Borel sets $B \subseteq \R$.
We write $U(t)$ for $U([0,t])$, $t \geq 0$.
Throughout the paper, we denote by $h$ a real-valued, measurable and locally bounded function on
the positive half-line $\R_+ = [0,\infty)$.
Further, let
\begin{equation}    \label{eq:X(t)}
X(t)        ~:=~    \sum_{k=0}^{N(t)-1} h(t-S_k)    ~=~ \int_{[0,t]} h(t\!-\!y) \, N(\dy),  \qquad  t \geq 0.
\end{equation}
The stochastic process $(X(t))_{t \geq 0}$ is called \emph{renewal shot noise process}, $h$ is called \emph{response function}.

In the recent paper \cite{Iksanov:2012}, functional limit theorems
for $(X(ut))_{u \geq 0}$ are derived in the case that the response
function is eventually increasing\footnote{We call a function $h$
increasing if for all $s < t$ we have $h(s) \leq h(t)$. We call
$h$ \emph{strictly} increasing if for $s < t$ we have $h(s) <
h(t)$. Analogously, $h$ is said to be decreasing if $s < t$
implies $h(s) \geq h(t)$ and it is said to be \emph{strictly}
decreasing if $s<t$ implies $h(s) > h(t)$. $h$ is said to be
\emph{eventually increasing/decreasing} if, for some $t_0 \geq 0$,
$h$ is decreasing/decreasing on $[t_0,\infty)$.}. The motivation
behind the present work in general and the use of the specific
time scaling in particular is the following. First, we intend to
obtain counterparts of the results derived in \cite{Iksanov:2012}
for functions $h$ that are eventually decreasing. Second, in
\cite{Iksanov+Marynych+Vatutin:2013} some results of this paper
are used to prove the finite-dimensional convergence of the number
of empty boxes in the Bernoulli sieve (see
\cite{Gnedin+Iksanov+Marynych:2010} for the definition and
properties of the Bernoulli sieve). Of course, transformations of
time other than $ut$ may also lead to useful limit theorems. For
instance, convergence of $(X(t+u))_{u \geq 0}$ may be worth
investigating. Yet another transformation of time has proved
important \cite{Iglehart:1973}, where one only rescales the time
of the underlying renewal process, whereas the deterministic
component runs in its original time scale.

Unlike in \cite{Iksanov:2012}, where functional limit theorems are
derived, in the paper at hand, we investigate convergence of
finite-dimensional distributions only. In the situations of part
(a) of Theorem \ref{Thm:dRi h}, Theorem \ref{Thm:decreasing h} and part (A3) of Theorem \ref{Thm:non-endogenous & mu<infty},   
any versions of the limiting processes do not take values in the
Skorokhod space of right-continuous functions with left limits
which excludes the possibility that a classical functional limit
theorem holds. In the other cases, where the limiting processes do
take values in the Skorokhod space, convergence in a functional
space remains open for future research.

\section{Main results}  \label{sec:main results}

As mentioned in the introduction,
we centre our attention on the case of eventually decreasing response functions.
Let us remark right away that the situations where $h$ is eventually decreasing and
either $\lim_{t \to \infty} h(t) = c \in (-\infty, 0)$ or $\lim_{t \to \infty} h(t)=-\infty$
and $-h(t)$ is regularly varying at $\infty$ with some index $\beta \geq 0$ are covered by Theorem 1.1 in \cite{Iksanov:2012}.
See Remark \ref{Rem:Iksanov:2012} for more details.
Keeping this in mind our main results mainly treat eventually decreasing functions with non-negative limit at infinity.

Our results fall into two fundamentally different categories. The
first type of results considers finite-dimensional convergence of
the process $(X(ut))_{u \geq 0}$ as $t \to \infty$ when no scaling
(normalization) is needed. In this case, all randomness in the
limiting process can be described in terms of copies of the
stationary renewal counting process $N^*$ to be introduced below.
In the second type of results scaling is needed.
Then some of the fine features of the process $(X(ut))_{u
\geq 0}$ vanish in the limit and robust limit theorems are
obtained in the sense that the limiting behavior only depends on
the asymptotic behavior of $h$ and the tails of $S_1$. The
limiting processes are stochastic integrals with the integrators being
$\alpha$-stable L\'evy motions or inverse stable subordinators.

For the formulation of our main results, we need to introduce
further notation. First, let $\mu := \E \xi$. Since $\xi > 0$
a.s., $\mu$ is well-defined but may equal $+\infty$. Whenever $\mu
< \infty$ and the law of $\xi$ is non-lattice, we denote by
$S_0^*$ a positive random variable which is independent of $(\xi_k)_{k \in \N}$ and has distribution function
\begin{equation}    \label{eq:S_0^*}
F^*(t)  ~:=~    \Prob\{S_0^* \leq t\}   ~:=~    \frac{1}{\mu} \int_0^t \Prob\{\xi > x\} \dx,    \qquad  t \geq 0.
\end{equation}
We set $S_k^* := S_0^* + S_k$, $k \in \N_0$.
The associated renewal counting process $N^* := \sum_{k \geq 0} \delta_{S_k^*}$ has stationary increments.
Equivalently, the corresponding intensity measure $U^*(\cdot) := \E N^*(\cdot)$ satisfies $U^*(\dx) = \mu^{-1} \dx$,
see Subsection \ref{subsec:stationary renewal processes and coupling}
and \cite[Section III.1.2]{Lindvall:1992} or \cite[Section II.9]{Thorisson:2000} for further information.

For stochastic processes $(Z_t(u))_{u \geq 0}$, $t \geq 0$ and
$(Z(u))_{u \geq 0}$, we write $Z_t(u)
\stackrel{\mathrm{f.d.}}{\Rightarrow} Z(u)$ as $t \to \infty$ to
denote weak convergence of finite-dimensional distributions,
\textit{i.e.}, for any $n \in \N$ and any selection $0 < u_1 <
\ldots < u_n < \infty$
\begin{equation*}
(Z_t(u_1), \ldots, Z_t(u_n)) ~\stackrel{\mathrm{d}}{\to}~ (Z(u_1),\ldots, Z(u_n))   \quad   \text{as } t \to \infty
\end{equation*}
where $\stackrel{\mathrm{d}}{\to}$ denotes convergence in distribution.

\subsection{Limit theorems without scaling} \label{subsec:endogenous thms}

If $\mu < \infty$ and $F$ is non-lattice, define
\begin{equation}    \label{eq:X*}
X^* ~:=~ \lim_{t \to \infty} \sum_{k \geq 0} h(S_k^*) \1_{\{S_k^* \leq t\}}
\end{equation}
whenever the limit exists as a limit in probability and is a.s.\ finite.
In this case, denote by $(X^*(u))_{u \geq 0}$ a family of i.i.d.\ copies of $X^*$.

Our first result states that if $h$ is
directly Riemann integrable\footnote{A definition of \emph{direct Riemann integrability} is provided in the proof of Proposition \ref{Prop:X* well-defined}} (d.R.i.),
then the finite-dimensional distributions of $(X(ut))_{u > 0}$ converge weakly to the finite-dimensional distributions of the process $(X^*(u))_{u > 0}$.

\begin{Theorem} \label{Thm:dRi h}
Let $h:\R_+ \to \R$ be directly Riemann integrable and $F$ be non-lattice.
\begin{itemize}
    \item[(a)]
        If $\mu < \infty$, then the random series $X^*$ converges a.s.\ and
        \begin{equation}    \label{eq:X(t) when h dRi}
        X(ut)   ~\stackrel{\mathrm{f.d.}}{\Rightarrow}~  X^*(u) \quad   \text{as } t \to \infty.
        \end{equation}
    \item[(b)]
        If $\mu = \infty$, then
        \begin{equation*}
        X(t)    ~\stackrel{\Prob}{\to}~ 0   \quad   \text{as } t \to \infty.
        \end{equation*}
\end{itemize}
\end{Theorem}

\begin{Rem} \label{Rem:dRi h}
Since the focus of this paper is on eventually decreasing response functions
it is worth mentioning that Theorem \ref{Thm:dRi h} covers the case when $h$
is eventually decreasing and improperly Riemann integrable
since any such function is necessarily d.R.i.
\end{Rem}

\begin{Exa} \label{Exa:increments of renewal process}
Assume that $\mu < \infty$ and that $F$ is non-lattice.
For fixed $0 \leq a < b$, choose  $h(t) := \1_{[a,b)}(t)$, $t \geq 0$ as response function.
Then Theorem \ref{Thm:dRi h} implies that
\begin{equation*}
N(t-a)-N(t-b)   ~=~ \sum_{k \geq 0} h(t-S_k)    ~\stackrel{\mathrm{d}}{\to}~    N^*(b-a)    \quad   \text{as } t \to \infty.
\end{equation*}
\end{Exa}

Though the one-dimensional convergence in Theorem \ref{Thm:dRi h} is quite expected, a rigorous proof is necessary.
It is tempting to conclude this from Theorem 6.1 in \cite{Miller:1974}.
However, the cited theorem does not hold in the generality stated there.
Regularity assumptions on the function $h$ in the theorem above cannot be avoided.
This will be demonstrated in Example \ref{Exa:Miller's theorem not true}.

Our second result is an extension of Theorem \ref{Thm:dRi h} to the situation where $h$ is not integrable.
In this case, $X^*$ is not well-defined
and in order to still obtain non-trivial finite-dimensional convergence of the process $(X(ut))_{u \geq 0}$ as $t \to \infty$
centering is needed.
When $\mu < \infty$ and $F$ is non-lattice, define
\begin{equation}    \label{eq:X*_o}
X^*_{\circ} ~:=~ \lim_{t \to \infty} \, \bigg( \sum_{k \geq 0} h(S_k^*) \1_{\{S_k^* \leq t\}} - \frac{1}{\mu} \int_0^t h(y) \dy\bigg).
\end{equation}
Whenever $X^*_{\circ}$ exists as the limit in probability and is a.s.\ finite,
denote by $(X^*_{\circ}(u))_{u \geq 0}$ a family of i.i.d.\ copies of $X^*_{\circ}$.

\begin{Theorem} \label{Thm:decreasing h}
Assume that $F$ is non-lattice. Let $h:\R_+ \to \R$ be locally
bounded, a.e.\ continuous, eventually decreasing and
non-integrable.
\begin{itemize}
    \item[(C1)]
        Suppose $\sigma^2:= \Var \xi < \infty$ and
        \begin{equation}    \label{eq:int h^2 < infty}
        \int_0^{\infty} h(y)^2 \dy  ~<~ \infty.
        \end{equation}
        Then $X^*_{\circ}$ exists as the limit in $\mathcal{L}^2$ in \eqref{eq:X*_o} and
        \begin{equation}    \label{eq:X_o(ut)->}
        X(ut) - \mu^{-1} \int_{0}^{ut} h(y) \dy   ~\stackrel{\mathrm{f.d.}}{\Rightarrow}~ X^*_{\circ}(u)  \quad   \text{as } t \to \infty.
        \end{equation}
        \eqref{eq:X_o(ut)->} also holds with $\mu^{-1} \int_{0}^{ut} h(y) \dy$ replaced by $\E X(ut)$.
\end{itemize}
For the rest of the theorem, assume that $h$ is eventually twice differentiable\footnote{$h$ is called \emph{eventually twice differentiable} if there exists a $t_0 \geq 0$ such that $h$ is twice differentiable on $(t_0,\infty)$} and that $h''$ is eventually nonnegative.
\begin{itemize}
    \item[(C2)]
        Suppose $\E \xi^r < \infty$ for some $1 < r < 2$. If
        there exists an $a>0$ such that $h(y) > 0$ for $y\geq a$ and
        \begin{equation}    \label{eq:int h^r<infty}
        \int_a^\infty h(y)^r \dy < \infty,
        \end{equation}
        and\footnote{
        If $h''$ is eventually monotone, then
        \eqref{eq:h''=Ot^(-2-1/r)} and \eqref{eq:h''=O(t^(-2)c^(-1)(t)} are consequences
        of \eqref{eq:int h^r<infty} and \eqref{eq:int h^alphal(1/h)<infty}, respectively.}
        \begin{equation}    \label{eq:h''=Ot^(-2-1/r)}
        h''(t)=O(t^{-2-1/r})    \quad   \text{as } t \to \infty,
        \end{equation}
        then $X^*_{\circ}$ is well-defined as
        the a.s.\ limit in \eqref{eq:X*_o}. Further, \eqref{eq:X_o(ut)->} holds.
    \item[(C3)]
        Suppose $\Prob\{\xi>x\} \sim x^{-\alpha}\ell(x)$ as $x \to \infty$
        for some $1 < \alpha < 2$ and some $\ell$ slowly varying at $\infty$. If
        there exists an $a>0$ such that $h(y) > 0$ for $y\geq a$ and
        \begin{equation}    \label{eq:int h^alphal(1/h)<infty}
        \int_a^\infty h(y)^{\alpha} \ell(1/h(y)) \dy ~<~    \infty,
        \end{equation}
        and
        \begin{equation}    \label{eq:h''=O(t^(-2)c^(-1)(t)}
        h''(t)=O(t^{-2}c(t)^{-1})   \quad   \text{as } t \to \infty
        \end{equation}
        where $c(t)$ is any positive function such that
        \begin{equation}    \label{eq:t(l(c(t))/c(t)^alpha->1}
        \lim_{t \to \infty} \frac{t \ell(c(t))}{c(t)^{\alpha}} = 1,
        \end{equation}
        then $X^*_{\circ}$ exists as the limit in probability in \eqref{eq:X*_o} and \eqref{eq:X_o(ut)->} holds.
\end{itemize}
\end{Theorem}

\begin{Rem} \label{Rem:decreasing h}
The cases (C2) and (C3) of Theorem \ref{Thm:decreasing h}
impose, besides conditions on the law of $\xi$, smoothness and integrability conditions on $h$.
The smoothness conditions may seem rather restrictive
but are an essential ingredient of our proof
which is based on an idea we have learned in \cite{Kersting:2012}.
We believe that in each assertion (C1)--(C3), given the respective assumption on the law of $\xi$,
the corresponding integrability condition is close to optimal.
In a sense, the extra smoothness conditions in (C2) and (C3) are the price one has to pay for this precision
(we do not claim, however, that the smoothness conditions are indeed necessary).
For comparison, we mention the following.
Assuming nothing beyond the standing conditions of the theorem
(in particular, not requiring $h$ to be differentiable)
we can prove that \eqref{eq:X_o(ut)->} holds
under more restrictive integrability conditions:
\begin{equation*}
\E \xi^r ~<~ \infty \quad \text{and} \quad
\int_{[b,\,\infty)}y^{1/r} \, \mathrm{d}(-h(y)) ~<~ \infty
\end{equation*}
for some $1 < r <2$, and
\begin{equation*}
\Prob\{\xi>x\} ~\sim~ x^{-\alpha}\ell(x)    \quad   \text{as } x
\to \infty \quad \text{and} \quad \int_{[b,\,\infty)} c(y)
\mathrm{d}(-h(y)) ~<~ \infty
\end{equation*}
for some $1 < \alpha < 2$ and some $\ell$ slowly varying at
$\infty$, respectively, where $b\geq 0$ is such that $h$ is
decreasing on $[b,\infty)$. Without going into the details, we
mention that the conditions $\int_{[b,\,\infty)}y^{1/r}
\mathrm{d}(\!-h(y))<\infty$ and $\int_{[b,\,\infty)} c(y)
\mathrm{d}(\!-h(y)) < \infty$ are sufficient for the a.s.\ {\it
absolute} convergence of the improper integral
$\int_{[b,\,\infty)}(N^*(y)\!-\!y/\mu)\mathrm{d}(\!-h(y))$,
whereas the conditions \eqref{eq:int h^r<infty} and \eqref{eq:int
h^alphal(1/h)<infty} are sufficient for the a.s.\ {\it
conditional} convergence of that integral.
\end{Rem}

\begin{Exa} \label{Exa:Miller's theorem not true}
Let $h(t) := (1 \wedge 1/t^2)\1_{\Q}(t)$, $t \geq 0$, where $\Q$
denotes the set of rationals. Let the distribution of
$\xi$ be such that $\Prob\{\xi \in \Q \cap (0,1]\} = 1$ and
$\Prob\{\xi = r\} > 0$ for all $r \in \Q \cap (0,1]$. Then the
distribution of $\xi$ is non-lattice. From \eqref{eq:S_0^*} we
conclude that the distribution of $S^*_0$ is continuous w.r.t.\
Lebesgue measure and concentrated on $[0,1]$.
Therefore, $\Prob\{S^*_0 \in [0,1] \cap (\R \setminus \Q)\} = 1$.
Since the $\xi_k$ take rational values a.s., all $S^*_k$
take irrational values on a set of probability $1$.
Consequently, $X^* = \sum_{k \geq 0} h(S^*_k) = 0$ a.s.
On the other hand, in the given situation, $X(t)$ does not converge to $0$
in distribution when $t$ approaches $+\infty$ along a sequence of
rationals. In fact, for $t \in \Q$, $X(t) = Y(t)$ a.s.\ where
$Y(t) = \sum_{k \geq 0} f(t-S_k) \1_{\{S_k \leq t\}}$ with $f(t) = 1 \wedge 1/t^2$ for $t \geq 0$.
Therefore, from Theorem \ref{Thm:dRi h} we conclude that
\begin{equation*}
X(t)    ~=~ Y(t)    ~\stackrel{\mathrm{d}}{\to}~    \sum_{k \geq 0} f(S^*_k)    \quad   \text{as } t \to \infty,\ t \in \Q.
\end{equation*}
Plainly, the latter random variable is positive a.s.
\end{Exa}

Example \ref{Exa:Miller's theorem not true} does not only demonstrate that Theorem 6.1 in \cite{Miller:1974}
fails when assuming only that $\lim_{t \to \infty} h(t) = 0$. It moreover shows that also Lebesgue integrability
of $h$ is not enough to ensure \eqref{eq:X(t) when h dRi} to hold.
A stronger assumption such as the direct Riemann integrability of $h$ is needed.

\subsection{Limit theorems with scaling}    \label{subsec:non-endogenous thms}

In the case when scaling is needed our main assumption on the response function $h$
is regular variation at $\infty$:
\begin{equation}    \label{eq:h in R^-beta}
h(t)    ~\sim~  t^{-\beta} \ell_h(t)    \quad   \text{as }t \to \infty
\end{equation}
for some $\beta \geq 0$ and some $\ell_h$ slowly varying at $\infty$.
Recall that $\ell_h(t) > 0$ for all $t \geq 0$ by the definition of slow variation, see \textit{e.g.}~\cite{Bingham+Goldie+Teugels:1989}.
Note further that the functions $h$ with $\lim_{t \to \infty} h(t) = b \in (0,\infty)$ are
covered by condition \eqref{eq:h in R^-beta} with $\beta=0$ and
$\lim_{t \to \infty} \ell_h(t)=b$.

\begin{Theorem} \label{Thm:non-endogenous & mu<infty}
Let $h:\R_+ \to \R$ be locally bounded, measurable and eventually decreasing.
Further, let $(W_{2}(u))_{u\geq 0}$ denote a standard Brownian motion
and, for $1 < \alpha < 2$, let $(W_{\alpha}(u))_{u\geq 0}$ denote an $\alpha$-stable L\'{e}vy motion\footnote{
For the definition of $\alpha$-stable L\'evy motion see \textit{e.g.}~\cite[Example 3.1.3 or Definition 7.5.1]{Samorodnitsky+Taqqu:1994}.}
such that $W_{\alpha}(1)$ has the characteristic function
\begin{equation}    \label{eq:stable ch f}
z \mapsto \exp\big\{-|z|^\alpha \Gamma(1\!-\!\alpha)(\cos(\pi\alpha/2)+\i\sin(\pi\alpha/2)\, {\sign}(z))\big\}, \ z \in \R
\end{equation}
with $\Gamma(\cdot)$ denoting the gamma function.
\begin{itemize}
    \item[(A1)]
        Suppose $\sigma^2:=\Var \xi<\infty$.
        If \eqref{eq:h in R^-beta} holds for some $\beta\in (0,1/2)$, then
        \begin{equation*}
        \frac{X(ut)-\mu^{-1} \int_{0}^{ut} h(y)\dy}{\sqrt{\sigma^2\mu^{-3}t}h(t)}
        ~\stackrel{\mathrm{f.d.}}{\Rightarrow}~ \int_{[0,\,u]}(u-y)^{-\beta} \, \mathrm{d}W_2(y)
        \quad   \text{as } t\to\infty,
        \end{equation*}
        whereas if \eqref{eq:h in R^-beta} holds
        with $\beta=0$, the limiting process is $(W_2(u))_{u \geq 0}$.
    \item[(A2)]
        Suppose $\sigma^2=\infty$ and that, for some $\ell$ slowly varying at $\infty$,
        \begin{equation*}
        \E[\xi^2 \1_{\{\xi \leq t\}}]    ~\sim~  \ell(t) \quad   \text{as } t\to\infty.
        \end{equation*}
        Let $c(t)$ be any positive continuous function such that $\lim_{t \to \infty} \frac{t \ell(c(t))}{c(t)^2}=1$.
        If condition \eqref{eq:h in R^-beta} holds with $\beta \in (0,1/2)$, then
        \begin{equation*}
        \frac{X(ut)-\mu^{-1}\int_{0}^{ut} h(y) \dy}{\mu^{-3/2}c(t)h(t)}
        ~\stackrel{\mathrm{f.d.}}{\Rightarrow}~ \int_{[0,\,u]}(u-y)^{-\beta} \, \mathrm{d}W_2(y)  \quad   \text{as } t\to\infty,
        \end{equation*}
        whereas if \eqref{eq:h in R^-beta} holds with $\beta=0$, the limiting process is $(W_2(u))_{u \geq 0}$.
    \item[(A3)]
        Suppose that, for some $1 < \alpha < 2$ and some $\ell$ slowly varying at $\infty$,
        \begin{equation*}
        \Prob\{\xi>t\} ~\sim~ t^{-\alpha}\ell(t)    \quad   \text{as }  t \to \infty.
        \end{equation*}
        Let $c(t)$ be any positive continuous function such that $\lim_{t \to \infty} \frac{t\ell(c(t))}{c(t)^\alpha}=1$.
        If condition \eqref{eq:h in R^-beta} holds with $\beta\in (0,1/\alpha)$, then
        \begin{equation*}
        \frac{X(ut)-\mu^{-1}\int_0^{ut}h(y)\dy}{\mu^{-1-1/\alpha}c(t)h(t)}
        ~\stackrel{\mathrm{f.d.}}{\Rightarrow}~ \int_{[0,\,u]}(u-y)^{-\beta} \,\mathrm{d}W_{\alpha}(y)
        \quad   \text{as }  t\to\infty,
        \end{equation*}
        whereas if \eqref{eq:h in R^-beta} holds with $\beta=0$, the limiting process is $(W_{\alpha}(u))_{u \geq 0}$.
\end{itemize}
\end{Theorem}

\begin{Rem} \label{Rem:slowly varying scaling}
In Theorem \ref{Thm:non-endogenous & mu<infty},
we only consider limit theorems with regularly varying scaling.
However, there are cases in which the scaling function is slowly varying.
The treatment of these requires different techniques.
\end{Rem}

We do not claim that the next result, which is needed in the proof
of Theorem \ref{Thm:non-endogenous & mu<infty}, is new.
However, with the exception of assertion (A1), which is Theorem 3.8.4(i) in
\cite{Gut:2009}, we have been unable to locate it in the
literature. In the proposition, we retain the notation of Theorem \ref{Thm:non-endogenous & mu<infty}.

\begin{Prop}    \label{Prop:moment convergence}
The following assertions hold.
\begin{itemize}
    \item[(A1)]
        If $\sigma^2:= \Var \xi<\infty$, then
        \begin{equation*}
        \lim_{t \to \infty} \frac{\E |N(t)-\mu^{-1}t|}{\sqrt{t}}
        ~=~ \frac{\sigma}{\mu^{3/2}}\E |W_2(1)| ~=~ \sigma \sqrt{\frac{2}{\pi \mu^3}}.
        \end{equation*}
    \item[(A2)]
        Suppose $\sigma^2=\infty$ and that, for some $\ell$ slowly varying at $\infty$,
        \begin{equation*}
        \E[\xi^2 \1_{\{\xi \leq t\}}]  ~\sim~  \ell(t) \quad   \text{as } t \to \infty.
        \end{equation*}
        Let $c(t)$ be a positive function satisfying $\lim_{t \to \infty} t \ell(c(t))/c(t)^2=1$.
        Then
        \begin{equation*}
        \lim_{t \to \infty} \frac{\E |N(t)-\mu^{-1}t|}{c(t)} ~=~ \frac{1}{\mu^{3/2}} \E |W_2(1)| ~=~ \sqrt{\frac{2}{\pi\mu^3}}.
        \end{equation*}
    \item[(A3)]
        Suppose $\Prob\{\xi>t\} \sim t^{-\alpha} \ell(t)$ as $t \to \infty$ for some $\alpha \in (1,2)$ and some $\ell$ slowly varying at $\infty$.
        Then
        \begin{equation*}
        \lim_{t \to \infty} \frac{\E |N(t)-\mu^{-1}t|}{c(t)}
        ~=~ \frac{\E |W_{\alpha}(1)|}{\mu^{1+1/\alpha}}
        ~=~ \frac{2\Gamma(1-\frac{1}{\alpha})|\Gamma(1-\alpha)|^{1/\alpha} \sin(\frac{\pi}{\alpha})}{\pi\mu^{1+1/\alpha}}
        \end{equation*}
        where $c(t)$ is a positive function such that $\lim_{t \to \infty} t \ell(c(t)) c(t)^{-\alpha} = 1$.
\end{itemize}
In all cases (A1)-(A3), $\E |N^*(t)-\mu^{-1} t| \sim \E |N(t)-\mu^{-1}t|$ as $t \to \infty$.
\end{Prop}

While all the previous statements of this subsection deal with the
case of finite $\mu$, our next two results are concerned with the
case of infinite $\mu$. Here the assumptions on the response
function $h$ are less restrictive.

\begin{Theorem} \label{Thm:non-endogenous & mu=infty}
Let $h:\R_+ \to \R$ be locally bounded and measurable.
Suppose that $\Prob\{\xi>t\} \sim t^{-\alpha} \ell(t)$ as $t \to \infty$
for some $0 < \alpha < 1$ and some $\ell$ slowly varying at $\infty$,
and that $h$ satisfies \eqref{eq:h in R^-beta} for some $\beta\in [0,\alpha]$.
If $\alpha = \beta$, assume additionally that
\begin{equation*}
\lim_{t \to \infty} \frac{h(t)}{\Prob\{\xi>t\}} ~=~ \lim_{t \to \infty} \frac{\ell_h(t)}{\ell(t)} ~=~   c \in (0,\infty]
\end{equation*}
and if $c=\infty$ that there exists an increasing function $u(t)$ such that
\begin{equation*}
\lim_{t \to \infty} \frac{\ell_h(t)}{\ell(t)u(t)} =1.
\end{equation*}
Let $(W_{\alpha}(u))_{u\geq 0}$ denote an inverse $\alpha$-stable subordinator defined by
\begin{equation*}
W_{\alpha}(u)    ~:=~    \inf\{t \geq 0: D_{\alpha}(t)>u\},   \quad   u \geq 0
\end{equation*}
where $(D_{\alpha}(t))_{t \geq 0}$ is an $\alpha$-stable subordinator with
$-\log \E e^{-t D_{\alpha}(1)} = \Gamma(1-\alpha) t^\alpha$ for $t \geq 0$.
Then
\begin{equation*}
\frac{\Prob\{\xi>t\}}{h(t)}X(ut) ~\stackrel{\mathrm{f.d.}}{\Rightarrow}~ \int_{[0,\,u]}(u-y)^{-\beta} \mathrm{d}W_{\alpha}(y) \quad \text{as } t \to \infty.
\end{equation*}
Furthermore, there is convergence of moments:
\begin{align}
\lim_{t \to \infty} \bigg(&\frac{\Prob\{\xi>t\}}{h(t)}\bigg)^k\E X(ut)^k
~=~
\E\left(\int_{[0,\,u]}(u-y)^{-\beta}  \mathrm{d}W_{\alpha}(y)\right)^k   \nonumber   \\
& =~
u^{k(\alpha-\beta)}\frac{k!}{\Gamma(1-\alpha)^k}\prod_{j=1}^k \frac{\Gamma(1-\beta+(j-1)(\alpha-\beta))}{\Gamma(j(\alpha-\beta)+1)},
\quad k\in\N.    \label{eq:convergence of moments}
\end{align}
\end{Theorem}

\begin{Rem} \label{Rem:Iksanov:2012}
Let the assumptions concerning $\xi$ in Theorem \ref{Thm:non-endogenous & mu<infty}
or Theorem \ref{Thm:non-endogenous & mu=infty}
be in force with eventually decreasing $h$ and with condition
\eqref{eq:h in R^-beta} replaced by
$-h(t) \sim t^\beta \ell_h(t)$ as $t \to \infty$
for some $\beta \geq 0$ and some $\ell_h$ slowly varying at
$\infty$. No further restrictions on $\beta$ like those appearing
in Theorem \ref{Thm:non-endogenous & mu<infty} are needed.
Then the limit relations of the theorems remain valid
when the limiting processes are replaced by
$\int_{[0,\,u]} (u-y)^\beta \mathrm{d} W_{\alpha}(y)$, \textit{cf.}\ Theorem 1.1 in \cite{Iksanov:2012}.
\end{Rem}

From Theorem \ref{Thm:non-endogenous & mu=infty} it follows that
if $\alpha=\beta$ and
\begin{equation}    \label{eq:h(t)/P(xi>t)->c}
\lim_{t \to \infty} \frac{h(t)}{\Prob\{\xi>t\}} = c \in  (0,\infty),
\end{equation}
then $X(t) \stackrel{\mathrm{d}}{\to} \Exp(c^{-1})$ as $t \to \infty$
where $\Exp(c^{-1})$ denotes an exponentially distributed random variable with mean $c$.
In fact, the one-dimensional convergence takes place under the sole assumption \eqref{eq:h(t)/P(xi>t)->c}.
In particular, the regular variation of neither $h(t)$, nor $\Prob\{\xi>t\}$ is needed.

\begin{Prop}    \label{Prop:exponential limit}
Assume that $\mu=\infty$ and let $h:\R_+\to\R_+$ be a measurable
and locally bounded function which satisfies condition \eqref{eq:h(t)/P(xi>t)->c}. Then
\begin{equation*}
\lim_{t \to \infty} \E X(t)^k   ~=~ c^k k!, \quad   k \in \N,
\end{equation*} which
entails $X(t) \stackrel{\mathrm{d}}{\to} \Exp(c^{-1})$ as $t \to \infty$.
\end{Prop}

\subsection{Properties of the limiting processes in Theorems \ref{Thm:non-endogenous & mu<infty} and \ref{Thm:non-endogenous & mu=infty}}   \label{subsec:properties of limiting processes}

\subsubsection*{Limits in Theorem \ref{Thm:non-endogenous & mu<infty}}  \label{subsec:Properties of limits 1<alpha<=2}

Let $1<\alpha\leq2$.
We define the limiting stochastic integral
\begin{equation*}
Y_{\alpha,\beta}(u) ~:=~ \int_{[0,u]}(u-y)^{-\beta} \, \mathrm{d}W_\alpha(y), \quad u>0
\end{equation*}
via the formula
\begin{equation}    \label{eq:def of Y_alpha,beta}
\int_{[0,u]}\!\!(u\!-\!y)^{-\beta} \, \mathrm{d}W_\alpha(y)
:= u^{-\beta}W_\alpha(u)+\beta \int_{0}^{u} \!\!(W_\alpha(u)\!-\!W_\alpha(y))(u\!-\!y)^{-\beta-1} \dy\!.
\end{equation}
This definition is consistent with the usual definition of a stochastic integral
with a deterministic integrand and the integrator being a semimartingale.
However, since $\lim_{y \uparrow u}(u-y)^{-\beta-1} = \infty$,
it is necessary to check the existence of the Lebesgue integral $\int_{0}^{u} (W_\alpha(u)\!-\!W_\alpha(y))(u\!-\!y)^{-\beta-1} \dy$.
Since
\begin{eqnarray*}
\int_{0}^{u} \E |W_\alpha(u)\!-\!W_\alpha(y)|(u\!-\!y)^{-\beta-1} \dy
& = &
\int_{0}^{u} \E |W_\alpha(u-y)| (u\!-\!y)^{-\beta-1} \dy    \\
& = &
\E |W_{\alpha}(1)| \int_{0}^{u} (u\!-\!y)^{1/\alpha-\beta-1} \dy,
\end{eqnarray*}
the integral exists in the a.s.\ sense if $\beta < 1/\alpha$.
This explains the restrictions imposed on $\beta$ in the theorem.
The processes $(Y_{\alpha,\,\beta}(u))_{u>0}$ can be called {\it
fractionally integrated $\alpha$-stable L\'{e}vy motion}.

Further, let $M_{\alpha}$ denote an $\alpha$-stable random
measure\footnote{ For the definition and existence of
$\alpha$-stable random measures see \textit{e.g.}~\cite[Section
3.3]{Samorodnitsky+Taqqu:1994}.} on $([0,\infty),\B)$ where $\B$
denotes the Borel $\sigma$-algebra over $[0,\infty)$ with constant
times Lebesgue control measure and constant skewness intensity
$-1$. We choose the constant to be $1/2$ for $\alpha=2$ and
$\Gamma(1\!-\!\alpha)\cos(\pi\alpha/2)$ for $1<\alpha<2$.
According to Example 3.3.3 in \cite{Samorodnitsky+Taqqu:1994},
$(M_{\alpha}([0,t]))_{t \geq 0}$ has the same finite-dimensional
distributions as $(W_{\alpha}(t))_{t \geq 0}$. We can thus assume
w.l.o.g.~that $W_{\alpha}(t)=M_{\alpha}([0,t])$, $t \geq 0$. Now,
\begin{align*}
u^{-\beta} & W_\alpha(u) + \beta \int_{0}^{u} \!\!(W_\alpha(u)\!-\!W_\alpha(y))(u\!-\!y)^{-\beta-1} \dy \\
&=~ u^{-\beta} W_\alpha(u) + \int_{0}^{u} \!\!(W_\alpha(u)\!-\!W_\alpha(y)) {\rm d} \big((u\!-\!y)^{-\beta}\big)    \\
&=~ u^{-\beta} W_\alpha(u) + \lim \sum_{k=1}^n (W_\alpha(u)\!-\!W_\alpha(y_k)) \big((u\!-\!y_k)^{-\beta}-(u\!-\!y_{k-1})^{-\beta}\big)  \\
&=~ \lim \bigg[W_\alpha(y_1)u^{-\beta}+
(W_{\alpha}(u)\!-\!W_{\alpha}(y_n))(u\!-\!y_n)^{-\beta} \\
&\hphantom{=~}
+~ \sum_{k=1}^{n-1} M_{\alpha}((y_{k},y_{k+1}]) (u\!-\!y_k)^{-\beta}\bigg]
\end{align*}
where we have used summation by parts in the last step. Further,
the limit can be understood as follows. Fix a sequence of nested
partitions $\Delta_n = \{y_{n,0},\ldots,y_{n,n}\}$ of $[0,u]$ such
that $0 = y_{n,0} < \ldots < y_{n,n} < u$ and such that ${\rm
mesh} (\Delta_n) = \max\{u-y_{n,n}, y_{n,k}-y_{n,k-1}:
k=1,\ldots,n\} \to 0$ as $n \to \infty$. In the displayed
formulas, for notational convenience, we write $y_k$ for
$y_{n,k}$. Notice that the first two summands in the last
displayed formula tend to $0$ in probability as $n \to \infty$,
whereas the sum can be interpreted as an integral over a step
function w.r.t.~the $\alpha$-stable random measure $M_{\alpha}$.
These integrals tend to
\begin{equation*}
\int f_{\beta}(u,y) M_{\alpha}(\dy) ~=:~    \int_{[0,u]} (u-y)^{-\beta} M_{\alpha}(\dy)
\end{equation*}
as $n \to \infty$ in probability by the construction of stable integrals \cite[Section 3.4]{Samorodnitsky+Taqqu:1994}
where
\begin{equation*}
f_{\beta}(u,y)  ~=~ \begin{cases}
                (u-y)^{-\beta}  &   \text{for } 0 \leq y < u \text{ and}    \\
                0               &   \text{for } 0 \leq u \leq y.
                \end{cases}
\end{equation*}
In particular, $Y_{\alpha,\beta}(u) = \int_{[0,u]} (u-y)^{-\beta}
M_{\alpha}(\dy)$ a.s.~for all $u>0$. Notice that
$Y_{\alpha,\beta}$ is similar to, but not identical with the
integral representation of the fractional Brownian motion
($\alpha=2$) or linear fractional stable motion ($1<\alpha<2$),
see \cite[Chapter 7]{Samorodnitsky+Taqqu:1994}.

From known properties of stable integrals \cite[Section
3.5]{Samorodnitsky+Taqqu:1994}, we infer
\begin{equation}    \label{eq:distribution of Y_alpha,beta}
Y_{\alpha,\,\beta}(u)   ~\stackrel{\mathrm{d}}{=}~
\frac{u^{1/\alpha-\beta}}{(1-\alpha\beta)^{1/\alpha}}W_\alpha(1)
\end{equation}
which means that $Y_{\alpha,\,\beta}(u)$ has a normal law in the
cases (A1) and (A2) and a spectrally negative $\alpha$-stable law
in the case (A3). The increments of $(Y_{\alpha,\beta}(u))$ are
neither independent, nor stationary.

From Theorem \ref{Thm:non-endogenous & mu<infty} and the fact that
$c$ and $h$ are regularly varying of index $1/\alpha$
(\textit{cf.}~Lemma \ref{Lem:g regularly varying}) and $-\beta$,
respectively, it follows that $Y_{\alpha,\beta}$ is self-similar
with Hurst-index $1/\alpha-\beta$, \textit{i.e.}, for every $a>0$,
\begin{equation*}
(Y_{\alpha,\beta}(au))_{u> 0}   ~\stackrel{\mathrm{f.d.}}{=}~
(a^{1/\alpha-\beta} Y_{\alpha,\beta}(u))_{u> 0}.
\end{equation*}

Finally, we provide a result on sample path properties of
$(Y_{\alpha,\beta}(u))_{u>0}$.

\begin{Prop}    \label{Prop:sample path properties}
Consider the stochastic process $(Y_{\alpha,\beta}(u))_{u> 0}$
defined by \eqref{eq:def of Y_alpha,beta} for $1 < \alpha \leq 2$
and $0 < \beta < 1/\alpha$.
\begin{itemize}
    \item[(a)]  If $\alpha=2$, then a.s., $Y_{2,\beta}$ has continuous paths.
    \item[(b)]  If $1 < \alpha < 2$, then every version $Y$ of $Y_{\alpha,\beta}$ is unbounded on every interval of positive length,
                that is, there is an event $\Omega_0$ of probability $1$ such that $\sup_{a < t < b} |Y(t)| = \infty$
                for all $0 \leq a < b$ on $\Omega_0$.
\end{itemize}
\end{Prop}
\begin{proof}
We first observe that $u\mapsto u^{-\beta}W_2(u)$ is a.s.
continuous on $(0,\infty)$. Further, from Theorem 1.14 in
\cite{Moerters+Peres:2010} (L{\'e}vy's modulus of continuity), we
conclude that for every $T>0$, there exists some measurable set
$\Omega' = \Omega'(T) \subseteq \Omega$ with $\Prob\{\Omega'\}=1$
such that, for all $\gamma \in (0,1/2)$,
\begin{equation}    \label{eq:BM modulus of continuity}
\lim_{h \downarrow 0} \frac{\sup_{u\in[0,T]}|W_2(u+h,\omega)-W_2(u,\omega)|}{h^\gamma}=0,
\quad   \omega \in \Omega'.
\end{equation}
Fix $T>0$, $\gamma\in(\beta,1/2)$ and $\omega\in \Omega'(T)$ and
set
\begin{equation*}
\phi(y) := y^{\gamma-\beta-1}
\quad   \text{and} \quad    K(u,y):=y^{-\gamma}(W_2(u,\omega)-W_2(u-y,\omega))\1_{\{0<y\leq u\}}.
\end{equation*}
Then
\begin{equation*}
\int_0^u (W_2(u,\omega)-W_2(y,\omega))(u-y)^{-\beta-1} \dy = \int_0^u K(u,y) \phi(y) \dy.
\end{equation*}
Let $0<t<u<T$ and write
\begin{align}
\Big|\int_0^{u} & K(u,y) \phi(y) \dy - \int_0^{t} K(t,y)\phi(y) \dy\Big|    \notag  \\
& \leq
\int_0^{t} |K(u,y)-K(t,y)| \phi(y) \dy
+\sup_{y \in [0,T]} |K(u,y)| \int_{t}^{u} \phi(y) \dy.  \label{eq:K}
\end{align}
Since $\omega$ is such that \eqref{eq:BM modulus of continuity} holds, one can deduce that
\begin{equation*}
\sup_{0 \leq y \leq u \leq T} K(u,y) < \infty.
\end{equation*}
This implies that each of the two summands in \eqref{eq:K} tends
to $0$ as $t\uparrow u$, where for the first summand one
additionally needs the dominated convergence theorem. Starting
with $0<u<t<T$ and repeating the argument proves that a.s.,
$Y_{2,\beta}$ is continuous on $(0,T)$. 
Since $T>0$ was arbitrary, we infer that a.s., $Y_{2,\beta}$ is
continuous on $(0,\infty)$.

We now turn to the proof of assertion (b) and consider the integral representation
\begin{equation*}
Y_{\alpha,\beta}(u) ~=~ \int f_\beta(u,y) M_{\alpha}(\dy)
\end{equation*}
where $f_{\beta}$ is as above. Define $f^*_\beta(y) := \sup_{t \in
\Q} f_\beta(t,y)$, $y \geq 0$. Clearly, $f^*_\beta(y) \geq \lim_{t
\downarrow y, t \in \Q} (t-y)^{-\beta} = \infty$ for all $y \geq
0$. Hence, condition (10.2.18) in \cite{Samorodnitsky+Taqqu:1994}
is trivially fulfilled (recall that the control measure of
$M_{\alpha}$ is a constant times Lebesgue measure), and Corollary
10.2.4 in \cite{Samorodnitsky+Taqqu:1994} yields the assertion.
\end{proof}

Finally, it is worth mentioning that with an argument as in
Proposition \eqref{Prop:sample path properties}(a) and using
(non-uniform) estimates for stable L\'{e}vy motions instead of
\eqref{eq:BM modulus of continuity}, one can show that
$Y_{\alpha,\beta}$ is a.s.~continuous at any fixed point $u>0$.

\subsubsection*{Limits in Theorem \ref{Thm:non-endogenous & mu=infty}}

In this case, $0 < \alpha < 1$ and $(W_{\alpha}(u))_{u\geq 0}$
denotes an inverse $\alpha$-stable subordinator as defined in
Theorem \ref{Thm:non-endogenous & mu=infty}. The limiting process
is
\begin{equation*}
Y_{\alpha,\,\beta}(u) := \int_{[0,\,u]} (u-y)^{-\beta} \mathrm{d}
W_\alpha(y), \quad    u>0,
\end{equation*}
where the integral can be thought of as a pathwise
Lebesgue-Stieltjes integral since the integrator $W_{\alpha}$ has
increasing paths. However, the finiteness of the integral should
be verified. This is done in the following lemma:

\begin{Lemma}   \label{Lem:check}
Let $0 < \beta \leq \alpha < 1$ and $u>0$. Then $\E
Y_{\alpha,\,\beta}(u) < \infty$. In particular,
$Y_{\alpha,\,\beta}(u) < \infty$ a.s.\ and
\begin{equation*}
\int_{(\rho u, \,u]} (u-y)^{-\beta} \mathrm{d} W_\alpha(y)  ~\to~   0   \quad   \text{a.s.\ as } \rho \uparrow 1.
\end{equation*}
\end{Lemma}
\begin{proof}
It is well known that $W_\alpha(1)$ has a Mittag-Leffler law with
$\E W_\alpha(1) = (\Gamma(1-\alpha)\Gamma(1+\alpha))^{-1} =: c_\alpha$.
Since the $\alpha$-stable subordinator is self-similar with Hurst
index $1/\alpha$, $(W_\alpha(u))_{u \geq 0}$ is self-similar with
Hurst index $\alpha$. In particular, $\E W_\alpha(u) = c_\alpha u^\alpha$.
Therefore,
\begin{eqnarray*}
\E Y_{\alpha,\,\beta}(u) & = &
\E \bigg( \int_{[0, u]} \bigg(u^{-\beta} + \beta \int_0^y (u-x)^{-\beta-1} \dx \bigg) \, \mathrm{d}W_\alpha(y)\bigg)    \\
& = &
u^{-\beta} \E W_\alpha(u) + \beta \int_{0}^{u} \E (W_\alpha(u) - W_{\alpha}(x)) (u-x)^{-\beta-1} \dx    \\
& = &
c_\alpha u^{\alpha-\beta} + \beta c_{\alpha} u^{\alpha-\beta} \int_{0}^{1} (1 - x^{\alpha}) (1-x)^{-\beta-1} \dx
~<~ \infty.
\end{eqnarray*}
\end{proof}
The processes $(Y_{\alpha,\,\beta}(u))_{u>0}$ can be called {\it
fractionally integrated inverse $\alpha$-stable subordinators}.
From Theorem \ref{Thm:non-endogenous & mu=infty} and the fact that
$1/\Prob\{\xi>t\}$ and $h(t)$ are regularly varying of index
$\alpha$
and $-\beta$, respectively, it follows that $Y_{\alpha,\beta}$ is self-similar with Hurst index $\alpha-\beta$. 
The latter implies that its increments are not stationary.

Further we infer from \cite{Iksanov:2012BSI} that the law of
$Y_{\alpha,\,\beta}(u)$ is uniquely determined by its moments
\begin{equation}    \label{eq:moments}
\E Y_{\alpha,\,\beta}(u)^k
    ~=~u^{k(\alpha-\beta)} \frac{k!}{\Gamma(1\!-\!\alpha)^k}
    \prod_{j=1}^k \frac{\Gamma(1\!-\!\beta\!+\!(j\!-\!1)(\alpha\!-\!\beta))}{\Gamma(j(\alpha\!-\!\beta)\!+\!1)},\ k \in \N.
\end{equation}
In particular,
\begin{equation}    \label{eq:Y_alpha,beta distribution}
Y_{\alpha,\, \beta}(u) ~\stackrel{\mathrm{d}}{=}~ u^{\alpha-\beta} \int_0^R e^{-cZ_\alpha(t)}\dt,
\end{equation}
where $c:=(\alpha-\beta)/\alpha$, $R$ is a random variable with
the standard exponential law which is independent of
$(Z_\alpha(u))_{u \geq 0}$, a drift-free subordinator with no
killing and L\'{e}vy measure
\begin{equation*}
\nu_\alpha(\dt) ~=~ \frac{e^{-t/\alpha}}{(1-e^{-t/\alpha})^{\alpha+1}} \1_{(0,\,\infty)}(t) \dt.
\end{equation*}

Now we want to investigate the covariance structure of
$(Y_{\alpha,\,\beta}(u))_{u>0}$. One can check that
\eqref{eq:moments} with $k=1$ remains valid whenever $\alpha \in
(0,1)$ and $\beta\in (-\infty, 1)$. Hence the process
$(Y_{\alpha,\,\beta}(u))_{u>0}$ is well-defined for such $\alpha$
and $\beta$.
\begin{Lemma}   \label{Lem:covariance of Y}
For any $\alpha\in (0,1)$, $\beta\in (-\infty,1)$ and $0 < t_1\leq t_2$,
\begin{align}
\E & Y_{\alpha,\,\beta}(t_1) Y_{\alpha,\,\beta}(t_2)    \label{eq:covariance of Y}
~=~ \frac{\Gamma(1\!-\!\beta)}{\Gamma(\alpha)\Gamma(1\!-\!\alpha)^2 \Gamma(1\!+\!\alpha\!-\!\beta)} \\
& \hphantom{\E Y_{\alpha,\,\beta}(t_1) Y_{\alpha,\,\beta}(t_2) ~=}
\times \int_{0}^{t_1} (t_1\!-\!y)^{-\beta}(t_2\!-\!y)^{-\beta}y^{\alpha-1}((t_1\!-\!y)^\alpha+(t_2\!-\!y)^\alpha) \dy.  \nonumber
\end{align}
\end{Lemma}
\begin{proof}
If $t_1=t_2$, \eqref{eq:covariance of Y} coincides with \eqref{eq:moments} in the case $k=2$ as it must be.
Now fix $t_1 < t_2$ and set
\begin{eqnarray*}
H_1(y)
& := &
\int_{[0,\,y]}(t_1-x)^{-\beta}\mathrm{d}W_\alpha(x), \quad      y \in [0,\,t_1],    \\
H_2(y)
& := &
\int_{[0,\,t_2-t_1+y]}(t_2-x)^{-\beta}\mathrm{d}W_\alpha(x), \quad  y \in [0,\,t_1].
\end{eqnarray*}
Integrating by parts we obtain
\begin{eqnarray*}
Y_{\alpha,\,\beta}(t_1)Y_{\alpha,\,\beta}(t_2)
& = &
H_1(t_1)H_2(t_1)    \\
& = &
\int_{[0,\,t_1]} H_1(x) \, \mathrm{d}H_2(x) + \int_{[0,\,t_1]} H_2(x) \, \mathrm{d}H_1(x)   \\
& = &
\int_{[0,\,t_1]}\int_{[0,\,x]}(t_1-y)^{-\beta}\mathrm{d}W_\alpha(y)(t_1-x)^{-\beta}\mathrm{d}W_\alpha(t_2-t_1+x)    \\
& & +
\int_{[0,\,t_1]}\int_{[0,\,x]}(t_2-y)^{-\beta}\mathrm{d}W_\alpha(y)(t_1-x)^{-\beta}\mathrm{d}W_\alpha(x)    \\
& & +
\int_{[0,\,t_1]}\int_{[x,\,t_2-t_1+x]}(t_2-y)^{-\beta}\mathrm{d}W_\alpha(y)(t_1-x)^{-\beta}\mathrm{d}W_\alpha(x) \\
& =: &
I_1+I_2+I_3.
\end{eqnarray*}
According to Proposition 1(a) in \cite{Bingham:1971}\footnote{Keep in mind that Bingham uses a different scaling.}, 
\begin{equation}    \label{eq:Bingham:1971}
\E \big(\mathrm{d}W_\alpha(x)\mathrm{d}W_\alpha(y)\big)
~=~ \frac{x^{\alpha-1}(y-x)^{\alpha-1}}{\Gamma^2(\alpha)\Gamma^2(1-\alpha)} \dx \dy,    \quad   0 < x < y < \infty.
\end{equation}
Below we make a repeated use of the formula (see Lemma \ref{Lem:multi integral})
\begin{equation*}
\E \int_{A} f(x)g(y) \, \mathrm{d}W_\alpha(x)\mathrm{d}W_\alpha(y)
~=~ \int_{A} f(x)g(y) \E \big(\mathrm{d}W_\alpha(x)\mathrm{d}W_\alpha(y)\big),
\end{equation*}
where $f,g$ are arbitrary non-negative measurable functions and $A \subset \R_+^2$ Borel.
Using \eqref{eq:Bingham:1971} and a change of variable, we arrive at
\begin{equation*}
\E I_3 ~=~ b_\alpha \int_{0}^{t_1} (t_1-x)^{-\beta} x^{\alpha-1} \int_{0}^{t_2-t_1} (t_2-x-y)^{-\beta} y^{\alpha-1} \dy \dx
\end{equation*}
where $b_\alpha := \Gamma(\alpha)^{-2} \Gamma(1-\alpha)^{-2}$.
\eqref{eq:Bingham:1971} and changing the order of integration followed by a change of variable ($z=t_2-t_1+x-y$) give
\begin{eqnarray*}
\E I_1
& = &
b_\alpha \int_{0}^{t_1} (t_1-x)^{-\beta}\int_{0}^{x} (t_1-y)^{-\beta}(t_2-t_1+x-y)^{\alpha-1}y^{\alpha-1} \dy \dx   \\
& = &
b_\alpha \int_{0}^{t_1} (t_1-y)^{-\beta} y^{\alpha-1} \int_{y,}^{t_1} (t_1-x)^{-\beta}(t_2-t_1+x-y)^{\alpha-1}\dx \dy   \\
& = &
b_\alpha \int_{0}^{t_1} (t_1-y)^{-\beta} y^{\alpha-1} \int_{0}^{t_2-y} (t_2-y-x)^{-\beta} x^{\alpha-1} \dx \dy  \\
& & - b_\alpha \int_{0}^{t_1} (t_1-y)^{-\beta} y^{\alpha-1} \int_{0}^{t_2-t_1} (t_2-y-x)^{-\beta} x^{\alpha-1} \dx \dy  \\
& = & B_{\alpha,\,\beta} \int_{0}^{t_1}
(t_1-y)^{-\beta}(t_2-y)^{\alpha-\beta}y^{\alpha-1} \dy - \E I_3,
\end{eqnarray*}
where $B_{\alpha,\,\beta} :=
\frac{\Gamma(1-\beta)}{\Gamma(\alpha)\Gamma^2(1-\alpha)\Gamma(1+\alpha-\beta)}$.
Analogously
\begin{eqnarray*}
\E I_2
& = &
b_\alpha \int_{0}^{t_1} (t_1-x)^{-\beta} \int_{0}^{x} (t_2-y)^{-\beta} (x-y)^{\alpha-1} y^{\alpha-1} \dy \dx    \\
& = &
b_\alpha \int_{0}^{t_1} (t_2-y)^{-\beta}y^{\alpha-1} \int_{y}^{t_1} (t_1-x)^{-\beta} (x-y)^{\alpha-1} \dx \dy   \\
& = &
b_\alpha \int_{0}^{t_1} (t_2-y)^{-\beta}y^{\alpha-1}\int_{0}^{t_1-y} (t_1-y-x)^{-\beta} x^{\alpha-1} \dx \dy        \\
& = & B_{\alpha,\,\beta} \int_{0}^{t_1}
(t_1-y)^{\alpha-\beta}(t_2-y)^{-\beta}y^{\alpha-1} \dy.
\end{eqnarray*}
It remains to sum up these expectations.
\end{proof}

Similar to the preceding lemma our next result treats both positive and negative $\beta$
thereby solving a problem which has remained open in \cite{Iksanov:2012}.

\begin{Prop}    \label{Prop:no independent increments}
For $\alpha\in (0,1)$ and $\beta\in (-\infty, 1)$, the process
$(Y_{\alpha,\,\beta}(u))_{u>0}$ does not have independent
increments.
\end{Prop}
\begin{proof}
We use the idea of the proof of Theorem 3.1 in \cite{Meerschaert+Scheffler:2004}.
Assume that the increments are independent. Then, with $0 < t_1 < t_2 < t_3 < \infty$,
\begin{align*}
\E & \big(Y_{\alpha,\,\beta}(t_2)-Y_{\alpha,\,\beta}(t_1)\big)\big(Y_{\alpha,\,\beta}(t_3)-Y_{\alpha,\,\beta}(t_2)\big) \\
& =~
\E \big(Y_{\alpha,\,\beta}(t_2)-Y_{\alpha,\,\beta}(t_1)\big)\E \big(Y_{\alpha,\,\beta}(t_3)-Y_{\alpha,\,\beta}(t_2)\big)    \\
& =~
\frac{\Gamma(1-\beta)^2}{\Gamma(1-\alpha)^2 \Gamma(1+\alpha-\beta)^2}
\big(t_2^{\alpha-\beta}-t_1^{\alpha-\beta}\big) \big(t_3^{\alpha-\beta}-t_2^{\alpha-\beta}\big)
~=:~    A(t_1,t_2,t_3).
\end{align*}
On the other hand,
\begin{align*}
\E & \big(Y_{\alpha,\,\beta}(t_2)-Y_{\alpha,\,\beta}(t_1)\big)\big(Y_{\alpha,\,\beta}(t_3)-Y_{\alpha,\,\beta}(t_2)\big) \\
&=~
\E Y_{\alpha,\,\beta}(t_2)Y_{\alpha,\,\beta}(t_3) - \E Y_{\alpha,\,\beta}(t_2)^2 - \E Y_{\alpha,\,\beta}(t_1)Y_{\alpha,\,\beta}(t_3)
+ \E Y_{\alpha,\,\beta}(t_1)Y_{\alpha,\,\beta}(t_2) \\
&=~
\frac{\Gamma(1-\beta)}{\Gamma(\alpha)\Gamma^2(1-\alpha)\Gamma(1+\alpha-\beta)}  \\
&\hphantom{=}~ \times \bigg(\int_{0}^{t_2} (t_2-y)^{-\beta}(t_3-y)^{-\beta}y^{\alpha-1}\big((t_2-y)^\alpha+(t_3-y)^\alpha\big) \dy  \\
&\hphantom{=\times \bigg(}~ - \int_{0}^{t_1} (t_1-y)^{-\beta}(t_3-y)^{-\beta}y^{\alpha-1}\big((t_1-y)^\alpha+(t_3-y)^\alpha\big) \dy    \\
&\hphantom{=\times \bigg(}~ + \int_{0}^{t_1} (t_1-y)^{-\beta}(t_2-y)^{-\beta}y^{\alpha-1}\big((t_1-y)^\alpha+(t_2-y)^\alpha\big) \dy \bigg) \\
&\hphantom{=}~ - \frac{2\Gamma(1-\beta)\Gamma(1+\alpha-2\beta)}{\Gamma^2(1-\alpha)\Gamma(1+\alpha-\beta)\Gamma(1+2\alpha-2\beta)} t_2^{2\alpha-2\beta}
~=:~ B(t_1,t_2,t_3).
\end{align*}
By the assumption $A(t_1,t_2,t_3)=B(t_1,t_2,t_3)$ for all $t_1<t_2<t_3$.
On the other hand,
\begin{equation*}
\frac{\partial^2 A(t_1,t_2,t_3)}{\partial t_1 \partial t_3}
~=~ -\frac{\Gamma^2(1-\beta)}{\Gamma^2(1-\alpha)\Gamma^2(1+\alpha-\beta)}(\alpha-\beta)^2 (t_1t_3)^{\alpha-\beta-1},
\end{equation*}
and
\begin{align*}
& \frac{\partial^2 B(t_1,t_2,t_3)}{\partial t_1 \partial t_3}   \\
& ~=~   -\frac{\Gamma(1-\beta)}{\Gamma(\alpha)\Gamma^2(1-\alpha)\Gamma(1+\alpha-\beta)} \\
& ~\hphantom{=}~
\frac{\partial^2}{\partial t_1 \partial t_3}\bigg(\int_{0}^{t_1}(t_1\!-\!y)^{-\beta}(t_3\!-\!y)^{-\beta}y^{\alpha-1}\big((t_1\!-\!y)^\alpha+(t_3\!-\!y)^\alpha\big) \dy \bigg)  \\
& ~=~
-\frac{\Gamma(1-\beta)}{\Gamma(\alpha)\Gamma^2(1-\alpha)\Gamma(1+\alpha-\beta)} \\
&~\hphantom{=}~ \times \frac{\partial^2}{\partial t_1 \partial t_3}
\bigg(t_1^{\alpha-\beta} \!\! \int_0^1 \! (1\!-\!y)^{-\beta}(t_3\!-\!t_1y)^{-\beta}y^{\alpha-1}\big(t_1^{\alpha}(1-y)^\alpha\!+\!(t_3\!-\!t_1y)^\alpha\big) \dy \bigg)  \\
&~=~
-\frac{\Gamma(1-\beta)}{\Gamma(\alpha)\Gamma^2(1-\alpha)\Gamma(1+\alpha-\beta)} \\
&~\hphantom{=}~
\times \bigg[\frac{\partial}{\partial t_1}\Big(-\beta t_1^{2\alpha-\beta}\int_{0}^{1}(1-y)^{\alpha-\beta}y^{\alpha-1}(t_3-t_1y)^{-\beta-1} \dy \Big)    \\
&~\hphantom{= \times \bigg[}~ + \frac{\partial}{\partial t_1}\Big((\alpha-\beta) t_1^{\alpha-\beta}\int_{0}^{1}(1-y)^{-\beta}y^{\alpha-1}(t_3-t_1 y)^{\alpha-\beta-1} \dy \Big)\bigg].
\end{align*}
Therefore,
\begin{align*}
& \frac{\partial^2 B(t_1,t_2,t_3)}{\partial t_1 \partial t_3}   \\
&~=~
-\frac{\Gamma(1-\beta)}{\Gamma(\alpha)\Gamma^2(1-\alpha)\Gamma(1+\alpha-\beta)} \\
&~\hphantom{=}~ \times \bigg[-\beta(2\alpha-\beta) t_1^{2\alpha-\beta-1}\int_{0}^{1}(1-y)^{\alpha-\beta}y^{\alpha-1}(t_3-t_1y)^{-\beta-1}\dy    \\
&~\hphantom{=\times \bigg[}~ - \beta(\beta+1) t_1^{2\alpha-\beta}\int_{0}^{1}(1-y)^{\alpha-\beta}y^{\alpha}(t_3-t_1y)^{-\beta-2}\dy\\
&~\hphantom{=\times \bigg[}~ + (\alpha-\beta)^2t_1^{\alpha-\beta-1}\int_{0}^{1}(1-y)^{-\beta}y^{\alpha-1}(t_3-t_1y)^{\alpha-\beta-1}\dy\\
&~\hphantom{=\times \bigg[}~ -(\alpha-\beta)(\alpha-\beta-1)t_1^{\alpha-\beta}\int_{0}^{1}(1-y)^{-\beta}y^{\alpha}(t_3-t_1y)^{\alpha-\beta-2}\dy\bigg].
\end{align*}
To show that these expressions are not equal, assume that $0<t_1<1$ and set $t_3=t_1^{-1}$, $z:=t_1^2$.
Then the first one does not depend on $z$. The second, after some manipulations, becomes
\begin{eqnarray*}
D(z)
& := &
-\frac{\Gamma(1-\beta)}{\Gamma(\alpha)\Gamma^2(1-\alpha)\Gamma(1+\alpha-\beta)}     \\
& &
\times \bigg[-\beta(2\alpha-\beta) z^{\alpha}\int_{0}^{1}(1-y)^{\alpha-\beta}y^{\alpha-1}(1-zy)^{-\beta-1}\dy\\
& &
\hphantom{\times \bigg[}
-\beta(\beta+1) z^{\alpha+1}\int_{0}^{1}(1-y)^{\alpha-\beta}y^{\alpha}(1-zy)^{-\beta-2}\dy\\
& &
\hphantom{\times \bigg[}
+ (\alpha-\beta)^2\int_{0}^{1}(1-y)^{-\beta}y^{\alpha-1}(1-zy)^{\alpha-\beta-1}\dy\\
& &
\hphantom{\times \bigg[}
-(\alpha-\beta)(\alpha-\beta-1)z \int_{0}^{1}(1-y)^{-\beta}y^{\alpha}(1-zy)^{\alpha-\beta-2} \dy \bigg].\\
\end{eqnarray*}
Using the asymptotic expansion $(1-z)^{\alpha}=1-\alpha z +O(z^2)$ as $z\to 0$ ($\alpha\in\R$) yields
\begin{eqnarray*}
D(z)
& = &
-\frac{\Gamma(1-\beta)}{\Gamma(\alpha)\Gamma^2(1-\alpha)\Gamma(1+\alpha-\beta)} \times      \\
& &
\times \Big[(\alpha\!-\!\beta)^2 \int_{0}^{1}(1\!-\!y)^{-\beta}y^{\alpha-1}\dy - \beta(2\alpha\!-\!\beta) z^{\alpha}\int_{0}^{1}(1\!-\!y)^{\alpha-\beta}y^{\alpha-1}\dy \\
& &
\hphantom{\times \Big[} +O(z) \Big],
\end{eqnarray*}
as $z\to 0$. From this expansion it is clear that $D(z)$ depends on $z$ if $\beta(2\alpha-\beta)\neq 0$ since $\alpha<1$. If $\beta=0$ then $Y_{\alpha,\beta}(u)=W_{\alpha}(u)$ and this process does not have independent increments as was shown in Theorem 3.1 in \cite{Meerschaert+Scheffler:2004}. If $2\alpha=\beta$ using the same idea one can show
\begin{equation*}
D(z)=c_1+c_2z+O(z^{\alpha+1})
\end{equation*}
where $c_1 c_2 \neq 0$, we omit the details.
\end{proof}

Formula \eqref{eq:Y_alpha,beta distribution} entails that
$Y_{\alpha,\alpha}(u) \stackrel{\mathrm{d}}{=} R$,
\textit{i.e.}, all one-dimensional distributions of
$(Y_{\alpha,\alpha}(u))_{u>0}$ are standard exponential. This
leads to the conjecture that the process
$(Y_{\alpha,\alpha}(u))_{u>0}$ may bear some kind of
stationarity.

\begin{Lemma}
The process $(Y_{\alpha,\,\alpha}(e^u))_{u \in \R}$ is strictly stationary with covariance function
$R(s) := \E (Y_{\alpha,\,\alpha}(e^u)-1)(Y_{\alpha,\,\alpha}(e^{u+s})-1)$, $s \in \R$ given by
\begin{equation}    \label{eq:R(s)}
R(s) ~=~ \frac{1}{\Gamma(\alpha)\Gamma(1-\alpha)}\int_{|s|}^{\infty} (1-e^{-y})^{-\alpha}e^{-\alpha y} \dy, \quad   s \in \R.
\end{equation}
\end{Lemma}
\begin{proof}
The strict stationarity follows from the case $\alpha=\beta$ and $c<\infty$ of Theorem \ref{Thm:non-endogenous & mu=infty}.
Indeed, by that theorem, as $t \to \infty$,
$(X(u_1t),\ldots, X(u_n t)) \stackrel{\mathrm{d}}{\to}
c (Y_{\alpha,\,\alpha}(u_1),\ldots, Y_{\alpha,\,\alpha}(u_n))$
for any $n \in \N$ and any $0<u_1<\ldots<u_n$, and, for any $h>0$, the
weak limit of $(X(u_1ht),\ldots, X(u_n h t))$ is the same.
To prove \eqref{eq:R(s)}, it suffices to show that, for $0<t_1<t_2<\infty$,
\begin{equation*}
\E Y_{\alpha,\,\alpha}(t_1)Y_{\alpha,\,\alpha}(t_2)
~=~ 1+\frac{1}{\Gamma(\alpha)\Gamma(1-\alpha)}\int_{0}^{t_1/t_2}(1-y)^{-\alpha}y^{\alpha-1} \dy.
\end{equation*}
The last equality follows from \eqref{eq:covariance of Y} with $\alpha=\beta$.
\end{proof}

\begin{Prop}    \label{Prop:sample path propertiesII}
For $0 < \alpha < 1$ and $0 < \beta <\alpha$, the process
$(Y_{\alpha,\beta}(u))_{u>0}$ has a.s.~continuous sample paths.
\end{Prop}
\begin{proof}
Formula (6) in \cite{Hawkes:1971} says that, for some particular choice of $a>0$ and any $\gamma<\alpha$, a.s.,
\begin{equation}    \label{eq:Hoelder continuity inverse subordinator}
\lim_{h \downarrow 0} \frac{\sup_{0 \leq u \leq a} |W_{\alpha}(u+h)-W_{\alpha}(u)|}{h^\gamma} ~=~ 0.
\end{equation}
However, the perusal of its proof reveals that \eqref{eq:Hoelder continuity inverse subordinator} holds for arbitrary fixed $a>0$.  
It is implicit in the proof of Lemma \ref{Lem:check} that
representation \eqref{eq:def of Y_alpha,beta} holds also for the
present $Y_{\alpha,\beta}$. With this at hand, the rest of the
proof literally repeats the proof of Proposition \ref{Prop:sample
path properties}(a) and is thus omitted.
\end{proof}

\section{Preliminaries} \label{sec:preliminaries}

\subsection{Stationary renewal processes and coupling}  \label{subsec:stationary renewal processes and coupling}

Our first result in this section shows that the finite-dimensional distributions of the increments of the stationary renewal counting process
are invariant under time reversal.

\begin{Prop}    \label{Prop:N^ backwards in time}
Let $\mu < \infty$ and $F$ be non-lattice. Then, for every $t > 0$,
\begin{equation*}   
(N^*(t)-N^*((t-s)-): 0 \leq s \leq t)
~\stackrel{\mathrm{d}}{=}~  (N^*(s): 0 \leq s \leq t).
\end{equation*}
\end{Prop}
\begin{proof}
For the proof of this proposition, it is convenient to embed the zero-delayed random walk $(S_k)_{k \in \N_0}$
into a two-sided random walk $(S_k)_{k \in \Z}$.
To this end, assume that on the basic probability space, there is an independent copy $(\xi_{-k})_{k \in \N}$ of the
sequence $(\xi_k)_{k \in \N}$.
Let $S_{-k} := -(\xi_{-1}+\ldots+\xi_{-k})$ for $k \in \N$.
Further, assume there is a random variable $\xi_0$ which is independent of the $\xi_k$, $k \in \Z \setminus \{0\}$ with size-biased distribution
\begin{equation*}
\Prob\{\xi_0 \in B\} ~=~ \mu^{-1} \E [\xi \1_{\{\xi \in B\}}],    \quad   B \subseteq \R_+ \text{ Borel}.
\end{equation*}
Finally, let $U$ have the uniform distribution on $(0,1)$ and assume that $U$ is independent of the sequence $(\xi_k)_{k \in \Z}$.
Define $S^*_0 := U \xi_0$ and $S^*_{-1} := -(1-U) \xi_0$.
Then $S^*_0$ and $-S^*_{-1}$ have distribution function $F^*$ (see \textit{e.g.}\ \cite[p.\;261]{Thorisson:2000} for a quick proof).
Now recall that $S^*_k = S^*_0+S_k$ for $k \in \N$ and define analogously $S^*_{k} := S^*_{-1} + S_{k+1}$ for $k < -1$.
By the first Palm duality (point-at-zero duality, Theorem 4.1 in Chapter 8 of \cite{Thorisson:2000}),
the process $N_{\Z}^* := \sum_{k \in \Z} \delta_{S^*_k}$ is distributionally invariant under shifts.
Using this and the fact that $(-S^*_{-k})_{k \in \N}$ has the same distribution as $(S^*_{k-1})_{k \in \N}$,
we infer for given $t>0$
\begin{equation*}
(N_{\Z}^*[0,s]\!: 0 \! \leq \! s \! \leq \! t)
\stackrel{\mathrm{d}}{=}
(N_{\Z}^*[-t,-(t-s)]\!: 0 \! \leq \! s \! \leq \! t)
\stackrel{\mathrm{d}}{=}
(N_{\Z}^*[t-s,t]\!: 0 \! \leq \! s \! \leq \! t).
\end{equation*}
This implies the assertion in view of the fact that $N^*(\cdot) = N_{\Z}^*(\cdot \cap [0,\infty))$.
\end{proof}

Next, we briefly introduce a classical coupling that will be useful in several proofs in this paper and
which works in the case when $\mu < \infty$ and $F$ is non-lattice.
Let $(\hat{\xi}_k)_{k \in \N}$ be an independent copy of the sequence $(\xi_k)_{k \in \N}$.
Let $\hat{S}_0^*$ denote a random variable that is independent of all previously introduced random variables
and has distribution function $F^*$ (recall the definition of $F^*$ from \eqref{eq:S_0^*}).
Put
\begin{equation*}
\hat{S}_0 := 0  \quad   \text{and}  \quad   \hat{S}_k := \hat{\xi}_1+ \ldots + \hat{\xi}_k,     \quad k \in \N.
\end{equation*}
Let $\hat{N}$ be the renewal counting process associated with the process $(\hat{S}_k)_{k \in \N_0}$.
In particular, $\hat{N}(t) := \hat{N}[0,t] = \#\{k\in\N_0: \hat{S}_k \leq t\}$, $t \geq 0$.
Further, define $\hat{S}_k^* := \hat{S}_0^* + \hat{S}_k$, $k \in \N_0$
and let $\hat{N}^* := \sum_{k \geq 0} \delta_{\hat{S}_k^*}$ denote the associated renewal counting process.
As usual, put $\hat{N}^*(t) := \hat{N}^*([0,t])$, $t \geq 0$.
By construction, $(\hat{N}^*(t))_{t\geq 0}$ is a stationary renewal process.

It is known (\textit{cf.}\ p.~210 in \cite{Durrett:2010}) that,
for any fixed $\varepsilon>0$, there exist almost surely finite
$\tau_1$ and $\tau_2$ such that
\begin{equation*}
|S_{\tau_2}-\hat{S}^{\ast}_{\tau_1}|\leq \varepsilon
\end{equation*}
almost surely. Define the coupled random walk
\begin{equation*}
\tilde{S}_k^*   ~:=~    \begin{cases}
                \hat{S}_k^*,                                                                    &   \text{for } k \leq \tau_1,   \\
                \hat{S}^{\ast}_{\tau_1} + \sum_{j=\tau_2+1}^{\tau_2+k-\tau_1} \xi_j,     &   \text{for } k \geq \tau_1+1,
\end{cases}
\end{equation*}
$k \in \N_0$. Then $(\tilde{S}_k^*)_{k \in \N_0}
\stackrel{\mathrm{d}}{=} (\hat{S}_k^*)_{k \in \N_0}
\stackrel{\mathrm{d}}{=} (S_k^*)_{k \in \N_0}$. In particular, the
random process $(\tilde{N}^*(t))_{t\geq 0}$ defined by
\begin{equation*}   
\tilde{N}^*(t) := \#\{k \in \N_0: \tilde{S}_k^* \leq t\},
\end{equation*}
is a stationary renewal process. Further, the construction of the
process $(\tilde{S}_k^*)_{k \in \N_0}$ guarantees that
\begin{equation}    \label{eq:epsilon close}
\tilde{S}_{\tau_1+k}^* - \varepsilon \leq S_{\tau_2+k} \leq
\tilde{S}_{\tau_1+k}^*+\varepsilon
\end{equation}
for $k \in \N_0$. Further, one can check that, for any fixed
$\varepsilon>0$ and arbitrary $0 \leq y \leq t$, on $\{\tau_1\vee
\tau_2 < \infty\}$ (hence with probability one),
\begin{eqnarray}
\sum_{k \geq \tau_2} \1_{\{t-y<S_k \leq t\}} & \leq &
\sum_{k \geq \tau_1} \1_{\{t-y-\varepsilon<\tilde{S}_k^* \leq t+\varepsilon\}}    \nonumber   \\
& \leq &
\tilde{N}^*(t\!+\!\varepsilon)-\tilde{N}^*(t\!-\!y\!-\!\varepsilon)
\label{eq:N leq N*+O(1)}
\end{eqnarray}
where $\tilde{N}^*(x) := 0$ for $x < 0$ is stipulated. Similarly,
for any fixed $\varepsilon>0$ and $0 \leq y \leq t$, again on
$\{\tau_1\vee\tau_2 < \infty\}$ (and thus with probability one),
\begin{eqnarray}
\sum_{k\geq \tau_2} \1_{\{t-y<S_k\leq t\}} & \geq &
\sum_{k \geq \tau_1} \1_{\{t-y+\varepsilon<\tilde{S}_k^* \leq t-\varepsilon\}}    \notag  \\
& = &
\tilde{N}^*(t\!-\!\varepsilon)-\tilde{N}^*(t\!-\!y\!+\!\varepsilon)
- \sum_{k=0}^{\tau_1-1} \1_{\{t-y+\varepsilon<\tilde{S}_k^*\leq
t-\varepsilon\}}. \qquad  \label{eq:N geq N*+O(1)}
\end{eqnarray}

\subsection{Stable distributions and domains of attraction} \label{subsec:domains of attraction}

Naturally, the asymptotics of the shot noise process $(X(ut))_{u
\geq 0}$ as $t \to \infty$ is connected to the limiting behavior
of $S_k$ as $k \to \infty$. Under the assumptions of the limit
theorems with scaling, $F$ (the distribution of $\xi$) is in the
domain of attraction of a stable law. To be more precise, for
appropriate constants $c_k > 0$ and $b_k \in \R$,
\begin{equation}    \label{eq:domain of attraction}
\frac{S_k-b_k}{c_k} ~\stackrel{\mathrm{d}}{\to}~    W   \quad
\text{as } k \to \infty
\end{equation}
where $W$ has a stable law $S_{\alpha}(\sigma',\beta',\mu')$ that
is characterized by four parameters, the \emph{index of stability}
$\alpha$ and the \emph{scale}, \emph{skewness}, and \emph{shift}
parameters, $\sigma'$, $\beta'$, and $\mu'$, respectively. We
refer to \cite{Samorodnitsky+Taqqu:1994} for the precise
definition of the law $S_{\alpha}(\sigma',\beta',\mu')$ as well as
for a general introduction to stable random variables. Due to the
assumption that $\xi > 0$ a.s., $S_{\alpha}(\sigma',\beta',\mu')$
will automatically be totally skewed to the right, \textit{i.e.},
it will have skewness parameter $\beta'=1$.

By changing $b_k$ and $c_k$ if necessary, it can be arranged that
the shift parameter $\mu'$ equals $0$ and the scale parameter
$\sigma'$ is as we wish. Particularly, the limit can be arranged
to be standard normal in the case $\alpha=2$ and to have
characteristic function
\begin{equation}    \label{eq:stable ch f spectrally positive}
z \mapsto \exp\big\{-|z|^\alpha
\Gamma(1\!-\!\alpha)(\cos(\pi\alpha/2)-\i\sin(\pi\alpha/2)\,
{\sign}(z))\big\}, \ z \in \R
\end{equation}
in case $0 < \alpha < 2$, $\alpha \not = 1$. The constants $b_k$
and $c_k$ that produce this limit in \eqref{eq:domain of
attraction} can be described in terms of the distribution of
$\xi$.
\begin{itemize}
    \item[(D1)]
        The case $\sigma^2 < \infty$:   \\
        If $\sigma^2    :=  \Var \xi < \infty$,
        then \eqref{eq:domain of attraction} holds with $b_k = k\mu$, $c_k = \sigma \sqrt{k}$.
        $W$ then has the standard normal law.
    \item[(D2)]
        The case when $\sigma^2 = \infty$, yet there is attraction to a normal law: \\
        If $\int_{[0,t]} y^2 \Prob\{\xi\in\dy\} \sim \ell(t)$ for some function $\ell$ that is slowly varying at $+\infty$,
        then \eqref{eq:domain of attraction} holds with $b_k = k\mu$ and the $c_k$, $k\in \N$ being such that
        $\lim_{k \to \infty} \frac{k \ell(c_k)}{c_k^2}=1$.
        Again, $W$ has the standard normal law.
    \item[(D3)]
        The case $1 < \alpha < 2$:  \\
        If $\Prob\{\xi>t\} \sim t^{-\alpha}\ell(t)$ for some $1 < \alpha < 2$ and some $\ell$ slowly varying at $\infty$,
        then \eqref{eq:domain of attraction} holds with $b_k = k\mu$ and the $c_k$, $k\in \N$ being such that
        $k \Prob\{\xi > c_k\} \sim k \ell(c_k)/c_k^\alpha \to 1$.
        $W$ then has a spectrally positive stable law with characteristic function given by \eqref{eq:stable ch f spectrally positive}.
    \item[(D4)]
        The case $0 < \alpha < 1$:  \\
        If $\Prob\{\xi>t\} \sim t^{-\alpha} \ell(t)$ as $t \to \infty$ for some $0 < \alpha < 1$ and some $\ell$ slowly varying at $\infty$,
        \eqref{eq:domain of attraction} holds with $b_k=0$ and $c_k$, $k\in \N$ being such that $k \Prob\{\xi > c_k\} \sim k \ell(c_k)/c_k^\alpha \to 1$.
        Again, $W$ has a spectrally positive stable law with characteristic function given by \eqref{eq:stable ch f spectrally positive}.
        Further, $W > 0$ a.s.\ in this case and its Laplace exponent is given by
        \begin{equation*}   
        -\log \varphi(s)    ~=~ \Gamma(1-\alpha) s^{\alpha},    \quad   s \geq 0.
        \end{equation*}
\end{itemize}
Clearly, (D1) is the classical central limit theorem. (D2) follows
from \cite[Theorem IX.8.1 and Eq.\ (IX.8.12)]{Feller:1971}. (D3)
and (D4) follow from the lemma on p.\;107 of \cite{Feller:1949}.
(D4) is also Theorem XIII.7.2 in \cite{Feller:1971} (where Laplace
transforms are used rather than characteristic functions).

\subsection{Convergence in distribution of the renewal counting process}   \label{subsec:convergence of N(t)}

Denote by $D := D[0,\infty)$ the space of right-continuous
real-valued functions on $[0,\infty)$ with finite limits from the
left. It is well known (see, for instance, Theorem 5.3.1 and Theorem
5.3.2 in \cite{Gut:2009} or Section 7.3.1 in \cite{Whitt:2002})
that the following functional limit theorems hold:
\begin{equation}    \label{eq:FLT for N(t)}
W_t(u)  ~:=~    \frac{N(ut)-\mu^{-1}ut}{g(t)}   ~\Rightarrow~   W_{\alpha}(u)    \quad   \text{as } t \to \infty
\end{equation}
where $W_{\alpha}$ is as in Theorem \ref{Thm:non-endogenous & mu<infty}
and where in the case
\begin{itemize}
    \item[(D1)]
        $\alpha=2$, $g(t)=\sqrt{\sigma^2\mu^{-3}t}$ and the convergence takes place in the $J_1$ topology on $D$;
    \vspace{-0,25cm}
    \item[(D2)]
        $\alpha=2$, $g(t)=\mu^{-3/2}c(t)$ with $c(t)$ being any positive continuous function satisfying $\lim_{t \to \infty} t \ell(c(t))c(t)^{-2} = 1$
        and the convergence takes place in the $J_1$ topology on $D$;
    \vspace{-0,25cm}
    \item[(D3)]
        $1<\alpha<2$, $g(t)=\mu^{-1-1/\alpha}c(t)$ where $c(t)$ is any positive continuous function with $\lim_{t \to \infty} t\ell(c(t))c(t)^{-\alpha}=1$
        and the convergence takes place in the $M_1$ topology on $D$.
\end{itemize}
We refer to \cite{Whitt:2002} for extensive information concerning
both the $J_1$ and $M_1$ convergence in $D$.

For later use we note the following lemma.
\begin{Lemma}   \label{Lem:g regularly varying}
$c(t)$ and $g(t)$ are regularly varying with index $1/2$ in the cases (D1) and (D2), and with index $1/\alpha$ in the case (D3).
As a consequence, given $A > 1$ and $\delta\in (0,1/\alpha)$
(here, $\alpha=2$ in the cases (D1) and (D2))
there exists $t_0>0$ such that
\begin{equation}    \label{eq:Potter's bound for g}
\frac{g(tv)}{g(t)} ~\leq~ A v^{1/\alpha-\delta},
\end{equation}
for all $0 < v \leq 1$ and $t > 0$ such that $t v \geq t_0$.
\end{Lemma}
\begin{proof}
We only check this for the case (D2), the case (D3) being similar,
and the case (D1) being trivial.

The function $c(t)$ is an asymptotic inverse of
\begin{equation*}
t^2 \big(\E[\xi^2 \1_{\{\xi \leq t\}}]\big)^{-1} ~\sim~ t^2/\ell(t).
\end{equation*}
Hence, by Proposition 1.5.15 in \cite{Bingham+Goldie+Teugels:1989}, $c(t) \sim t^{1/2}(L^\#(t))^{1/2}$
where $L^\#(t)$ is the de Bruijn conjugate of $L(t)=1/\ell(t^{1/2})$.
The de Bruijn conjugate is slowly varying and hence $c$ regularly varying of index $1/2$.
\eqref{eq:Potter's bound for g} is Potter's bound for $g$
(see Theorem 1.5.6 in \cite{Bingham+Goldie+Teugels:1989}).
\end{proof}

There is also an analogue of \eqref{eq:FLT for N(t)} in the case (D4).
The functional convergence
\begin{equation}    \label{eq:FLT for N(t) when mu=infty}
W_t(u):=\frac{N(ut)}{g(t)}      ~\Rightarrow~   W_{\alpha}(u)    \quad \text{as } t \to \infty
\end{equation}
under the $J_1$ topology in $D$ where $(W_{\alpha}(u))_{u \geq 0}$ is an inverse $\alpha$-stable subordinator and $g(t):=1/\Prob\{\xi>t\}$
was proved in Corollary 3.4 in \cite{Meerschaert+Scheffler:2004}.

\section{Proofs of the limit theorems without scaling}  \label{sec:endogenous case}

\subsection{Preparatory results}

We first show that under the assumptions of Theorems \ref{Thm:dRi h} and \ref{Thm:decreasing h},
the limiting variables $X^*$ and $X^*_{\circ}$, respectively, are well-defined.
It turns out that in the case of $X^*_{\circ}$, this is more than halfway to proving one-dimensional convergence in Theorem \ref{Thm:decreasing h}.

\begin{Prop}    \label{Prop:X* well-defined}
Assume that $\mu < \infty$ and that $F$ is non-lattice.
If $h$ is Lebesgue integrable, then $X^*$ exists as the a.s.\ limit and the limit in $\mathcal{L}^1$ in \eqref{eq:X*}.
In particular, $X^*$ is integrable, \textit{a fortiori} a.s.\ finite.
\end{Prop}
\begin{proof}
Since $\mu < \infty$, the stationary walk $(S^*_k)_{k \in \N_0}$ exists.
Recall that $U^*(\dx) = \sum_{k \geq 0} \Prob\{S^*_k \in \dx\} = \mu^{-1} \dx$.
Hence, by the Lebesgue integrability of $h$,
\begin{equation*}
\E \sum_{k \geq 0} |h(S_k^*)|   ~=~ \int_0^{\infty} |h(x)| \dx  ~<~ \infty,
\end{equation*}
from which the assertion easily follows.
\end{proof}

Before we proceed with sufficient conditions for the well-definedness of $X^*_{\circ}$ we prove two auxiliary lemmas that will be needed later.

\begin{Lemma}   \label{Lem:approximating X*}
Assume that $\mu < \infty$ and that $F$ is non-lattice.
Further let $h, h_1, h_2, \ldots$ be Lebesgue integrable and $h_n \downarrow h$ a.e.\ or $h_n \uparrow h$ a.e.
Define $X^*$ as usual and $X^*_n := \sum_{k \geq 0} h_n(S^*_k)$, $n \in \N$.
Then $X_n^* \to X^*$ a.s.\ and in $\mathcal{L}^1$.
\end{Lemma}
\begin{proof}
Assume without loss of generality that $h_n \downarrow h$ a.e.\ as $n \to \infty$.
Then
\begin{equation*}
\lim_{n \to \infty} \E |X_n^* - X^*|
~=~ \lim_{n \to \infty} \frac{1}{\mu} \int_{0}^{\infty} (h_n(x)-h(x)) \, \dx
~=~ 0
\end{equation*}
by the monotone convergence theorem.
The asserted a.s.\ convergence follows from the convergence in $\mathcal{L}^1$ together with the monotonicity of $X_n^*$ in $n$.
Indeed, $\mathcal{L}^1$-convergence implies the existence of a subsequence $(n_k)_{k \geq 1}$ along which a.s.\ convergence holds.
The monotonicity of $X_n^*$ in $n$ then implies that a.s.\ convergence must hold as $n \to \infty$.
\end{proof}

\begin{Lemma}   \label{Lem:h^eps is dRi}
If $h: \R \to \R$ is d.R.i., so is $h^{\varepsilon}(x) :=
\sup_{|y-x| \leq \varepsilon} h(y)$, $x \in \R$, for each fixed
$\varepsilon>0$.
\end{Lemma}
\begin{proof}
For $\delta > 0$, let $I_n^{\delta}:=[n\delta,(n+1)\delta)$, $n \in \Z$.
Further, define
\begin{equation*}
\overline{\sigma}(\delta) := \delta \sum_{n \in \Z} \sup_{x \in I_n^{\delta}} h(x)
\quad \text{and} \quad \underline{\sigma}(\delta) :=\delta \sum_{n \in \Z} \inf_{x \in I_n^{\delta}} h(x).
\end{equation*}
$h$ being d.R.i.\ means that $\overline{\sigma}(\delta)$ and $\underline{\sigma}(\delta)$
converge absolutely for every $\delta > 0$ and that $\lim_{\delta \downarrow 0}(\overline{\sigma}(\delta)-\underline{\sigma}(\delta))=0$.
Now define $\overline{\sigma}^{\varepsilon}(\delta)$ and $\underline{\sigma}^{\varepsilon}(\delta)$ analogously with $h$
replaced by $h^{\varepsilon}$.
There exists $m \in \N$ such that  $m \delta > \varepsilon$, hence
$[n\delta-\varepsilon,(n+1)\delta+\varepsilon)\subset
\cup_{k=n-m}^{n+m}I_k^{\delta}$. Therefore, for $n \in \Z$,
\begin{equation*}
\bigg|\sup_{x \in I_n^{\delta}} h^{\varepsilon}(x)\bigg|
~\leq~ \sum_{k=n-m}^{n+m}\bigg|\sup_{x \in I_{k}^{\delta}} h(x)\bigg|.
\end{equation*}
This implies the absolute convergence of
$\overline{\sigma}^{\varepsilon}(\delta)$. The corresponding
assertion for  $\underline{\sigma}^{\varepsilon}(\delta)$ follows
similarly. Further, if $x$ is a discontinuity of
$h^{\varepsilon}$, then $x-\varepsilon$ or $x+\varepsilon$ is a
discontinuity of $h$. Consequently, since $h$ is continuous a.e., so is $h^{\varepsilon}$.
Combining this with the absolute convergence of $\overline{\sigma}^{\varepsilon}(\delta)$
and $\underline{\sigma}^{\varepsilon}(\delta)$, we conclude that
$h^{\varepsilon}$ is Lebesgue integrable.
The monotone convergence theorem yields
\begin{equation*}
\lim_{\delta \downarrow 0}
\underline{\sigma}^{\varepsilon}(\delta) ~=~ \int_{\R}
h^{\varepsilon}(x) \dx ~=~ \lim_{\delta \downarrow 0}
\overline{\sigma}^{\varepsilon}(\delta).
\end{equation*}
\end{proof}

\begin{Prop}    \label{Prop:X*_o well-defined}
Assume that $\mu < \infty$ and that $F$ is non-lattice.
Let $h:\R_+ \to \R$ be locally bounded, eventually decreasing and non-integrable and recall that
\begin{equation*}
X^*_{\circ} ~:=~ \lim_{t \to \infty} \,\bigg( \sum_{k \geq 0} h(S_k^*) \1_{\{S_k^* \leq t\}} - \frac{1}{\mu} \int_0^t h(y) \dy\bigg).
\end{equation*}
Under the assumptions of Theorem \ref{Thm:decreasing h},
$X^*_{\circ}$ exists as the limit in $\mathcal{L}^2$ in the case (C1),
as the a.s.\ limit in the case (C2) and as the limit in probability in the case (C3).
In all three cases, it is a.s.\ finite.
\end{Prop}
\begin{proof}
Define
\begin{equation*}
X^*_t := \sum_{k \geq 0} h(S_k^*)\1_{\{S_k^* \leq t\}} - \frac{1}{\mu} \int_0^t h(y) \dy,   \quad   t \geq 0.
\end{equation*}
Our aim is to show that
$X^*_t$ converges as $t \to \infty$ in the asserted sense.

We start with the case (C1) and first prove the result assuming that $h$ is decreasing on $\R_+$.
We then have to show that $X^*_t$ converges in $\mathcal{L}^2$ as $t \to \infty$, equivalently,
\begin{equation*}
\lim_{s \to \infty} \sup_{t>s} \, \E (X^*_t-X^*_s)^2 ~=~    0.
\end{equation*}
Since
\begin{equation*}
X^*_t - X^*_s
~=~ \sum_{k \geq 0} h(S_k^*) \1_{\{s < S_k^* \leq t\}} - \E \sum_{k\geq 0} h(S_k^*) \1_{\{s < S_k^* \leq t\}}
\end{equation*}
for $t>s$, we conclude that
\begin{eqnarray*}
\E(X^*_t-X^*_s)^2
& = &
\E \bigg(\sum_{k \geq 0} h(S_k^*) \1_{\{s<S_k^*\leq t\}}\bigg)^2
-\bigg(\E \sum_{k \geq 0}h(S_k^*) \1_{\{s<S_k^*\leq t\}}\bigg)^2    \\
& = &
\E \bigg(\sum_{k\geq 0}h(S_k^*) \1_{\{s<S_k^* \leq t\}}\bigg)^2-\bigg(\frac{1}{\mu} \int_s^t h(y) \dy \bigg)^2  \\
& = &
\E \bigg(\sum_{k \geq 0}h(t-S_k^*) \1_{\{S_k^* < t-s\}}\bigg)^2 - \bigg(\frac{1}{\mu} \int_s^t h(y) \dy \bigg)^2,
\end{eqnarray*}
where the last equality follows from Proposition \ref{Prop:N^ backwards in time}.
The first term on the right-hand side equals
\begin{align*}
\E & \sum_{k\geq 0} h(t\!-\!S_k^*)^2 \! \1_{\{S_k^*< t-s\}}
+ 2 \E \! \sum_{0 \leq i<j} h(t\!-\!S_i^*) \! \1_{\{S_i^*<t-s\}} h(t\!-\!S_j^*) \! \1_{\{S_j^*< t-s\}}  \\
& =~    \frac{1}{\mu}\int_{0}^{t-s} \!\!\! h(t\!-\!y)^2 \dy
+ \frac{2}{\mu} \int_{0}^{t-s} \!\!\! h(t\!-\!y) \int_{(0,\,t-s-y)} \!\! h(t\!-\!y\!-\!x) \, U(\dx) \dy \\
&=~ \frac{1}{\mu}\int_s^th(y)^2 \dy + \frac{2}{\mu}\int_s^th(y)\int_{(0,\,y-s)}h(y-x)\, U(\dx)\dy.
\end{align*}
Hence
\begin{align*}
\E (X^*_t&-X^*_s)^2 \\
&=~ \frac{1}{\mu} \int_s^t h(y)^2 \dy
+ \frac{2}{\mu} \int_s^t h(y) \int_{(0,\,y-s)} h(y-x) \, \mathrm{d}(U(x)\!-\!\mu^{-1}x) \dy.
\end{align*}
Since $h^2$ is assumed to be integrable,
$\lim_{s \to \infty} \sup_{t>s} \int_s^t h(y)^2 \dy = 0$ and it
remains to check that
\begin{equation} \label{eq:remainder}
\lim_{t \to \infty} \, \sup_{t>s} \, \int_s^t h(y) \int_{(0,\,y-s)} h(y-x) \, \mathrm{d}(U(x)\!-\!\mu^{-1}x)\dy ~=~ 0.
\end{equation}
Put $H_{s,t}(x) := \int_{s}^{t-x} h(x+y) h(y) \dy$ for $x \in [0,t-s)$ and $H_{s,t}(x):=0$ for all other $x$.
Note that $H_{s,t}(x)$ is right-continuous and decreasing on $[0,\infty)$.
Changing the order of integration followed by integration by parts gives
\begin{align*}
\int_s^t & h(y) \int_{(0,\,y-s)} h(y-x) \, \mathrm{d}\big(U(x)-\mu^{-1}x\big) \dy   \\
&=~ \int_{(0,t-s)} \int_{s}^{t-x} h(x+y) h(y) \dy \mathrm{d} \big(U(x)-\mu^{-1}x\big)   \\
&\leq~  \int_{(0,\,t-s)}(U(x)-\mu^{-1}x) \, \mathrm{d}(\!-H_{s,t}(x))   \\
&\leq~  \sup_{x\geq 0}\,\big|U(x)-\mu^{-1} x \big| \, H_{s,t}(0)    \\
&=~ \sup_{x\geq 0}\,\big|U(x)-\mu^{-1} x \big| \, \int_s^t h(y)^2 \dy.
\end{align*}
It is known (see Theorem XI.3.1 in \cite{Feller:1971})
that $\lim_{t \to \infty} (U(t)-\mu^{-1}t) = \mu^{-2}(\sigma^2 + \mu^2) < \infty$,
hence $\sup_{x\geq 0} |U(x)-\mu^{-1}x|<\infty$, and \eqref{eq:remainder} follows.

Next we assume that $h$ is only eventually decreasing (rather than decreasing everywhere).
Then we can pick some $t_0 > 0$ such that $h$ is decreasing on $[t_0,\infty)$.
Define $\overline{h}(t) := h(t_0+t)$, $t \geq 0$. Then $\overline{h}$ is decreasing on $\R_+$.
Further, the post-$t_0$ walk $(\overline{S}_k^*)_{k \in \N_0} := (S_{N^*(t_0)+k}^*-t_0)_{k \in \N_0}$ is a distributional copy of $(S_k^*)_{k \in \N_0}$.
Therefore, by what we have already shown,
\begin{equation*}
\overline{X}^*_{\circ}
~:=~ \lim_{t \to \infty} \, \bigg(\sum_{k \geq 0} \overline{h}(\overline{S}_k^*) \1_{\{\overline{S}_k^* \leq t\}} - \frac{1}{\mu} \int_0^t \overline{h}(y) \dy \bigg)
\end{equation*}
exists in the $\mathcal{L}^2$-sense.
Therefore, also
\begin{eqnarray*}
X^*_{\circ}
& = &
\lim_{t \to \infty} \, \bigg( \sum_{k \geq 0} h(S_k^*) \1_{\{S_k^* \leq t_0+t\}} - \frac{1}{\mu} \int_0^{t_0+t} h(y) \dy \bigg) \\
& = &
X_{t_0}^* +  \lim_{t \to \infty} \, \bigg( \sum_{k \geq 0} \overline{h}(\overline{S}_k^*) \1_{\{\overline{S}_k^* \leq t\}} - \frac{1}{\mu} \int_0^t \overline{h}(y) \dy \bigg)
\end{eqnarray*}
exists in the $\mathcal{L}^2$-sense.

Now we turn to the cases (C2) and (C3).
We begin by assuming that $h$ satisfies the assumptions of the theorem and is decreasing and twice differentiable on $\R_+$ with $h'' \geq 0$.
The proof is divided into three steps.
\begin{itemize}
    \item[{\sc Step 1:}]
            Prove that if, as $n \to \infty$,
        $U_n:=\sum_{k=0}^n (h(S^*_k)-h(\mu k))$ converges a.s.\ in the case (C2) or converges in probability in the case (C3),
        then, as $t \to \infty$, $\sum_{k\geq 0} h(S_k^*) \1_{\{S_k^* \leq t\}}-\mu^{-1} \int_0^t h(y) \dy$ converges in the same sense.
    \vspace{-0,25cm}
    \item[{\sc Step 2:}]
        Prove that if the series $\sum_{j \geq 0} (\xi^*_j-\mu) \sum_{k\geq j} h'(\mu k)$ converges a.s.,
        then, as $n\to\infty$, $U_n$ converges a.s.\ in the case (C2) and converges in probability in the case (C3).
    \vspace{-0,25cm}
    \item[{\sc Step 3.}]
        Use the three series theorem to check that, under the conditions stated,
        the series $\sum_{j \geq 0}(\xi^*_j-\mu)\sum_{k\geq j} h'(\mu k)$ converges a.s.
\end{itemize}

{\sc Step 1}.

\noindent {\sc Case (C2)}.
Assume that $U_n$ converges a.s.
Then, by Lemma \ref{Lem:int<->sum},
the sequence $\sum_{k=0}^n h(S^*_k)-\mu^{-1} \int_{0}^{\mu n} h(y) \dy$ converges a.s., too.
Since $\lim_{t \to \infty} N^*(t)=\infty$ a.s., we further have that
$\sum_{k=0}^{N^*(t)-1} h(S^*_k) - \mu^{-1} \int_{0}^{\mu(N^*(t)-1)} h(y) \dy$ converges a.s.\ as $t \to \infty$.
To complete this step, it remains to prove that
\begin{equation}    \label{eq:integral difference 1st}
\lim_{t \to \infty} \bigg|\int_{0}^{\mu(N^*(t)-1)} h(y) \dy - \int_{0}^t h(y) \dy \bigg| = 0    \quad   \text{a.s.}
\end{equation}
To this end, write
\begin{align}
\bigg|\int_{0}^{\mu(N^*(t)-1)} & h(y) \dy - \int_0^t h(y) \dy \bigg|    \nonumber   \\
&=~
\int_{\mu(N^*(t)-1) \wedge t}^{\mu(N^*(t)-1) \vee t} h(y) \dy           \nonumber   \\
&\leq~
|\mu(N^*(t)-1)- t| \, h(\mu(N^*(t)-1) \wedge t),    \label{eq:integral difference}
\end{align}
where the inequality follows from the monotonicity of $h$.
By Theorem 3.4.4 in \cite{Gut:2009}, $\E \xi^r<\infty$ implies that
\begin{equation}    \label{eq:Marcinkiewicz-Zygmund SLLN of N(t)}
N(t)-\mu^{-1}t ~=~ o(t^{1/r}) \quad \text{a.s.\ as } t \to \infty,
\end{equation}
where it should be recalled that
\begin{equation*}
N(t)    ~:=~    \inf\{k\in\N: S_k > t\} ~=~ \inf\{k \in \N: S^*_k-S^*_0>t\}.
\end{equation*}
Since
\begin{equation*}
N^*(t) ~=~ \1_{\{S_0^*\leq t\}} + N(t-S^*_0) \1_{\{S_0^* \leq t\}}  \quad   \text{a.s.}
\end{equation*}
and $S_0^*$ is a.s.\ finite, we infer
\begin{equation*}   
N^*(t)-\mu^{-1}t ~=~ o(t^{1/r}) \quad   \text{a.s.\ as } t \to \infty.
\end{equation*}
This relation implies that the first factor in \eqref{eq:integral difference} is
$o(t^{1/r})$, whereas the second factor is $o(t^{-1/r})$ as $t \to \infty$.
The latter relation can be derived as follows.
First, in view of \eqref{eq:int h^r<infty}
and the monotonicity of $h$,
we have
\begin{equation}    \label{eq:h(t)=o(t^(-1/r))}
h(t) ~=~ o(t^{-1/r}) \quad  \text{as } t \to \infty.
\end{equation}
Second, by the strong law of large numbers for $N^*(t)$, we have
\begin{equation*}   
[\mu (N^*(t)-1)] \wedge t ~\sim~ t \quad    \text{a.s.\ as } t \to \infty.
\end{equation*}
Altogether, \eqref{eq:integral difference 1st} has been proved.

\noindent {\sc Case (C3)}.
Assume that $U_n$ converges in probability.
In view of Lemma \ref{Lem:int<->sum} we conclude that, as $t\to\infty$,
$\sum_{k=0}^{\lfloor t/\mu \rfloor} h(S_k^*)-\mu^{-1}\int_0^ {\mu \lfloor t/\mu \rfloor} h(y) \dy$ converges in probability, too.

From \eqref{eq:int h^alphal(1/h)<infty} it follows that
$h(t)^{\alpha}\ell(h(t)^{-1})=o(t^{-1})$. This and \eqref{eq:t(l(c(t))/c(t)^alpha->1}
imply that
\begin{equation*}
0 = \lim_{t \to \infty} t h(t)^{\alpha}\ell(h(t)^{-1})
=   \lim_{t \to \infty} c(t)^{\alpha} h(t)^{\alpha} \frac{\ell(h(t)^{-1})}{\ell(c(t))}
=   \lim_{t \to \infty} \frac{\Prob\{\xi>h(t)^{-1}\}}{\Prob\{\xi > c(t)\}}.
\end{equation*}
From this, using the monotonicity and regular variation of $x \mapsto \Prob\{\xi > x\}$, we conclude that
\begin{equation}    \label{eq:c(t) h(t) -> 0}
\lim_{t \to \infty} c(t) h(t) = 0
\end{equation}
The latter relation implies that $\lim_{t \to \infty} h(t)=0$. Hence
\begin{equation*}
\lim_{t \to \infty} \bigg(\int_0^{\mu \lfloor t/\mu \rfloor} h(y) \dy - \int_0^t h(y) \dy \bigg) ~=~ 0.
\end{equation*}
Further,
\begin{equation*}
\lim_{t \to \infty} \frac{N^*(t)\wedge (\lfloor t/\mu \rfloor +1)}{t} = \mu^{-1} \text{ a.s.}
\text{ \ and \ }
\lim_{t \to \infty} \frac{S_{N^*(t)\wedge (\lfloor t/\mu \rfloor +1)}}{N^*(t) \wedge (\lfloor t/\mu \rfloor +1)} = \mu  \text{ a.s.}
\end{equation*}
by the strong laws of large numbers for renewal processes and random walks, respectively.
Hence
\begin{equation*}
\frac{S_{N^*(t)\wedge (\lfloor t/\mu \rfloor+1)}}{t} ~=~ 1  \quad   \text{a.s.}
\end{equation*}
Using this and \eqref{eq:h(t)=o(t^(-1/r))} we obtain that
\begin{eqnarray*}
\bigg|\sum_{k=0}^{\lfloor t/\mu \rfloor } h(S_k^*)-\sum_{k=0}^{N^*(t)-1} h(S_k^*)\bigg|
& = &
\sum_{k=N^*(t)\wedge (\lfloor t/\mu \rfloor +1)}^{(N^*(t)-1)\vee \lfloor t/\mu \rfloor } h(S_k^*)   \\
& \leq &
\big|N^*(t)-1-\lfloor t/\mu \rfloor \big|h\big(S_{{N^*(t)\wedge(\lfloor t/\mu \rfloor +1)}}\big)    \\
& = &
\big|N^*(t)-1-\lfloor t/\mu \rfloor \big| o(1/c(t)).
\end{eqnarray*}
From \eqref{eq:FLT for N(t)} for $u=1$
we get that $\frac{|\mu(N^*(t)-1)- t|}{c(t)}$ converges in distribution to an $\alpha$-stable law with characteristic function given by \eqref{eq:stable ch f}.
This entails
\begin{equation*}
\sum_{k=0}^{\lfloor t/\mu \rfloor } h(S_k^*)-\sum_{k=0}^{N^*(t)-1} h(S_k^*) ~\stackrel{\Prob}{\to}~ 0
\quad   \text{as } t \to \infty.
\end{equation*}
Combining pieces together gives the needed conclusion for this step.

{\sc Step 2.}
For each $k \in \N$, by Taylor's formula, there exists a
$\theta_k$ between $S^*_k$ and $\mu k$ such that
\begin{equation*}
h(S_k^*)-h(\mu k)   ~=~ h'(\mu k)(S_k^*-\mu k) + \frac{1}{2} h''(\theta_k)(S_k^*-\mu k)^2.
\end{equation*}
Set
\begin{equation*}
I_n ~:=~ \frac{1}{2} \sum_{k=1}^n h''(\theta_k)(S_k^*-\mu k)^2
\end{equation*}
and write
\begin{eqnarray}
U_n-h\big(S_0^*\big)+h(0)
& = &
\sum_{k=1}^nh'(\mu k)(S^*_k-\mu k)+I_n  \nonumber   \\
& = &
S^*_0 \sum_{k=1}^n h'(\mu k)+\sum_{k=1}^n(\xi_k-\mu) \sum_{j=k}^n h'(\mu j)+I_n     \nonumber   \\
& = &
S^*_0 \sum_{k=1}^n h'(\mu k)+\sum_{k=1}^n(\xi_k-\mu) \sum_{j\geq k}h'(\mu j)        \nonumber   \\
& &
-(S_n-\mu n) \sum_{k \geq n+1} h'(\mu k) + I_n. \label{eq:4 terms}
\end{eqnarray}
Since $-h'$ is decreasing and nonnegative we have
\begin{equation}    \label{eq:sum -h'(muk) leq 1/mu h(mun)}
\sum_{k \geq n+1} -h'(\mu k)
~\leq~ \int_n^{\infty} -h'(\mu y) \dy
~=~ \mu^{-1} h(\mu n)
~\leq~ \sum_{k \geq n} -h'(\mu k).
\end{equation}
for all $n$. Using the first inequality in \eqref{eq:sum -h'(muk) leq 1/mu h(mun)} and the fact that $\lim_{y \to \infty} h(y) = 0$,
one immediately infers that the first summand in the penultimate line of \eqref{eq:4 terms} converges as $n \to \infty$.
The a.s.\ convergence of the second (principal) term is assumed to hold here.
As to the third and fourth terms, we have to consider the cases (C2) and (C3) separately.

\noindent {\sc Case (C2)}.
By the Marcinkiewicz-Zygmund law of large numbers \cite[Theorem 2 on p.~125]{Chow+Teicher:1997},
\begin{equation}    \label{eq:MZ SLLN for S_n-mun}
S_n-\mu n ~=~ o(n^{1/r}) \quad  \text{as } n \to \infty \text{ a.s.}
\end{equation}
Therefore, in view of \eqref{eq:h(t)=o(t^(-1/r))} and \eqref{eq:sum -h'(muk) leq 1/mu h(mun)},
the third term converges to zero a.s.
Further,
$\lim_{k \to \infty} k^{-1}\theta_k=\mu$ a.s.\ by the strong law of large numbers.
Hence, in view of \eqref{eq:h''=Ot^(-2-1/r)},
\begin{equation*}
h''(\theta_k) ~=~ O(\theta_k^{-2-1/r}) ~=~ O(k^{-2-1/r}) \quad \text{as } k \to \infty.
\end{equation*}
From \eqref{eq:MZ SLLN for S_n-mun} we infer
\begin{equation*}
h''(\theta_k) (S^*_k-\mu k)^2 ~=~ o(k^{-(2-1/r)}) \quad \text{a.s.\ as } k \to \infty,
\end{equation*}
which implies that $I_n$ converges a.s.\ as $n \to \infty$, for $2-1/r>1$.
Hence the a.s.\ convergence of
$\sum_{k\geq 1} (\xi^*_k-\mu) \sum_{j\geq k} h^\prime (\mu j)$ entails that of $U_n$.

\noindent {\sc Case (C3)}.
In the given situation,
$\frac{S_n-\mu n}{c(n)}$ converges in distribution to an $\alpha$-stable law.
Hence, in view of \eqref{eq:c(t) h(t) -> 0} and \eqref{eq:sum -h'(muk) leq 1/mu h(mun)},
the third term converges to zero in probability.

\noindent
Now pick some $0 < \varepsilon < \alpha - 1$.
Since $\E \xi^{\alpha-\varepsilon}<\infty$, we conclude
(again from the Marcinkiewicz-Zygmund law of large numbers, \cite[Theorem 2 on p.~125]{Chow+Teicher:1997})
that \eqref{eq:MZ SLLN for S_n-mun} holds with $r=\alpha-\varepsilon$.
Using \eqref{eq:h''=O(t^(-2)c^(-1)(t)} and the facts that $\theta_k \sim \mu k$ a.s.\
and that $c(t) \sim t^{1/\alpha}L(t)$ for some slowly varying $L$ (see Lemma \ref{Lem:g regularly varying}),
we conclude:
\begin{equation*}
h''(\theta_k)
~=~ O(\theta_k^{-2}c(\theta_k)^{-1})
~=~ O(k^{-2-1/\alpha}L(k)^{-1}) \quad   \text{a.s.\ as } k \to \infty.
\end{equation*}
Therefore,
\begin{equation*}
h''(\theta_k)(S^*_k-\mu k)^2    ~=~ o(k^{-(2-\frac{\alpha+\varepsilon}{\alpha(\alpha-\varepsilon)})}L(k)^{-1})
\quad \text{a.s.,} \quad k\to\infty,
\end{equation*}
which implies that the fourth term $I_n$
converges a.s., as for sufficiently small $\varepsilon$,
$2-\frac{\alpha+\varepsilon}{\alpha(\alpha-\varepsilon)} > 1$.
Hence we arrive at the conclusion that the a.s.\ convergence of
$\sum_{k\geq 1} (\xi_k-\mu) \sum_{j \geq k} h' (\mu j)$
entails convergence in probability of $U_n$.

{\sc Step 3.}
Set
\begin{equation*}
c_k := \sum_{j \geq k} -h'(\mu j)   \quad   \text{and} \quad    \zeta_k := -c_k(\xi_k-\mu), \quad   k \in \N.
\end{equation*}

\noindent {\sc Case (C2)}.
Condition \eqref{eq:int h^r<infty} ensures that
$\sum_{k\geq 1} h(\mu k)^r < \infty$.
In view of \eqref{eq:sum -h'(muk) leq 1/mu h(mun)},
\begin{eqnarray*}
\sum_{k \geq 1}\E |\zeta_k|^r
& = &
\E |\xi-\mu|^r\sum_{k \geq 1} \left(\sum_{j\geq k} (-h'(\mu j))\right)^r    \\
& \leq &
\mu^{-r} \E |\xi-\mu|^r \sum_{k \geq 1} h(\mu (k-1))^r
~<~ \infty.
\end{eqnarray*}
Hence the series $\sum_{k \geq 1} \zeta_k$ converges a.s.\ by Corollary 3 on p.\;117 in \cite{Chow+Teicher:1997}.

\noindent {\sc Case (C3)}.
By the three series theorem \cite[Theorem 2 on p.\;117]{Chow+Teicher:1997},
it suffices to show that the following series converge
\begin{equation*}
\sum_{k \geq 1} \Prob\{|\zeta_k|>1\},
\quad   \sum_{k\geq 1} \E [\zeta_k \1_{\{|\zeta_k|\leq 1\}}]
\quad   \text{and}  \quad   \sum_{k \geq 1} \Var [\zeta_k \1_{\{|\zeta_k|\leq 1\}}].
\end{equation*}
By Markov's inequality, the first series converges if $\sum_{k \geq 1} \E [|\zeta_k| \1_{\{|\zeta_k| > 1\}}]$ converges.
Since $\E \zeta_j=0$ for all $j \geq 1$, the second series converges if and only if the
series $\sum_{k \geq 1} \E [\zeta_k \1_{\{|\zeta_k| > 1\}}]$ converges.
By Theorem 1.6.5 in \cite{Bingham+Goldie+Teugels:1989},
\begin{equation*}
\E [|\zeta_k| \1_{\{|\zeta_k| > 1\}}] ~=~ c_k \int_{[c_k^{-1},\,\infty)} \!\!\! x \Prob\{|\xi-\mu|\in \dx\}
~\sim~ \frac{\alpha}{\alpha-1}c^\alpha_k \ell(c_k^{-1})
\quad   \text{as } k \to \infty.
\end{equation*}
Hence, recalling
\eqref{eq:sum -h'(muk) leq 1/mu h(mun)} and \eqref{eq:int h^alphal(1/h)<infty},
\begin{equation*}
\sum_{k\geq 1} \E [|\zeta_k| \1_{\{|\zeta_k| > 1\}}] ~<~ \infty.
\end{equation*}
Further, by Theorem 1.6.4 in \cite{Bingham+Goldie+Teugels:1989},
\begin{equation*}
\E [\zeta_k^2 \1_{\{|\zeta_k|\leq 1\}}]
~=~ c^2_k\int_{[0,c_k^{-1}]} x^2\Prob\{|\xi-\mu| \in \dx\}
~\sim~  \frac{\alpha}{2-\alpha} c^\alpha_k\ell(c_k^{-1})
\quad \text{as } k \to \infty.
\end{equation*}
Again using \eqref{eq:sum -h'(muk) leq 1/mu h(mun)} and \eqref{eq:int h^alphal(1/h)<infty}, this entails
\begin{equation*}
\sum_{k\geq 1}\Var [\zeta_k \1_{\{|\zeta_k|\leq 1\}}]
~\leq~
\sum_{k\geq 1}\E [\zeta_k^2 \1_{\{|\zeta_k|\leq 1\}}]
~<~\infty.
\end{equation*}

Finally, we need to prove that the assertion also holds for $h$ that are only eventually decreasing and eventually twice differentiable with $h'' \geq 0$ eventually.
Indeed, for any such $h$, there is some $t_0 > 0$ such that $h$ is decreasing and twice differentiable on $[t_0,\infty)$ with $h'' \geq 0$ on $[t_0,\infty)$.
Using this $t_0$, define $\overline{h}$ and $(\overline{S}_k^*)_{k \in \N_0}$ as in the proof in the case (C1).
Notice that with $h$, also $\overline{h}$ satisfies the assumptions of the theorem, for instance, in case (C3),
$\overline{h}''(t) = h''(t_0+t) = O((t_0+t)^{-2} c(t_0+t)^{-1}) = O(t^{-2} c(t)^{-1})$ as $t \to \infty$.
Now one can argue as in the corresponding part of the proof in the case (C1) to conclude that $X^*_{\circ}$ exists
as the a.s.\ limit or the limit in probability, respectively.
\end{proof}

\subsection{One-dimensional convergence}

The proofs of Theorems \ref{Thm:dRi h} and \ref{Thm:decreasing h}
are preceded by the corresponding statements on one-dimensional convergence and their proofs.

\begin{Prop}    \label{Prop:dRi h}
Assume that $F$ is non-lattice and let $h$ be d.R.i.
\begin{itemize}
    \item[(a)]
        If $\mu < \infty$, then the random series $X^*$ converges a.s.\ and
        \begin{equation*}   
        X(t)        ~\stackrel{\mathrm{d}}{\to}~     X^*    \quad   \text{as } t \to \infty.
        \end{equation*}
    \item[(b)]
        If $\mu = \infty$, then
        \begin{equation*}   
        X(t)    ~\stackrel{\mathcal{L}^1}{\to}~ 0   \quad   \text{as } t \to \infty.
        \end{equation*}
\end{itemize}
\end{Prop}
\begin{proof}
Assertion (b) is a consequence of the key renewal theorem. Turning
to assertion (a) we assume that $\mu < \infty$. For $\varepsilon
> 0$ let $\tau_1 = \tau_1(\varepsilon)$ and $\tau_2 = \tau_2(\varepsilon)$ be defined as in Section
\ref{subsec:stationary renewal processes and coupling}. Recall
that $\tau_1\vee\tau_2 < \infty$ a.s.

Further, let $h^{\varepsilon}(t) := \sup_{|s-t| \leq \varepsilon}
h(s)$, $t \geq 0$. Since $h$ is d.R.i., we have
$$C~:=~\underset{t\geq 0}{\sup}\,|h(t)|\in [0,\infty).$$ Then,
using \eqref{eq:epsilon close} for $t > \varepsilon$ and its
consequence
\begin{equation*}
|\1_{\{S_{\tau_2+k} \leq t\}}-\1_{\{\tilde{S}_{\tau_1+k}^* \leq
t\}}| ~\leq~ \1_{\{t - \varepsilon < \tilde{S}_{\tau_1+k}^* \leq
t+\varepsilon\}} \quad   \text{for } k \in \N_0,
\end{equation*}
we infer
\begin{align*}
X(t) & \leq~
\sum_{k=0}^{\tau_2-1} h(t\!-\!S_k) \1_{\{S_k \leq t\}} + \sum_{k \geq 0} h^{\varepsilon}(t\!-\!\tilde{S}^*_{\tau_1+k}) \1_{\{S_{\tau_2+k} \leq t\}}  \\
& =~
\sum_{k=0}^{\tau_2-1} h(t\!-\!S_k) \1_{\{S_k \leq t\}} + \sum_{k \geq 0} h^{\varepsilon}(t\!-\!\tilde{S}_{\tau_1+k}^*) \1_{\{\tilde{S}^*_{\tau_1+k} \leq t\}}    \\
& \hphantom{\leq~}
+ \sum_{k \geq 0} h^{\varepsilon}(t\!-\!\tilde{S}_{\tau_1+k}^*) (\1_{\{S_{\tau_2+k} \leq t\}}-\1_{\{\tilde{S}_{\tau_1+k}^* \leq t\}})   \\
& \leq~
\sum_{k=0}^{\tau_2-1} h(t\!-\!S_k) \1_{\{S_k \leq t\}} + \sum_{k \geq \tau_1} h^{\varepsilon}(t\!-\!\tilde{S}_k^*) \1_{\{\tilde{S}^*_k \leq t\}} + C \sum_{k \geq \tau_1} \1_{\{t-\varepsilon< \tilde{S}_k^* \leq t+\varepsilon\}}    \\
& \leq~ \sum_{k=0}^{\tau_2-1} h(t\!-\!S_k)\1_{\{S_k \leq
t\}}\!-\sum_{k=0}^{\tau_1-1}h^{\varepsilon}(t\!-\!\tilde{S}_k^*)
\1_{\{\tilde{S}^*_k \leq t\}}
+ \sum_{k \geq 0} h^{\varepsilon}(t\!-\!\tilde{S}^*_k) \1_{\{\tilde{S}^*_k \leq t\}}    \\
& \hphantom{\leq}~ +
C(\tilde{N}^*(t\!+\!\varepsilon)\!-\!\tilde{N}^*(t\!-\!\varepsilon)),
\end{align*}
where in the fourth line the inequality $h^\varepsilon(t)\leq C$,
$t\geq 0$, has been utilized. Regarding the last term, we have
$\tilde{N}^*(t\!+\!\varepsilon)\!-\!\tilde{N}^*(t\!-\!\varepsilon)
\stackrel{\mathrm{d}}{=} N^*(2\varepsilon) \to 0$ a.s.\ as
$\varepsilon \downarrow 0$. Since $h$ is d.R.i., so is
$h^\varepsilon$, by Lemma \ref{Lem:h^eps is dRi}. Hence $\lim_{t
\to \infty} h(t) = \lim_{t \to \infty} h^{\varepsilon}(t) = 0$,
and the first two summands in the penultimate line of the
displayed equation tend to $0$ a.s.\ as $t \to \infty$. Regarding
the third term of the displayed equation, Proposition \ref{Prop:N^
backwards in time} and Proposition \ref{Prop:X* well-defined} give
\begin{equation*}
\sum_{k \geq 0} h^{\varepsilon}(t\!-\!\tilde{S}^*_k)
\1_{\{\tilde{S}^*_k \leq t\}} ~\stackrel{\mathrm{d}}{=}~ \sum_{k
\geq 0} h^{\varepsilon}(S^*_k) \1_{\{S^*_k \leq t\}} ~\underset{t
\to \infty}{\to}~  \sum_{k \geq 0} h^{\varepsilon}(S^*_k) ~=:~
X^{\varepsilon,*} \ \text{a.s.}
\end{equation*}
Further, the d.R.i.\ of $h$ implies that $h$ is a.e.\ continuous,
which in turn implies that $h^{\varepsilon} \downarrow h$ a.e.\ as
$\varepsilon \downarrow 0$. Lemma \ref{Lem:approximating X*} thus
gives that $\lim_{\varepsilon \to 0} X^{\varepsilon,*} = X^*$ a.s.
We conclude that
\begin{equation*}
\limsup_{t\to\infty} \, \Prob \{{X}(t) > x\}        ~\leq~
\Prob\{X^*>x\}
\end{equation*}
for every continuity point $x$ of the law of $X^*$.

More precisely, let $(\varepsilon_n)_{n \in \N}$ be a sequence
with $\varepsilon_n \downarrow 0$ as $n \to \infty$. Let $x$ be a
continuity point of the law of $X^*$ and $x-\delta$ ($\delta>0$)
be a continuity point of the law $X^*$ and of the laws of
$X^{\varepsilon_n,*}$. (The set of these $\delta$ is dense in
$\R$.) Then,
\begin{align*}
\limsup_{t \to \infty} & \, \Prob\{X(t) > x\}   \\
& \leq~ \limsup_{t \to \infty} \Prob\bigg\{\sum_{k=0}^{\tau_2-1}
h(t\!-\!S_k)\1_{\{S_k \leq t\}}-
\sum_{k=0}^{\tau_1-1} h^{\varepsilon_n}(t\!-\!\tilde{S}_k^*) \1_{\{\tilde{S}^*_k \leq t\}}\! > \delta/2\bigg\}   \\
& \hphantom{\leq~} + \limsup_{t \to \infty} \Prob\{C(\tilde{N}^*(t\!+\!\varepsilon_n)\!-\!\tilde{N}^*(t\!-\!\varepsilon_n))  > \delta/2\}   \\
& \hphantom{\leq~} + \limsup_{t \to \infty} \Prob\bigg\{\sum_{k \geq 0} h^{\varepsilon_n}(t\!-\!\tilde{S}^*_k) \1_{\{\tilde{S}^*_k \leq t\}} > x-\delta \bigg\} \\
&=~ \Prob\{X^{\varepsilon_n,*}>x-\delta\} + \Prob\{C
N^*(2\varepsilon_n)>\delta/3\}.
\end{align*}
As $n \to \infty$, the second probability goes to zero, whereas
the first tends to $\Prob\{X^*>x-\delta\}$. Sending now $\delta
\downarrow 0$ along an appropriate sequence, we arrive at the
desired conclusion. Corresponding lower bounds can be obtained
similarly.
\end{proof}

\begin{Prop}    \label{Prop:decreasing h}
Assume that $\mu < \infty$ and that $F$ is non-lattice.
Let $h:\R_+ \to \R$ be locally bounded, a.e.\ continuous, eventually decreasing and non-integrable with $\lim_{t \to \infty} h(t) = 0$.
If $X^*_{\circ}$ exists as a limit in probability, then
\begin{equation}    \label{eq:X(t) when h^2 integrable}
        X_{\circ}(t) ~\stackrel{\mathrm{d}}{\to}~ X^*_{\circ}   \quad   \text{as } t \to \infty.
        \end{equation}
In particular, \eqref{eq:X(t) when h^2 integrable} holds under the assumptions of Theorem \ref{Thm:decreasing h}.
Further, $X(t) - \E X(t) \stackrel{\mathrm{d}}{\to} X^*_{\circ}$ as $t \to \infty$
in the case (C1) of Theorem \ref{Thm:decreasing h}.
\end{Prop}
\begin{proof}
Since $h$ is assumed to be eventually decreasing, there exists a
$t_0 > 0$ such that $h(t)$ is decreasing on $[t_0,\infty)$. Define
$h_1(t) := h(t_0) \1_{[0,\,t_0]}(t) + h(t) \1_{(t_0,\,\infty)}(t)$
and $h_2(t) := h(t)-h_1(t)$. Consequently, $X_{\circ}(t) =
X_{1,\circ}(t)+X_{2,\circ}(t)$, $t \geq 0$ where
\begin{equation*}
X_{j,\circ}(t) ~:=~ \sum_{k \geq 0} h_{j}(t-S_k) \1_{\{S_k \leq
t\}} - \frac{1}{\mu} \int_0^t h_j(y) \dy,   \quad   j=1,2,
\end{equation*}
and $h_1$ is nonnegative and decreasing, and $h_2$ is d.R.i. The
idea is to conclude convergence of $X_{2,\circ}(t)$ as in the
proof of Proposition \ref{Prop:dRi h} and to use the monotonicity
of $h_1$ to infer convergence of $X_{1,\circ}(t)$. However,
convergence in distribution of $X_{1,\circ}(t)$ and
$X_{2,\circ}(t)$ does not imply convergence in distribution of
their sum. This is why we will treat the two terms simultaneously.

We will only prove that
\begin{equation}\label{upp}
\limsup_{t \to \infty} \Prob\{X_{\circ}(t)
> x\} \leq \Prob\{X_{\circ}^* > x\}
\end{equation}
at any continuity point $x$ of the law of $X_{\circ}^*$. The
converse inequality for the lower limit can be obtained similarly.
As in the proof of Proposition \ref{Prop:dRi h}, we use a coupling
argument but for technical reasons, the coupling differs slightly
from that introduced in Subsection \ref{subsec:stationary renewal
processes and coupling}. For any fixed $\varepsilon>0$, there
exist almost surely finite $\sigma_1$ and $\sigma_2$ such that
\begin{equation*}
\hat{S}^{*}_{\sigma_1}-S_{\sigma_2} \in [0,\varepsilon].
\end{equation*}
almost surely. Now we can define the coupled random walk
$(\tilde{S}^*_k)_{k \in \N_0}$ as in Subsection
\ref{subsec:stationary renewal processes and coupling} but with
$\tau_1$ and $\tau_2$ replaced by $\sigma_1$ and $\sigma_2$. This
has the advantage that the estimate \eqref{eq:epsilon close} can
be replaced by $S_{\sigma_2+k} \leq \tilde{S}_{\sigma_1+k}^* \leq
S_{\sigma_2+k} + \varepsilon  $ for all $k \in \N_0$. This choice
of coupling is convenient for the derivation of upper bounds since
together with the monotonicity of $h_1$ it implies that
\begin{equation}    \label{eq:h_1 coupling bound}
h_1(t-S_{\sigma_2+k}) ~\leq~ h_1(t-\tilde{S}_{\sigma_1+k}^*) \quad
\text{ for all } k \in \N_0.
\end{equation}
Our starting point is the following representation for $X(t)$:
\begin{equation}    \label{eq:upper bound for X(t)}
X(t) ~=~ \sum_{k=0}^{\sigma_2-1} h(t\!-\!S_k) \1_{\{S_k \leq t\}}
+ \sum_{k \geq \sigma_2} h_1(t\!-\!S_k) \1_{\{S_k \leq t\}} +
\sum_{k \geq \sigma_2} h_2(t\!-\!S_k) \1_{\{S_k \leq t\}}.
\end{equation}
Using \eqref{eq:h_1 coupling bound} and the monotonicity of $h_1$,
we infer:
\begin{align}
\sum_{k \geq \sigma_2} & h_1(t\!-\!S_k) \1_{\{S_k \leq t\}}   \notag  \\
& \leq~ \sum_{k \geq \sigma_1} h_1(t\!-\!\tilde{S}_k^*)
\1_{\{\tilde{S}_k^* \leq t\}}
+ h_1(0) \sum_{k \geq \sigma_1} \1_{\{t< \tilde{S}_k^* \leq t + \varepsilon\}}    \notag  \\
& \leq~ \sum_{k \geq 0} h_1(t\!-\!\tilde{S}_k^*)
\1_{\{\tilde{S}_k^* \leq t\}} +
h_1(0)(\tilde{N}^*(t\!+\!\varepsilon)\!-\!\tilde{N}^*(t)).
\label{eq:upper bound for X_1(t)}
\end{align}
With $h_2^{\varepsilon}(t) := \sup_{|s-t| \leq \varepsilon}
h_2(s)$ and $C := \sup_{0 \leq s \leq t_0} |h_2(s)|$, we infer as
in the proof of Proposition \ref{Prop:dRi h}:
\begin{align}
\sum_{k \geq \sigma_2} & h_2(t\!-\!S_k) \1_{\{S_k \leq t\}}   \notag  \\
&\leq~ \sum_{k \geq 0} h_2^{\varepsilon}(t\!-\!\tilde{S}^*_k)
\1_{\{\tilde{S}^*_k \leq t\}} - \sum_{k=0}^{\sigma_1-1}
h_2^{\varepsilon}(t\!-\!\tilde{S}^*_k) \1_{\{\tilde{S}_k^* \leq
t\}} + C(\tilde{N}^*(t\!+\!\varepsilon)\!-\!\tilde{N}^*(t)).
\label{eq:upper bound for X_2(t)}
\end{align}
Now fix $x \in \R$ and $\delta > 0$. Combining \eqref{eq:upper
bound for X(t)}, \eqref{eq:upper bound for X_1(t)} and
\eqref{eq:upper bound for X_2(t)}, we conclude
\begin{align}
\limsup_{t \to \infty} & \, \Prob\{X_{\circ}(t) > x\}   \notag  \\
&\leq~ \limsup_{t \to \infty} \Prob\bigg\{\sum_{k \geq 0}
h_1(t\!-\!\tilde{S}_k^*) \1_{\{\tilde{S}_k^* \leq t\}}
- \frac{1}{\mu} \int_{0}^{t} \!\! h_1(y) \dy    \notag  \\
&\hphantom{\leq~ \limsup_{t \to \infty} \Prob\bigg\{}
+\sum_{k \geq 0} h_2^{\varepsilon}(t\!-\!\tilde{S}^*_k) \1_{\{\tilde{S}^*_k \leq t\}} - \frac{1}{\mu} \int_0^t h_2^{\varepsilon}(y) \dy >x-\delta\bigg\}    \notag   \\
&\hphantom{\leq~} +\limsup_{t \to \infty}
\Prob \bigg\{\sum_{k=0}^{\sigma_2-1} h(t\!-\!S_k) \1_{\{S_k\leq t\}}-\sum_{k=0}^{\sigma_1-1}h_2^{\varepsilon}(t\!-\!\tilde{S}^*_k)\1_{\{\tilde{S}_k^* \leq t\}} > \delta/3 \bigg\}     \notag \\
&\hphantom{\leq~} +\limsup_{t \to \infty} \Prob\big\{(h_1(0)+C)(\tilde{N}^*(t\!+\!\varepsilon)\!-\!\tilde{N}^*(t)) > \delta/3\big\} \notag  \\
&\hphantom{\leq~} + \Prob\bigg\{\frac{1}{\mu}\int_0^{t_0} \!\!
(h_2^{\varepsilon}(y)\!-\!h_2(y)) \dy > \delta/3\bigg\}
\label{eq:4 terms'}
\end{align}
The last probability equals $0$ when $\varepsilon$ is small
enough, for $h_2^{\varepsilon} \downarrow h_2$ a.e.\ as
$\varepsilon \downarrow 0$. Since $h(t) \to 0$ as $t \to \infty$,
we further have
\begin{equation*}
\lim_{t \to \infty} \sum_{k=0}^{\sigma_2-1} h(t\!-\!S_k)
\1_{\{S_k\leq
t\}}=\lim_{t\to\infty}\sum_{k=0}^{\sigma_1-1}h_2^{\varepsilon}(t\!-\!\tilde{S}^*_k)\1_{\{\tilde{S}_k^*
\leq t\}}= 0 \quad   \text{a.s.}
\end{equation*}
This implies that the term in the fourth line of \eqref{eq:4
terms'} equals $0$. Regarding the term in the penultimate line of
\eqref{eq:4 terms'}, we have
\begin{equation*}
\limsup_{t \to \infty
}\Prob\big\{(h_1(0)+C)(\tilde{N}^*(t\!+\!\varepsilon)\!-\!\tilde{N}^*(t))
> \delta/3\big\} = \Prob\big\{(h_1(0)+C)N^*(\varepsilon) >
\delta/3\big\}.
\end{equation*}
This tends to $0$ as $\varepsilon \to 0$. It remains to deal with
the principal term, the first term on the right-hand side of
\eqref{eq:4 terms'}. By Proposition \ref{Prop:N^ backwards in
time},
\begin{align*}
\sum_{k \geq 0} h_1(t\!-\!\tilde{S}_k^*) & \1_{\{\tilde{S}_k^*
\leq t\}} - \frac{1}{\mu} \int_{0}^{t} \!\! h_1(y) \dy
+\sum_{k \geq 0} h_2^{\varepsilon}(t\!-\!\tilde{S}^*_k) \1_{\{\tilde{S}^*_k \leq t\}} - \frac{1}{\mu} \int_0^t h_2^{\varepsilon}(y) \dy \notag  \\
& \stackrel{\mathrm{d}}{=}~
\sum_{k \geq 0} h_1(S_k^*) \1_{\{S_k^* \leq t\}} - \frac{1}{\mu} \int_{0}^{t} \!\! h_1(y) \dy      \\
& \hphantom{\stackrel{\mathrm{d}}{=}~}
+ \sum_{k \geq 0} h_2^{\varepsilon}(S^*_k) \1_{\{S^*_k \leq t\}} - \frac{1}{\mu} \int_0^t h_2^{\varepsilon}(y) \dy    \\
& \stackrel{\Prob}{\to}~ X_{1,\circ}^* +
X_{2,\circ}^{\varepsilon,*} \quad   \text{as } t \to \infty.
\end{align*}
The existence of $X_{1,\circ}^*$ follows from the existence of
$X^*_{\circ}$ as a limit in probability and the existence of the
a.s.\ limit $\lim_{t\to\infty}\sum_{k \geq 0} h_2(S_k^*)
\1_{\{S_k^* \leq t\}} - \frac{1}{\mu} \int_0^t \!\! h_2(y) \dy$
which is secured by Proposition \ref{Prop:X* well-defined} (using
that $h_2$ is d.R.i.). The existence of $X_{2,\circ}^*$ follows
from Proposition \ref{Prop:X* well-defined} (using that
$h_2^\varepsilon$ is d.R.i.). From Lemma \ref{Lem:approximating
X*} and the fact that $h_2^{\varepsilon} \downarrow h_2$ a.e.\ as
$\varepsilon \downarrow 0$, it follows that
\begin{equation*}
X_{2,\circ}^{\varepsilon,*} ~\underset{\varepsilon \downarrow
0}{\to}~  \sum_{k \geq 0} h_2(S_k^*) - \frac{1}{\mu} \int_0^{t_0}
h_2(y) \dy \quad   \text{a.s.}
\end{equation*}
Hence $\lim_{\varepsilon \downarrow 0} (X_{1,\circ}^* +
X_{2,\circ}^{\varepsilon,*}) = X_{\circ}^*$ a.s. Now one can argue
as in the end of the proof of Proposition \ref{Prop:dRi h} to
infer \eqref{upp}.

It follows from Proposition \ref{Prop:X*_o well-defined} that the
assumptions of Theorem \ref{Thm:decreasing h} are sufficient for
\eqref{eq:X(t) when h^2 integrable} to hold. In the case (C1) of
Theorem \ref{Thm:decreasing h}, $\lim_{t \to \infty} |\E
X(t)-\mu^{-1}\int_{0}^{t} h(y) \dy|=0$ by Corollary 3.1 in
\cite{Kaplan:1975}. Hence, the limiting distribution of $X(t)-\E
X(t)$ as $t \to \infty$ is the same as that of $X^*_{\circ}(t)$.
\end{proof}

\subsection{Finite-dimensional convergence}

It remains to extend one-dimensional convergence to finite-dimensional convergence.
This is done in this subsection.

\begin{proof}[Proof of Theorem \ref{Thm:dRi h}]
We have to show that for all $0 < u_1 < \ldots < u_n < \infty$,
the random vector $(X(u_1t),\ldots,X(u_n t))$ converges to $(X^*(u_1),\ldots,X^*(u_n))$ in distribution as $t \to \infty$.

Assume that $n=2$. Without loss of generality we may take $u_1=1$.
We further write $u$ for $u_2>1$, set $m_1 = (1+u)/2$ and let
$m_2\in (m_1,u)$. For $t > 0$, set
\begin{equation*}
X_1(ut) := \!\! \sum_{k=0}^{N(m_1t)-1} \! h(ut-S_k) \1_{\{S_k \leq
ut\}} \text{ and } X_2(ut) := \!\! \sum_{k\geq N(m_1 t)} \!\!\!
h(ut-S_k) \1_{\{S_k \leq ut\}}.
\end{equation*}
Clearly, $X(ut)=X_1(ut)+X_2(ut)$ for all $t > 0$.

We first prove that
\begin{equation}\label{eq:X_1(ut)->0}
X_1(ut) \stackrel{\Prob}{\to} 0 \quad   \text{as } t \to \infty.
\end{equation}
For every $\varepsilon>0$ there exists an $c=c(\varepsilon)>0$
such that $\int_c^\infty|h(y)| \dy<\varepsilon$. Setting
$h_c(t):=h(t)\1_{[c,\,\infty)}(t)$, we have, for $t$ large enough,
\begin{eqnarray*}
\E |X_1(ut)|&\leq&\E \sum_{k=0}^{N(m_1t)-1}|h(ut-S_k)|
~=~ \int _{[0,\,m_1t]}|h(ut-y)| \, U(\dy)   \\
& \leq &
\int_{[0,\,ut]}|h_c(ut-y)| \, U(\dy)    \\
& \underset{t\to\infty}{\to}& \int_0^\infty |h_c(y)|\dy
~=~ \int_c^\infty |h(y)|\dy ~\leq~ \varepsilon.\\
\end{eqnarray*}
Sending $\varepsilon\downarrow 0$ finishes the proof of
\eqref{eq:X_1(ut)->0}.

Our next purpose is to show that
\begin{equation}    \label{eq:X,X_2}
\Prob\{X(t)\leq a, X_2(ut) \leq b\} ~\to~ \Prob\{X^* \leq a\} \Prob\{X^* \leq b\}
\end{equation}
as $t \to \infty$ for continuity points $a,b \in \R$ of the law of
$X^*$. Write the probability on the left-hand side of
\eqref{eq:X,X_2} as follows:
\begin{align*}
\Prob\{X(t)&\leq a,X_2(ut)\leq b\}  \\
&=~
\int_{(m_1t,\,\infty)} \Prob\{X(t)\leq a,X_2(ut)\leq b, S_{N(m_1t)} \in \dy\}   \\
& =~
\Big(\int_{(m_1t,\,m_2t]} \ldots +\int_{(m_2t,\,\infty)} \ldots \Big)
~=:~    J_1(t)+J_2(t).
\end{align*}
Clearly, $0\leq J_2(t)\leq \Prob\{S_{N(m_1t)} > m_2t\} = \Prob\{S_{N(m_1t)}-m_1t > (m_2-m_1) t \} \to 0$ as $t \to \infty$
since $S_{N(m_1t)}-m_1t \stackrel{\mathrm{d}}{\to} S^*_0$ and $(m_2-m_1)t \to +\infty$ as $t \to \infty$.
For $J_1(t)$ we may write:
\begin{equation*}
J_1(t)
~=~ \int_{(m_1t,\,m_2t]}\Prob\{X(ut-y)\leq b\}\Prob\{X(t)\leq a,S_{N(m_1t)}\in\dy \}
\end{equation*}
where we have used that $(S_{k+N(m_1t)}-S_{N(m_1t)})_{k \in \N_0}$ has the same distribution as $(S_k)_{k \in \N_0}$
and is independent of $(S_k)_{0 \leq k \leq N(m_1t)}$. Further,
\begin{eqnarray*}
J_1(t)
& = &
\Prob\{X^*\leq b\}\int_{(m_1t,\,m_2t]}\Prob\{X(t)\leq a,S_{N(m_1t)}\in\dy \}    \\
& &
+ \!\!\! \underset{(m_1t,\,m_2t]}{\int}\!\!\!\!(\Prob\{X(ut\!-\!y)\leq b\}-\Prob\{X^*\leq b\}) \Prob\{X(t)\leq a,S_{N(m_1t)}\in\dy \}   \\
& =: &
J_{11}(t)+J_{12}(t).
\end{eqnarray*}
The integral in the first summand converges to $\Prob\{X^*\leq  a\}$ by Proposition \ref{Prop:dRi h}
since $a$ is a continuity point of the law of $X^*$ and $\Prob\{m_1t \leq S_{N(m_1t)} \leq m_2t\} \to 1$ as $t\to\infty$.
To show that $J_{12}(t)$ converges to zero, write
\begin{align*}
|J_{12}&(t)|    \\
& =~
\sup_{y\in [m_1t,\,m_2t]}\Big|\Prob\{X(ut-y)\leq b\}-\Prob\{X^*\leq b\}\Big|
\int_{(m_1t,\,m_2t]} \!\!\! \Prob\{S_{N(m_1t)}\in\dy\}  \\
&\leq~
\sup_{y\geq (u-m_2)t}\Big|\Prob\{X(y)\leq b\}-\Prob\{X^*\leq b\}\Big|,
\end{align*}
which goes to zero since $(u-m_2)t \to \infty$, as $t\to\infty$,
in view of Proposition \ref{Prop:dRi h}. The proof of
\eqref{eq:X,X_2} is complete. Now the desired result for the case
$n=2$ follows from \eqref{eq:X_1(ut)->0}, \eqref{eq:X,X_2} and
Slutsky's lemma.

The case of general $n\in\N$ can be treated similarly
by conditioning the probability $\Prob\{X(u_1t)\leq a_1,\ldots,X(u_nt)\leq a_n\}$ on $(S_{N(m_1 t)},\ldots,S_{N(m_{n-1}t)})$
at appropriately chosen middle points $u_{i} < m_i < u_{i+1}$.
\end{proof}

The scheme of the proof of Theorem \ref{Thm:decreasing h} is the same as that of the proof of Theorem \ref{Thm:dRi h} above.
On the other hand, it differs in many details which is why we decided to include it here.

\begin{proof}[Proof of Theorem \ref{Thm:decreasing h}]
As in the case of Theorem \ref{Thm:dRi h}, we will prove this theorem only for $n=2$
and assume that $u_1=1$ and $u:=u_2>1$.
Let $p(t) := \mu^{-1} \int_{0}^{t} h(y) \dy$ and set $m_1:=(1+u)/2$, $z_2(t):=m_1t+r(t)$
where $r(t)$ is some function to be specified below.
Decompose $X_{\circ}(ut):=X(ut)-p(ut)$ as follows
\begin{eqnarray*}
X_{\circ}(ut)
& = &
\bigg(\sum_{k=0}^{N(m_1t)-1} \!\! h(ut-S_k) - \frac{1}{\mu}\int_{(u-m_1)t}^{ut} h(y) \dy \bigg) \\
& &
+ \bigg(\sum_{k\geq N(m_1t)} \!\! h(ut-S_k)\1_{\{S_k\leq ut\}} - p((u-m_1)t)\bigg) \\
& =: & Y_1(t)+Y_2(t).
\end{eqnarray*}
Similar to the proof of Theorem \ref{Thm:dRi h},
we conclude that it is enough to show that
\begin{equation}\label{eq:Y_1(t)->0}
Y_1(t) ~\stackrel{\Prob}{\to}~ 0,
\end{equation}
and
\begin{equation}\label{eq:2-dim convergence decr h}
\Prob\{X(t)\leq a+p(t),Y_2(t)\leq b\}   ~\to~   \Prob\{X^*_{\circ}\leq a\}\Prob\{X^*_{\circ}\leq b\},
\end{equation}
as $t \to \infty$ for all $a, b \in \R$ that are continuity points of the law of $X^*_{\circ}$.

We begin by proving \eqref{eq:Y_1(t)->0}.
Using integration by parts, we infer
\begin{eqnarray}
Y_1(t)
& = &
\Big(\sum_{k=0}^{N(m_1t)-1}h(ut-S_k) \1_{\{S_k\leq ut\}} - \frac{1}{\mu} \int_{(u-m_1)t}^{ut} h(y) \dy \Big)    \notag  \\
& = &
\int_{[0,\,m_1t]} h(ut-y) \, \mathrm{d} \Big(N(y)-\frac{y}{\mu}\Big)    \notag  \\
& = &
h(ut) + h((u-m_1)t-) \Big((N(m_1t)-\frac{m_1t}{\mu}\Big) - h(ut-)   \notag  \\
& &
- \int_{(0,\,m_1t]} \Big(N(y)-\frac{y}{\mu}\Big) \, \mathrm{d} h(ut-y)  \notag  \\
& = &
h(ut)-h(ut-) + h((u-m_1)t-)\Big(N(m_1t)-\frac{m_1t}{\mu}\Big)   \notag  \\
& & + \int_{[(u-m_1)t,\,ut)}\Big(N(ut-y)-\frac{ut-y}{\mu}\Big) \,
\mathrm{d}(\!-h(y))  \label{eq:Y_1(t)}
\end{eqnarray}
for large enough $t$. We now consider two situations:

\noindent {\sc Cases (C1) and (C3):} We invoke the estimate $\E
|N(t)-\frac{t}{\mu}| \leq K_1+K_2c(t)$ which holds for all $t \geq
0$ and some fixed $K_1,K_2>0$, see Proposition \ref{Prop:moment
convergence}. Here $c(t)$ is chosen as $\sqrt{t}$ in the case (C1)
and in the case (C3) it is as stated there. Then
\begin{eqnarray*}
\E|Y_1(t)|
& \leq & h((u-m_1)t) \E \Big|N(m_1t)-\frac{m_1t}{\mu}\Big| + o(1)       \\
& & + \int_{((u-m_1)t,\,ut]}\E\Big|N(ut-y)-\frac{ut-y}{\mu}\Big| \, \mathrm{d}(\!-h(y))         \\
& \leq &
h((u-m_1)t) \Big(K_1+K_2c(m_1t)\Big)+o(1)       \\
& &
+\int_{((u-m_1)t,\,ut]}(K_1+K_2c(ut-y)) \, \mathrm{d}(\!-h(y)).
\end{eqnarray*}
Note that since $c$ is regularly varying and in view of \eqref{eq:c(t) h(t) -> 0} we have, for arbitrary $\kappa, \lambda >0$,
\begin{equation}    \label{eq:c(t)h(t)->0}
\lim_{t\to\infty}{c(\kappa t)h(\lambda t)}  ~=~ \lim_{t \to \infty} c(t) h(t)   ~=~ 0,
\end{equation}
in the case (C3). The same relation holds in the case (C1) since
$h(t)=o(t^{-1/2})$ in view of \eqref{eq:int h^2 < infty} and the
monotonicity of $h$. Recalling that $m_1 = (1+u)/2 < u$ and using
\eqref{eq:c(t)h(t)->0} we infer that the first summand in the
estimate for $\E|Y_1(t)|$ above converges to zero as $t \to
\infty$. Further, again using \eqref{eq:c(t)h(t)->0},
\begin{align*}
\int_{((u-m_1)t,\,ut]} & (K_1+K_2c(ut-y)) \, \mathrm{d}(\!-h(y))        \\
&=~ K_2 \int_{((u-m_1)t,\,ut]}c(ut-y) \, \mathrm{d}(\!-h(y)) + o(1) \\
&\leq~  c(m_1t) (h((u-m_1)t)-h(ut)) + o(1)
~\to~ 0
\end{align*}
as $t \to \infty$.

\noindent {\sc Case (C2):} In this case $\E \xi^r < \infty$. Thus,
as a consequence of the Marcin\-kie\-wicz-Zygmund law of large numbers
for $N(t)$, see \eqref{eq:Marcinkiewicz-Zygmund SLLN of N(t)},
\begin{equation*}
\lim_{t \to \infty} \frac{\sup_{0 \leq s \leq
t}|N(s)-\mu^{-1}s|}{t^{1/r}}  ~=~ 0   \quad \text{a.s.}
\end{equation*}
From the monotonicity of $h$ and \eqref{eq:int h^r<infty} it
follows that $h(t) = o(t^{-1/r})$ as $t \to \infty$. Hence
\begin{equation*}   
\lim_{t \to \infty} h(\kappa t) \sup_{0 \leq s \leq \lambda t}
\Big|N(s)-\mu^{-1}s \Big| ~=~ 0   \quad   \text{a.s.}
\end{equation*}
for arbitrary $\kappa, \lambda > 0$.
Using this in \eqref{eq:Y_1(t)} implies that $Y_1(t) \to 0$ a.s., in particular also in probability.

We now turn to the proof of \eqref{eq:2-dim convergence decr h}.
Write the probability on the left-hand side of \eqref{eq:2-dim
convergence decr h} as follows:
\begin{align*}
\Prob\{X(t) & \leq a+p(t),Y_2(t) \leq b\}   \\
&=~ \int_{(m_1t,\,\infty)}\Prob\{X(t)\leq a+p(t),Y_2(t)\leq b ,S_{N(m_1t)}\in \dy\}     \\
&=~ \Big(\int_{(m_1t,\,z_2(t)]}\cdots
+\int_{(z_2(t),\,\infty)}\cdots\Big)=:J_1(t)+J_2(t).
\end{align*}
Clearly, $0\leq J_2(t)\leq \Prob\{S_{N(m_1t)}\geq z_2(t)\} = \Prob\{S_{N(m_1t)}-m_1t\geq r(t)\}$.
The latter probability tends to $0$ as $t \to \infty$ whenever
\begin{equation}    \label{eq:r(t)->0}
r(t) \to \infty \quad   \text{as } t \to \infty
\end{equation}
since $S_{N(m_1t)}-m_1t\stackrel{\mathrm{d}}{\to} S_0^*$.

For $J_1(t)$ we may write:
\begin{align*}
J_1&(t)     \\
& =
\int_{(m_1t,\,z_2(t)]} \!\!\! \Prob\{X(t)\leq a+p(t),S_{N(m_1t)} \in \dy,   \\
&
\hphantom{=} \sum_{k=N(m_1t)}^{\infty} \!\!\!  h(ut\!-\!y\!-\!(S_k\!-\!S_{N(m_1t)})) \1_{\{S_k-S_{N(m_1t)}\leq ut-y\}}\leq b + p((u\!-\!m_1)t) \}       \\
& =
\int_{(m_1t,\,z_2(t)]} \!\!\!\!\! \Prob\{X(ut\!-\!y)\leq b\!+\!p((u\!-\!m_1)t)\} \Prob\{X(t) \leq a\!+\!p(t),S_{N(m_1t)} \in \dy \},
\end{align*}
where the last equality follows from the inequality $t < m_1t$ and
the fact that $(S_{k+N(m_1t)}-S_{N(m_1t)})_{k \in \N_0}$ has the
same distribution as $(S_k)_{k \in \N_0}$ and is independent of
$(S_k)_{0 \leq k \leq N(m_1t)}$. Further,
\begin{align*}
J_1 & (t)
~=~ \Prob\{X^*_{\circ}\leq b\} \int_{(m_1t,\,z_2(t)]}\Prob\{X(t)\leq a+p(t),S_{N(m_1t)}\in\dy \}        \\
&\hphantom{=\}}~ + \int_{(m_1t,\,z_2(t)]} \!\!\! \big(\Prob\{X(ut-y)\leq b+p((u-m_1)t)\}-\Prob\{X^*_{\circ}\leq b\}\big)    \\
&\hphantom{=\}~ + \int_{(m_1t,\,z_2(t)]} \big(} \times \Prob\{X(t)\leq a+p(t),S_{N(m_1t)}\in\dy \}  \\
&=:~    J_{11}(t)+J_{12}(t).
\end{align*}
When \eqref{eq:r(t)->0} holds, then $\Prob\{m_1t < S_{N(m_1t)}\leq z_2(t)\} \to 1$ as $t\to\infty$.
Consequently, an application of Proposition \ref{Prop:decreasing h} shows that
the integral in the first summand converges to $\Prob\{X^*_{\circ} \leq  a\}$.
To show that $J_{12}(t)$ converges to zero, write
\begin{align*}
|J_{12} & (t)|
\leq
\sup_{m_1t < y \leq z_2(t)} \!\! \big|\Prob\{X(ut\!-\!y)\leq b+p(ut\!-\!m_1t)\}-\Prob\{X^*_{\circ}\leq b\}\big| \\
& \hphantom{(t)| \leq }~ \times \int_{(m_1t,\,z_2(t)]}\Prob\{S_{N(m_1t)}\in\dy\}    \\
&\leq~
\sup_{m_1t < y \leq z_2(t)} \!\! \big|\Prob\{X_{\circ}(ut\!-\!y)\leq b\!+\!p(ut\!-\!m_1t)\!-\!p(ut\!-\!y)\}-\Prob\{X^*_{\circ}\leq b\}\big| \\
&=~
\sup_{m_1t < y \leq z_2(t)} \!\! \big|\Prob\{X_{\circ}(ut\!-\!y)\leq b+\mu^{-1} \!\!\! \int_{ut\!-\!y}^{(u\!-\!m_1)t} \!\!\! h(x)\dx - \Prob\{X^*_{\circ}\leq b\}\big|.
\end{align*}
Using the fact that $h$ is eventually nonnegative, we proceed as
follows:
\begin{eqnarray*}
|J_{12}(t)|
& \leq &
\sup_{m_1t < y \leq z_2(t)}\Big|\Prob\{X_{\circ}(ut\!-\!y)\leq b\}-\Prob\{X^*_{\circ}\leq b\}\Big|  \\
& &
+\sup_{m_1t < y \leq z_2(t)} \Prob\left\{b < X_{\circ}(ut\!-\!y) \leq b+\mu^{-1}\int_{ut-y}^{ut-m_1t} h(y) \dy \right\} \\
& \leq &
\sup_{y\geq ut-z_2(t)}\Big|\Prob\{X_{\circ}(y)\leq b\}-\Prob\{X^*_{\circ}\leq b\}\Big|  \\
& &
+\sup_{m_1t < y \leq z_2(t)} \Prob\left\{b < X_{\circ}(ut\!-\!y) \leq b+\mu^{-1}\int_{ut-y}^{ut-m_1t} h(y) \dy \right\}.
\end{eqnarray*}
Due to Proposition \ref{Prop:decreasing h}, the first summand converges to zero whenever
\begin{equation}    \label{eq:u(t)>>z_2(t)}
ut - z_2(t) ~=~ (u-1)t/2 - r(t) ~\to~ \infty    \quad   \text{as } t \to \infty.
\end{equation}
Assume that $r$ also satisfies
\begin{equation}    \label{eq:3rd condition for r}
\lim_{t \to \infty} h((u-1)t/2-r(t)) r(t) ~=~ 0.
\end{equation}
Then, for arbitrary $\varepsilon>0$, there exists $t_0$ such that for all $t>t_0$
\begin{equation*}
0 ~\leq~ \mu^{-1}\int_{ut-z_2(t)}^{ut-m_1t} h(x) \dx ~\leq~ \mu^{-1} h((u-1)t/2-r(t)) r(t) ~<~ \varepsilon.
\end{equation*}
Thus, when \eqref{eq:u(t)>>z_2(t)} and \eqref{eq:3rd condition for r} hold
and $\varepsilon > 0$ is chosen such that $b+\varepsilon$ is a continuity point of the law of $X^*_{\circ}$
\begin{align*}
\sup_{m_1t < y \leq z_2(t)} & \Prob \left\{b < X_{\circ}(ut\!-\!y) \leq b+\mu^{-1}\int_{ut-z_2(t)}^{(u-m_1)t} h(x) \dx \right\} \\
& \leq~ \sup_{m_1t < y \leq z_2(t)}\Prob\{X_{\circ}(ut\!-\!y)\in (b,b+\varepsilon]\}    \\
& \underset{t\to\infty}{\to}~ \Prob\{X^*_{\circ}\in (b,b+\varepsilon]\}\underset{\varepsilon\to 0}{\to} 0,
\end{align*}
since $b$ and $b+\varepsilon$ are continuity points of the law of $X^*_{\circ}$.
We thus have proved that $\lim_{t \to \infty} J_{12}(t)= 0$
if $r$ satisfies the conditions \eqref{eq:r(t)->0}, \eqref{eq:u(t)>>z_2(t)} and \eqref{eq:3rd condition for r}.
A possible choice of $r$ such that the above conditions are satisfied is the following.
Let $\delta = (u-1)/4$. Then choose $r$ as
$r(t) = h(\delta t)^{-1/2} \wedge \delta t$, $t \geq 0$.
Then \eqref{eq:r(t)->0} holds since $h(t) \to 0$ as $t \to \infty$ and \eqref{eq:u(t)>>z_2(t)} holds since $r(t) \leq \delta t$.
Finally, \eqref{eq:3rd condition for r} holds since $h(t(u-1)/2-r(t)) r(t) \leq h(\delta t)^{1/2} \to 0$ as $t \to \infty$.
The proof of \eqref{eq:2-dim convergence decr h} is complete.
\end{proof}

\section{Proofs of the limit theorems with scaling} \label{sec:non-endogenous case}

In this section, we prove the results presented in Section \ref{subsec:non-endogenous thms}.

\subsection{Proof of Theorem \ref{Thm:non-endogenous & mu<infty}}

Whenever possible we treat all cases simultaneously.
To this end, put
\begin{equation*}
X_t(u)  ~:=~    \frac{X(ut)-\int_{0}^{ut} h(y) \dy}{g(t)h(t)},
\quad   t>0,\ u \geq 0.
\end{equation*}

\subsubsection*{Reduction to continuous and decreasing response function}   \label{subsec:reduction}

We simplify the proof of Theorem \ref{Thm:non-endogenous & mu<infty}
by showing that without loss of generality we can replace $h$ by a decreasing and continuous function $h^*$ on $\R_+$
satisfying $h^*(t) \sim h(t)$ as $t\to\infty$.
This follows an idea in \cite{Iksanov:2012}.
We thus need to construct a function $h^*$ as above and prove that for this function
\begin{equation}    \label{eq:f.d. convergence X*}
X^*_t(u) ~:=~ \frac{\int_{[0,\,ut]} h^*(ut-y) \, \mathrm{d}N(y)-\mu^{-1}\int_{0}^{ut} h^*(y) \, \dy}{g(t)h^*(t)}
~\underset{t \to \infty}{\stackrel{\mathrm{f.d.}}{\Rightarrow}}~ Y(u)
\end{equation}
where $Y(u) := W_{\alpha}(u)$, $u\geq 0$ if $\beta=0$, and
\begin{equation*}
Y(u) ~:=~ W_{\alpha}(u)u^{-\beta}+\beta \int_0^u(W_{\alpha}(u)-W_{\alpha}(y))(u-y)^{-\beta-1}\dy, \quad u\geq 0
\end{equation*}
if $\beta>0$.
Then, to ensure the convergence $X_t(u) \stackrel{\mathrm{f.d.}}{\Rightarrow} Y(u)$ as $t \to \infty$,
it suffices to check that, for any $u>0$,
\begin{equation}    \label{eq:|X_t*-X_t|->0 1st}
\frac{\int_{[0,\,ut]} (h(ut-y)-h^*(ut-y)) \, \mathrm{d}N(y)}{g(t)h(t)}
~\stackrel{\Prob}{\to}~ 0 \quad \text{as } t \to \infty,
\end{equation}
and
\begin{equation}    \label{eq:|X_t*-X_t|->0 2nd}
\frac{\int_0^{ut} (h(y)-h^*(y)) \, \dy}{g(t)h(t)} ~\to~ 0   \quad \text{as } t \to \infty.
\end{equation}

We begin with the construction of $h^*$.
By assumption, $h$ is eventually decreasing.
Hence, there exists an $a>0$ such that $h$ is decreasing on $[a,\infty)$.
Let $\widehat{h}$ be a bounded, right-continuous and
decreasing function such that $\widehat{h}(t) = h(t)$ for $t \geq a$.
Note that the so defined $\widehat{h}$ is non-negative.
The first observation is that replacing $h$ by $\widehat{h}$ in the
definition of $X(t)$ does not change the asymptotics.
Indeed, if $\widehat{X}$ denotes the shot noise process with the shots occurring at times $S_0, S_1, \ldots$
and response function $\widehat{h}$ instead of $h$,
then for any $u > 0$ and large enough $t$,
\begin{eqnarray*}
|X(ut)-\widehat{X}(ut)|
& = &
\Bigg| \sum_{k=0}^{N(ut)} h(ut-S_k) - \widehat{h}(ut-S_k)\Bigg| \\
& \leq &
\underset{y \in [0,\,a]}{\sup}\, |h(y)-\widehat{h}(y)| (N(ut)-N(ut-a))  \\
& \stackrel{\mathrm{d}}{\leq} &
\underset{y \in [0,\,a]}{\sup}\, |h(y)-\widehat{h}(y)| N(a)
\end{eqnarray*}
by the well-known distributional subadditivity of $N$.
The local boundedness of $h$ and $\widehat{h}$ ensures the finiteness of the
last supremum.
Since $\beta < 1/\alpha$, in all cases
we have $\lim_{t \to \infty} g(t) h(t) = \infty$. Consequently,
\begin{equation}    \label{eq:aux4}
\frac{X(ut) - \widehat{X}(ut)}{g(t)h(t)}
~\stackrel{\Prob}{\to}~ 0   \quad   \text{as } t \to \infty.
\end{equation}
Further, for $ut \geq a$,
\begin{eqnarray}
\left|\frac{\int_{0}^{ut} (h(y)-\widehat{h}(y)) \dy}{g(t)h(t)}\right|
& \leq &
\frac{\int_0^{ut} |h(y)-\widehat{h}(y)| \dy}{g(t)h(t)}  \notag  \\
& = &
\frac{\int_{[0,\,a]}\big|h(y)-\widehat{h}(y)\big|\dy}{g(t)h(t)} ~\to~   0
\quad \text{as } t\to\infty.        \label{eq:aux2}
\end{eqnarray}
Thus, in what follows, we can replace $h$ by $\widehat{h}$.
We will now construct $h^*$ from $\widehat{h}$. To this end,
let $\theta$ be a random variable with the standard exponential distribution.
Set
\begin{equation}    \label{eq:definition of h^*}
h^*(t) ~:=~ \E \widehat{h}\big((t-\theta)^+\big)
~=~ e^{-t}\bigg(\widehat{h}(0)+\int_0^t \widehat{h}(y)e^y\dy\bigg), \quad t \geq 0.
\end{equation}
It is clear that $\widehat{h}(t) \leq h^*(t)$, $t\geq 0$
and that $h^*$ is continuous and decreasing on $\R_+$ with
$h^*(0)=\widehat{h}(0)<\infty$.
Furthermore, $h^*(t) \sim \widehat{h}(t)\sim h(t)$, $t\to\infty$.
While this is trivial if $\lim_{t \to \infty} h(t)\neq 0$,
in the opposite case ($\lim_{t \to \infty} h(t)=0$) the first
equivalence does require a proof.
We use the second equality in \eqref{eq:definition of h^*}.
Being a regularly varying function $1/\widehat{h}$ grows subexponentially fast.
Using this and the regular variation of $\widehat{h}$ at infinity, we infer for any $\varepsilon \in (0,1)$:
\begin{eqnarray*}
\frac{h^*(t)}{\widehat{h}(t)}
& = &
\E \left[\frac{\widehat{h}\big((t-\theta)^+\big)}{\widehat{h}(t)}  \1_{\{\theta > \varepsilon t\}}\right]
+ \E \left[\frac{\widehat{h}\big((t-\theta)^+\big)}{\widehat{h}(t)} \1_{\{\theta \leq \varepsilon t\}}\right]   \\
& \leq &
\frac{\widehat{h}(0)}{\widehat{h}(t)} e^{-\varepsilon t}
+ \frac{\widehat{h}((1-\varepsilon)t)}{\widehat{h}(t)} (1-e^{-\varepsilon t})
~\underset{t \to \infty}{\to}~  (1-\varepsilon)^{-\beta}    ~\underset{\varepsilon \to 0}{\to}~ 1.
\end{eqnarray*}
Since $\widehat{h}(t) \leq h^*(t)$ for all $t \geq 0$, this implies $h^*(t) \sim \widehat{h}(t)$ as $t \to \infty$.
Let us now prove that
\begin{equation}    \label{eq:int h*-h^=h^(0)}
\lim_{t \to \infty} \int_0^t  (h^*(y)-\widehat{h}(y)) \, \dy ~=~ \widehat{h}(0).
\end{equation}
We use the representation
\begin{align*}
\int_0^t (h^*(y) & -\widehat{h}(y)) \, \dy  \\
& =~
\widehat{h}(0) (1 - e^{-t}) - \E \int_{t-\theta}^{t} \widehat{h}(y)\dy \1_{\{\theta\leq t\}}- \int_0^t \widehat{h}(y)\dy e^{-t}.
\end{align*}
Since $\widehat{h}$ grows subexponentially, the last term vanishes as $t \to \infty$,
and we are left with investigating the second term.
By the monotonicity of $\widehat{h}$ and the dominated convergence theorem,
\begin{equation*}
\E \bigg(\int_{t-\theta}^{t} \widehat{h}(y)\dy \1_{\{\theta \leq t\}}\bigg)
~\leq~ \E ( \theta\widehat{h}(t-\theta) \1_{\{\theta\leq t\}})
~\to~   0   \quad \text{as } t \to \infty,
\end{equation*}
which proves \eqref{eq:int h*-h^=h^(0)}.
In particular,
\begin{equation*}
\left|\frac{\int_{0}^{ut}(\widehat{h}(y)-h^*(y)) \dy}{g(t)h(t)}\right|
~=~
\frac{\int_{0}^{ut}(h^*(y)-\widehat{h}(y)) \dy}{g(t)h(t)}
~\to~ 0     \quad \text{as } t \to \infty
\end{equation*}
because $\lim_{t \to \infty} g(t)h(t)=\infty$.
In combination with \eqref{eq:aux2} the latter proves \eqref{eq:|X_t*-X_t|->0 2nd}.
Recalling \eqref{eq:int h*-h^=h^(0)} and the fact that in all cases $\lim_{t \to \infty} g(t)h(t) = \infty$,
we conclude from Lemma \ref{Lem:reduction of integrals} (with $f_1=h^*$ and $f_2=\widehat{h}$)
\begin{align*}
\Bigg|&\frac{\int_{[0,\,ut]} (\widehat{h}(ut-y)-h^*(ut-y)) \, \mathrm{d}N(y)}{g(t)h(t)}\Bigg|   \\
&\hspace{1cm} =~ \frac{\int_{[0,\,ut]} (h^*(ut-y)-\widehat{h}(ut-y)) \, \mathrm{d}N(y)}{g(t)h(t)}
~\stackrel{\mathcal{L}^1}{\to}~ 0   \quad   \text{as } t \to \infty.
\end{align*}
This together with \eqref{eq:aux4} leads to \eqref{eq:|X_t*-X_t|->0 1st}.
It remains to prove \eqref{eq:f.d. convergence X*}.

\subsubsection*{Proof of \eqref{eq:f.d. convergence X*}}

By the Cram\'er-Wold device and the discussion in Subsection
\ref{subsec:reduction}, in order to show finite-dimensional
convergence of $X_t(u)$, it suffices to prove that for any
$n\in\N$, $\gamma_1,\ldots, \gamma_n\in\R$ and $0<u_1<\ldots<u_n$
we have that
\begin{equation*}   
\sum_{k=1}^n \gamma_k X^*_t(u_k)
~\stackrel{\mathrm{d}}{\to}~
\sum_{k=1}^n \gamma_k Y(u_k)    \quad \text{as } t \to \infty.
\end{equation*}
Integrating by parts we obtain that
\begin{align}
\sum_{k=1}^n & \gamma_k X^*_t(u_k)  \notag  \\
&=~
\sum_{k=1}^n \frac{\gamma_k}{g(t) h^*(t)} \left( h^*(u_k t) + \int_{(0,u_kt]} h^*(u_kt-y) \, \mathrm{d}\Big(N(y)-\frac{y}{\mu} \Big) \right)    \notag  \\
&=~
\sum_{k=1}^n \frac{\gamma_k}{g(t) h^*(t)} \bigg( h^*(0) \Big(N(u_kt)-\frac{u_k t}{\mu}\Big) \notag  \\
&\hphantom{=~ \sum_{k=1}^n \frac{\gamma_k}{g(t) h^*(t)} \bigg(}
- \int_{(0,u_kt]} \!\! (N(y)- \mu^{-1}y) \, \mathrm{d}(h^*(u_kt-y))\bigg)   \notag  \\
&=~
\sum_{k=1}^n \gamma_k W_t(u_k) \frac{h^*(u_k t)}{h^*(t)}    \notag  \\
&\hphantom{=~}
+ \sum_{k=1}^n \frac{\gamma_k}{g(t) h^*(t)} \bigg(
(h^*(0)-h^*(u_kt)) \Big(N(u_kt)-\frac{u_k t}{\mu}\Big)  \notag  \\
&\hphantom{=~ + \sum_{k=1}^n \frac{\gamma_k}{g(t) h^*(t)} \bigg(}
 - \int_{[0,u_kt)} \!\! \Big(N(u_kt-y)- \frac{u_kt-y}{\mu}\Big) \, \mathrm{d}(-h^*(y))\bigg)    \notag  \\
&=~ \sum_{k=1}^n \gamma_k W_t(u_k) \frac{h^*(u_k t)}{h^*(t)}    \notag  \\
&\hphantom{=~}
+ \sum_{k=1}^n\gamma_k \frac{\int_{[0,\,u_k t)} (N(u_k t)-N(u_k t\!-\!y)-\mu^{-1}y)
\, \mathrm{d}(-h^*(y))}{g(t)h^*(t)}     \label{eq:q}
\end{align}
where the definition of $W_t(u)$ should be recalled from \eqref{eq:FLT for N(t)}.

\noindent {\sc Case $\beta=0$}.
Our aim is to show that each summand of the second term in \eqref{eq:q} converges to zero in
probability, for the convergence
\begin{equation*}
\sum_{k=1}^n \gamma_k X^*_t(u_k)
~\stackrel{\mathrm{d}}{\to}~    \sum_{k=1}^n \gamma_k Y(u_k)
~=~ \sum_{k=1}^n \gamma_k W_{\alpha}(u_k)
\end{equation*}
is then an immediate consequence of $\lim_{t \to \infty} h^*(u_kt)/h^*(t)=1$, \eqref{eq:FLT for N(t)} and Slutsky's theorem.

For the $k$th summand in \eqref{eq:q}, we have
\begin{align*}
\Bigg| & \frac{\int_{[0,\,u_k t)} (N(u_k t)-N(u_k t\!-\!y)-\mu^{-1}y)
\, \mathrm{d}(\!-h^*(y))}{g(t)h^*(t)} \Bigg|        \\
&   \leq~
\Bigg| \frac{\int_{[0,\,u_k t)} (\tilde{N}^*(u_k t)-\tilde{N}^*(u_k t\!-\!y)-\mu^{-1}y) \, \mathrm{d}(\!-h^*(y))}{g(t)h^*(t)} \Bigg|            \\
&    \hphantom{\leq~} +
\Bigg| \frac{\int_{[0,\,u_k t)} \big(N(u_kt)-\tilde{N}^*(u_k t)-(N(u_k t\!-\!y)-\tilde{N}^*(u_k t\!-\!y))\big) \mathrm{d}(\!-h^*(y))}{g(t)h^*(t)} \Bigg|.            \\
\end{align*}
Both terms tend to $0$ in probability.
For the second, this follows from \eqref{eq:N leq N*+O(1)} and \eqref{eq:N geq N*+O(1)} with $t$ replaced by $u_k t$, Markov's inequality
and the fact that $g(t) h^*(t) \to \infty$ as $t \to \infty$.
The latter fact together with Markov's inequality also shows
that for the first term to converge to $0$ in probability it is sufficient to check that
\begin{equation*}
\lim_{t \to \infty} \frac{\int_{[0,\,ut]}
\E |N^*(y)-\mu^{-1}y| \, \mathrm{d}(-h^*(y))}{g(t)h^*(t)}   ~=~ 0.
\end{equation*}
By Proposition \ref{Prop:moment convergence}, $\E |N^*(y)-\mu^{-1}y| = O(g(y))$ as $y \to \infty$.
Consequently, it is enough to show that
\begin{equation*}
\lim_{t \to \infty} \frac{\int_{[0,\,ut]} g(y) \, \mathrm{d}(-h^*(y))}{g(t)h^*(t)}  ~=~ 0.
\end{equation*}
Since the function $g(t)h^*(t)$ is regularly varying, the latter is equivalent to
\begin{equation}    \label{eq:sufficient}
\lim_{t \to \infty} \frac{\int_{[0,\,t]}g(y) \, \mathrm{d}(-h^*(y))}{g(t)h^*(t)}    ~=~ 0.
\end{equation}
Using \eqref{eq:Potter's bound for g} gives
\begin{equation*}
\frac{\int_{[t_0,\,t]}g(y) \, \mathrm{d}(-h^*(y))}{g(t)h^*(t)}
~\leq~
A \frac{\int_{[t_0,\, t]} y^{1/\alpha-\delta} \, \mathrm{d}(-h^*(y))}{t^{1/\alpha-\delta}h^*(t)}
\end{equation*}
for $t\geq t_0$, and, as $t \to \infty$, the last ratio tends to zero
by Theorem 1.6.4 in \cite{Bingham+Goldie+Teugels:1989}.
Further, since $g(t) h^*(t) \to 0$, also
\begin{equation*}
\lim_{t \to \infty} \frac{\int_{[0,\,t_0]} g(y) \, \mathrm{d}(-h^*(y))}{g(t)h^*(t)} ~=~ 0.
\end{equation*}
Thus, \eqref{eq:sufficient} follows.

\noindent {\sc Case $\beta>0$}. For any $\rho \in (0,1)$, one can write
\begin{eqnarray*}
X_t^*(u_k)
& = &
W_t(u_k)\frac{h^*(u_kt)}{h^*(t)}
+ \int_{(0,\,u_k]}(W_t(u_k)-W_t(v)) \, \nu^*_{t,\,k}(\dv)   \\
& = &
W_t(u_k)\frac{h^*(u_kt)}{h^*(t)}+\int_{(0,\,\rho u_k]} \ldots + \int_{(\rho u_k,\,u_k]} \ldots,
\end{eqnarray*}
where
$\nu^*_{t,\,k}$ is the finite measure on $[0,u_k]$ defined by
\begin{equation*}
\nu^*_{t,\,k}(a,b]  ~:=~    \frac{h^*(t(u_k-b))-h^*(t(u_k-a))}{h^*(t)},
\quad 0 \leq a < b \leq u_k.
\end{equation*}
In view of \eqref{eq:FLT for N(t)} and the continuous mapping theorem,
\begin{equation*}
W_t(u_k)-W_t(v) ~\Rightarrow~   W_{\alpha}(u_k)-W_{\alpha}(v) \quad   \text{as } t \to \infty.
\end{equation*}
Further, by the regular variation of $h^*$, the finite measures $\nu^*_{t,\,k}$ converge weakly on $[0, \rho u_k]$
to a finite measure $\nu^*_k$ on $[0,\rho u_k]$ which is defined by
$\nu^*_k(a,b] = (u_k-b)^{-\beta}-(u_k-a)^{-\beta}$.
Clearly, the limiting measure is absolutely continuous with density
$x \mapsto \beta(u_k-x)^{-\beta-1}$, $x\in [0,\rho u_k]$.
Hence, by Lemma \ref{Lem:int X_t dnu_t->int X dnu},
\begin{multline*}
W_t(u_k) \frac{h(u_kt)}{h(t)}
+ \int_{(0,\,\rho u_k]} (W_t(u_k)-W_t(v)) \, \nu^*_{t,\,k}(\dv) \\
~\stackrel{\mathrm{d}}{\to}~
W_{\alpha}(u_k)u_k^{-\beta}+\beta \int_{0}^{\rho u_k}(W_{\alpha}(u_k)-W_{\alpha}(v))(u_k-v)^{-\beta-1} \, \dv.
\end{multline*}
Likewise, again by the continuous mapping theorem,
\begin{multline*}
\sum_{k=1}^n\gamma_k W_t(u_k)\frac{h(u_kt)}{h(t)}
+ \sum_{k=1}^n\gamma_k \int_{(0,\,\rho u_k]}(W_t(u_k)-W_t(v)) \, \nu^*_{t,\,k}(\dv) \\
~\stackrel{\mathrm{d}}{\to}~
\sum_{k=1}^n\gamma_k W_{\alpha}(u_k)u_k^{-\beta}
+ \sum_{k=1}^n\gamma_k \beta \int_{0}^{\rho u_k} (W_{\alpha}(u_k)-W_{\alpha}(v))(u_k-v)^{-\beta-1} \, \dv.
\end{multline*}
According to Theorem 3.2 in \cite{Billingsley:1999},
it remains to check that, as $\rho \uparrow 1$,
\begin{multline*}
\sum_{k=1}^n\gamma_k W_{\alpha}(u_k)u_k^{-\beta}
+ \sum_{k=1}^n\gamma_k\beta \int_0^{\rho u_k} (W_{\alpha}(u_k)-W_{\alpha}(v))(u_k-v)^{-\beta-1} \, \dv    \\
~\stackrel{\mathrm{d}}{\to}~
\sum_{k=1}^n \gamma_k W_{\alpha}(u_k)u_k^{-\beta}
+ \sum_{k=1}^n \gamma_k \beta \int_0^{u_k} (W_{\alpha}(u_k)-W_{\alpha}(v))(u_k-v)^{-\beta-1} \, \dv
\end{multline*}
and that, for any $c>0$,
\begin{equation}    \label{eq:Billingsley 2}
\lim_{\rho\uparrow 1} \limsup_{t\to\infty}
\Prob\bigg\{\bigg|\sum_{k=1}^n\gamma_k \int_{[\rho u_k,\,u_k]} (W_t(u_k)-W_t(v)) \, \nu^*_{t,\,k}(\dv)\bigg|>c\bigg\} ~=~ 0.
\end{equation}
The first relation is equivalent to
\begin{equation}    \label{eq:int W_alpha(y) y^-beta-1 dy}
\sum_{k=1}^n \gamma_k \beta \int_{\rho u_k}^{u_k} (W_{\alpha}(u_k)-W_{\alpha}(v))(u_k-v)^{-\beta-1} \dv
~\stackrel{\Prob}{\to}~ 0   \quad \text{as } \rho \uparrow 1.
\end{equation}
To prove \eqref{eq:int W_alpha(y) y^-beta-1 dy} it suffices to verify that each summand
converges to zero in probability.
But each summand actually tends to zero a.s.\ by the discussion in Subsection \ref{subsec:Properties of limits 1<alpha<=2}.

The sum in \eqref{eq:Billingsley 2} equals
\begin{equation*}
\sum_{k=1}^n\gamma_k \frac{\int_{[0,\,(1-\rho)u_kt]}(N(u_k t)-N(u_kt-y)-\mu^{-1}y)\mathrm{d}(-h^*(y))}{g(t)h^*(t)}.
\end{equation*}
In view of \eqref{eq:N leq N*+O(1)}, \eqref{eq:N geq N*+O(1)},
Proposition \ref{Prop:moment convergence} and Markov's inequality it suffices to check that, for $k=1,\ldots,n$,
\begin{equation*}
\lim_{\rho \uparrow 1} \limsup_{t \to \infty} \frac{\int_{[0,\,(1-\rho) u_k t]} g(y) \mathrm{d}(-h^*(y))}{g(t)h^*(t)} ~=~ 0
\end{equation*}
or just
\begin{equation}    \label{eq:at u_k=1}
\lim_{\rho \uparrow 1} \limsup_{t \to \infty} \frac{\int_{[0,\,(1-\rho)t]}g(y) \mathrm{d}(-h^*(y))}{g(t)h^*(t)} ~=~ 0,
\end{equation}
for $g(t)h^*(t)$ is regularly varying.
Since $g(t) h^*(t) \to \infty$ as $t \to \infty$ by the assumptions of the theorem, we have
\begin{equation*}
\lim_{t \to \infty} \frac{\int_{[0,\,t_0]} g(y) \, \mathrm{d}(-h^*(y))}{g(t)h^*(t)} ~=~ 0
\end{equation*}
and, further,
\begin{eqnarray*}
\frac{\int_{[t_0,\,(1-\rho)t]}g(y) \, \mathrm{d}(-h^*(y))}{g(t)h^*(t)}
& \stackrel{\eqref{eq:Potter's bound for g}}{\leq} &
A \frac{\int_{[t_0,\,(1-\rho)t]}y^{1/\alpha-\delta} \, \mathrm{d}(-h^*(y))}{t^{1/\alpha-\delta}h^*(t)}  \\
& \sim &
\frac{\beta}{1/\alpha-\beta-\delta}(1-\rho)^{1/\alpha-\beta-\delta},
\end{eqnarray*}
where the last relation is justified by Theorem 1.6.4 in \cite{Bingham+Goldie+Teugels:1989},
\eqref{eq:at u_k=1} follows.
To complete the argument, we need to prove Proposition \ref{Prop:moment convergence}.

\subsection{Moment convergence when $\mu < \infty$} \label{subsec:moment convergence}

In this subsection we prove Proposition \ref{Prop:moment
convergence}. Keeping in mind the weak convergence relation
\eqref{eq:FLT for N(t)} with $u=1$, uniform integrability of the
family $|N(t)-t/\mu|/c(t)$, $t \geq 1$ would suffice to conclude
convergence of the first absolute moments. However, it seems such
an argument only works in the case (A1), see \cite[Theorem
3.8.4(i)]{Gut:2009}. This is why we follow a different approach.
In particular, we offer a new proof for the case (A1), for it
requires no extra work in the given framework.

\begin{proof}[Proof of Proposition \ref{Prop:moment convergence}]
Our purpose is to show that
\begin{equation}    \label{eq:moment convergence}
\lim_{t \to \infty} \frac{\E |N(t)-\mu^{-1}t|}{g(t)}    ~=~ \E |W_{\alpha}(1)|.
\end{equation}
where $g$ is defined as in the context of \eqref{eq:FLT for N(t)}.
We start with the representation
\begin{eqnarray*}
\E |S_{N(\mu n)}-S_n|
& = &
\E(S_{N(\mu n)\vee n}-S_{N(\mu n) \wedge n})    \\
& = &
\mu\E\big((N(\mu n)\vee n) - (N(\mu n)\wedge n) \big) ~=~ \mu \E \big|N(\mu n)-n|,
\end{eqnarray*}
where the second equality follows from Wald's identity.
We thus infer
\begin{eqnarray}
\E \big( |S_n-\mu n| - (S_{N(\mu n)} - \mu n) \big)
& \leq &
\mu \E |N(\mu n)-n| \notag  \\
& \leq &
\E \big(|S_n-\mu n| +(S_{N(\mu n)}-\mu n)\big). \label{eq:bounds on E|N(mun)-n|}
\end{eqnarray}
It is known that $c(n)^{-1}(S_n-n\mu) \stackrel{\mathrm{d}}{\to} -W_{\alpha}(1)$
and that (\textit{cf.}~\cite[Lemma 5.2.2]{Ibragimov+Linnik:1971})
\begin{equation*}
\sup_{n \in \N} \E \bigg(\frac{\big|S_n-\mu n\big|}{c(n)}\bigg)^{1+\delta}  ~<~ \infty
\end{equation*}
for some $\delta>0$.
Consequently,
\begin{equation}    \label{eq:E|S_n-mun|/c(n)->E|stable|}
\lim_{n \to \infty} \frac{\E \big|S_n-\mu n\big|}{c(n)} ~=~ \E |W_{\alpha}(1)|.
\end{equation}
If $F$ is non-lattice, then from \cite{Mohan:1976} it is known that
\begin{equation}    \label{eq:S_N(t)-t asymptotics}
\E(S_{N(t)}-t)  ~\sim~  \begin{cases}
                        \mathrm{const}                              &   \text{in the case (A1),}        \\
                        \mathrm{const} \cdot \ell(t)                    &   \text{in the case (A2),}        \\
                        \mathrm{const} \cdot t^{2-\alpha} \ell(t)       &   \text{in the case (A3),}
                        \end{cases}
\end{equation}
as $t \to \infty$.
Similar asymptotics hold in the lattice case.
In fact, when $F$ is lattice with span $d>0$, then,
in the case (A1),
\begin{equation*}
\lim_{n \to \infty} \E(S_{N(nd)}-nd)    ~\to~   \mathrm{const}
\end{equation*}
by Theorem 9 in \cite{Feller:1949}.
Hence
\begin{equation*}
\E (S_{N(t)}-t) ~=~ O(1) \quad \text{as } t \to \infty.
\end{equation*}
In the cases (A2) and (A3), according to Theorem 6 in \cite{Sgibnev:1981},
$\E(S_{N(t)}-t)$ exhibits the same asymptotics as in the non-lattice case.

Recalling that $c(t)$ is regularly varying at $\infty$ with index
$1/\alpha$ (where $\alpha=2$ in the Cases (A1) and (A2)), we conclude that
\begin{equation*}
\lim_{n \to \infty} \frac{\E (S_{N(\mu n)}-\mu n)}{c(n)} ~=~ 0.
\end{equation*}
Applying this and \eqref{eq:E|S_n-mun|/c(n)->E|stable|} to \eqref{eq:bounds on E|N(mun)-n|}
we infer
\begin{equation*}
\lim_{n \to \infty} \mu \frac{\E |N(\mu n)-n|}{c(n)}    ~=~ \E |W_{\alpha}(1)|.
\end{equation*}

Now we have to check that this relation implies \eqref{eq:moment convergence}.
For any $t>0$ there exists $n=n(t) \in \N_0$ such that $t \in (\mu n,\mu(n+1)]$.
Hence, by subadditivity,
\begin{equation*}
\E \big(N(t)-N(\mu n)\big)  ~\leq~  \E \big(N(\mu(n+1))-N(\mu n)\big)   ~\leq~  \E N(\mu).
\end{equation*}
It remains to observe that, as a consequence of the regular variation of $c(t)$,
we have $\lim_{t \to \infty} c(\mu n(t)\mu^{-1})/c(t) =\mu^{-1/\alpha}$,
hence \eqref{eq:moment convergence}.

The formula for $\E |W_{\alpha}(1)|$ in the case (A3) is proved in Lemma \ref{Lem:mean}.

It remains to check that $\E |N^*(t)-\mu^{-1}t|$ exhibits the same asymptotics as $\E|N(t)-\mu^{-1}t|$.
This is a consequence of the chain of equalities
\begin{equation*}
\E |N^*(t)-N(t)| ~=~ \E(N(t)-N^*(t)) ~=~ \mu^{-1} (\E S_{N(t)} - t) ~=~ o(c(t)) \text{ as } t \to \infty
\end{equation*}
where the second equality follows from Wald's equation and the third is a consequence of \eqref{eq:S_N(t)-t asymptotics}.
\end{proof}

\subsection{Proof of Theorem \ref{Thm:non-endogenous & mu=infty}}
For $t>0$, put
\begin{equation*}
X_t(u)  ~:=~    \frac{X(ut)}{g(t)h(t)}  ~=~ \frac{\int_{[0,\,ut]} h(ut-x) \, \mathrm{d}N(x)}{g(t)h(t)},
\quad   u \geq 0.
\end{equation*}
For any $n\in\N$, fix $\gamma_1,\ldots,
\gamma_n\in\R$ and $0<u_1<\ldots<u_n$. We have to show that
\begin{equation*}
\sum_{k=1}^n \gamma_k X_t(u_k)
~\stackrel{\mathrm{d}}{\to}~    \sum_{k=1}^n \gamma_k Y(u_k)    \quad \text{as } t \to \infty,
\end{equation*}
where $Y(u):=\int_{[0,\,u]}(u-y)^{-\beta}\mathrm{d}W_{\alpha}(y)$, $u>0$.

Since the convergence $\lim_{t \to \infty} h(ut)/h(t) =
u^{-\beta}$ is uniform on compact subsets of $(0,\infty)$
\cite[Theorem 1.2.1]{Bingham+Goldie+Teugels:1989}, and $(W_{\alpha}(u))_{u
\geq 0}$ has continuous paths a.s., the relation \eqref{eq:FLT for
N(t) when mu=infty} and Lemma \ref{Lem:int X_t dnu_t->int X dnu}
entail
\begin{equation*}
\int_{[0,\,\rho u_k]}\frac{h(t(u_k-y))}{h(t)} \, \mathrm{d} \frac{N(ty)}{g(t)}
~\stackrel{\mathrm{d}}{\to}~ \int_{[0,\,\rho u_k]}(u_k-y)^{-\beta} \, \mathrm{d}W_{\alpha}(y)
\quad   \text{as } t \to \infty
\end{equation*}
for any $\rho\in (0,1)$ where here and hereafter integration
w.r.t.\ $\mathrm{d} \frac{N(ty)}{g(t)}$ means integration w.r.t.\
$\nu_t(\dy)$ where the measure $\nu_t$ is defined by $\nu_t(A) =
g(t)^{-1} N(tA)$, $A \subset \R_+$ Borel. By the continuous
mapping theorem,
\begin{equation*}
\sum_{k=1}^n\gamma_k \int_{[0,\,\rho u_k]} \frac{h\big(t(u_k-y)\big)}{h(t)} \, \mathrm{d} \frac{N(ty)}{g(t)}
~\stackrel{\mathrm{d}}{\to}~    \sum_{k=1}^n\gamma_k \int_{[0,\,\rho u_k]}(u_k-y)^{-\beta} \mathrm{d}W_{\alpha}(y)
\end{equation*}
as $t \to \infty$.
According to Theorem 3.2 in \cite{Billingsley:1999},
it remains to check that, as $\rho \uparrow 1$,
\begin{equation*}
\sum_{k=1}^n \gamma_k \int_{[0,\,\rho u_k]}(u_k-y)^{-\beta} \, \mathrm{d}W_{\alpha}(y)
~\stackrel{\mathrm{d}}{\to}~ \sum_{k=1}^n \gamma_k \int_{[0,\,u_k]} (u_k-y)^{-\beta} \mathrm{d}W_{\alpha}(y)
\end{equation*}
and that, for any $c>0$,
\begin{equation}    \label{eq:Billingsley's 2nd cond}
\lim_{\rho \uparrow 1} \limsup_{t \to \infty}
\Prob\bigg\{\bigg|\sum_{k=1}^n\gamma_k\int_{[\rho u_k,\,u_k]} \frac{h\big(t(u_k-y)\big)}{h(t)} \mathrm{d} \frac{N(ty)}{g(t)}\bigg|>c\bigg\} ~=~ 0.
\end{equation}
The first relation is equivalent to
\begin{equation}    \label{eq:equivalent to Billingsley's 1st cond}
\sum_{k=1}^n\gamma_k \int_{[\rho u_k,\, u_k]}(u_k-y)^{-\beta} \, \mathrm{d} W_{\alpha}(y)
~\stackrel{\Prob}{\to}~ 0   \quad \text{as } \rho \uparrow 1.
\end{equation}
To prove \eqref{eq:equivalent to Billingsley's 1st cond} it suffices to verify that each summand
converges to zero in probability. Hence \eqref{eq:equivalent to Billingsley's 1st cond} is a direct
consequence of Lemma \ref{Lem:check}.

For \eqref{eq:Billingsley's 2nd cond}, in view of Markov's inequality it suffices to check that
\begin{equation*}
\underset{\rho\uparrow 1}{\lim}\,\underset{t\to\infty}{\lim\sup}\,\frac{\int_{[\rho u_k t,\,u_k t]}h(u_kt-y) \, \mathrm{d}\E N(y)}{h(t)g(t)} ~=~ 0.
\end{equation*}
Recalling that $h(t)g(t)$ is regularly varying the latter holds true by Lemma \ref{Lem:convergence of 1st moment at t=1} below.
This completes the proof of finite-dimensional convergence.

Convergence of moments follows from Lemma \ref{Lem:convergence of moments} below.
The proof is complete.

\subsection{Moment convergence when $\mu=\infty$}

\begin{Lemma}   \label{Lem:mu=infty EX(t)}
Assume that $\mu=\infty$ and let $h:\R_+\to\R_+$ be a measurable
and locally bounded function such that
\begin{equation*}
\lim_{t \to \infty} \frac{h(t)}{\Prob\{\xi>t\}} ~=~ c \in [0,\infty].
\end{equation*}
Then
\begin{equation*}
\lim_{t \to \infty} \E X(t) = c.
\end{equation*}
In particular, $X(t) \stackrel{\Prob}{\to} 0$ as $t \to \infty$ if $c=0$.
\end{Lemma}
\begin{proof}
Denote by $Z(t):=t-S_{N(t)-1}$ the undershoot at time $t$ and put
\begin{equation*}
f(t) := \frac{h(t)}{\Prob\{\xi>t\}},    \quad   t \geq 0.
\end{equation*}
The proof is based on the representation
\begin{equation*}   
\E X(t) ~=~ \int_{[0,\,t]} h(t-x) \,U(\dx)  ~=~ \E f(Z(t)), \quad   t \geq 0.
\end{equation*}
Under the sole assumption $\mu=\infty$, the renewal theorem gives $Z(t) \to \infty$ in probability.
Hence $f(Z(t)) \to c$ in probability.
If $c<\infty$, the function $f$ is bounded, and $\lim_{t \to \infty} \E f(Z(t))=c$
by the dominated convergence theorem.
If $c=\infty$, we obtain $\lim_{t \to \infty} \E f(Z(t))=c=\infty$ by Fatou's lemma.

The last assertion of the lemma follows from Markov's inequality.
\end{proof}

Now we are ready to prove Proposition \ref{Prop:exponential limit}:

\begin{proof}[Proof of Proposition \ref{Prop:exponential limit}]
Since the exponential law is uniquely determined by its moments,
the second assertion of the proposition is an immediate consequence of the first.

To prove convergence of moments, we use induction on $k$, the order of the moments.
The case $k=0$ is trivial, the case $k=1$ follows from Lemma \ref{Lem:mu=infty EX(t)}.
Assuming that
\begin{equation*}
\lim_{t \to \infty} \E X^j(t) = c^j j!  \quad \text{for } j=0,\ldots, k-1,
\end{equation*}
we will prove that
\begin{equation}    \label{eq:EX(t)^k->c^kk!}
\lim_{t \to \infty} \E X^k(t)=c^k k!.
\end{equation}
To this end, we use the representation
\begin{equation}    \label{eq:X(t)=h(t)+X_*(t-xi)}
X(t)    ~=~ h(t)+X_*(t-\xi_1) \1_{\{\xi_1\leq t\}},
\end{equation}
where
\begin{equation*}
X_*(t)  ~:=~    \sum_{j\geq 1}h(t-S_j+S_1) \1_{\{S_j-S_1\leq t\}}   ~\stackrel{\mathrm{d}}{=}~ X(t).
\end{equation*}
The latter implies
\begin{equation*}
X(t)^k  ~=~ X_*(t-\xi_1)^k \1_{\{\xi_1\leq t\}} + \sum_{j=0}^{k-1}{k \choose j} h(t)^{k-j} X_*(t-\xi_1)^j \1_{\{\xi_1\leq t\}}.
\end{equation*}
We have $\E X(t)^k = \E f_k(Z(t))$ where
\begin{equation*}
f_k(t) := \frac{\sum_{j=0}^{k-1}{k \choose j} h(t)^{k-j} \E X_*(t-\xi_1)^j \1_{\{\xi_1\leq t\}}}{\Prob\{\xi>t\}}.
\end{equation*}
Arguing as in Lemma \ref{Lem:mu=infty EX(t)}, it suffices to show that
\begin{equation*}
\lim_{t \to \infty} f_k(t) = c^k k!,
\end{equation*}
as $f_k(t)$ is then bounded and relation \eqref{eq:EX(t)^k->c^kk!} follows by the dominated convergence theorem.

Since $\lim_{t \to \infty} h(t)=0$ by the assumption of the proposition,
and the expectations $\E X_*(t-\xi_1)^j \1_{\{\xi_1\leq t\}}$, $j =0 , \ldots, k-1$
are bounded by the induction hypothesis, we conclude
\begin{equation*}
\frac{\sum_{j=0}^{k-2}{k \choose j}h^{k-j}(t)\E X_*^j(t-\xi_1) \1_{\{\xi_1\leq t\}}}{\Prob\{\xi>t\}}    ~=~ 0.
\end{equation*}
Hence
\begin{eqnarray*}
\lim_{t \to \infty} f_k(t)
& = &
k \lim_{t \to \infty} \frac{h(t)\E X_*(t-\xi_1)^{k-1} \1_{\{\xi_1\leq t\}}}{\Prob\{\xi>t\}} \\
& = &
ck \lim_{t \to \infty} \E \big(X(t)-h(t)\big)^{k-1} \\
& = &
ck \lim_{t \to \infty}\bigg(\E X(t)^{k-1}+\sum_{j=0}^{k-2}{k-1\choose j} h(t)^{k-1-j} \E X(t)^j \bigg)  \\
& = &
ck \lim_{t \to \infty} \E X(t)^{k-1}
~=~ c^k k!,
\end{eqnarray*}
where the penultimate equality follows from the induction hypothesis and $\lim_{t \to \infty} h(t)=0$.
\end{proof}

\begin{Lemma}   \label{Lem:convergence of 1st moment at t=1}
Assume that the assumptions of Theorem \ref{Thm:non-endogenous & mu=infty} hold.
Then
\begin{equation}    \label{eq:double limit}
\lim_{\rho\uparrow 1} \limsup_{t\to\infty}
\frac{\Prob\{\xi>t\}}{h(t)} \int_{[\rho t,\,t]} h(t-y) \, U(\dy) ~=~ 0.
\end{equation}
In particular,
\begin{equation*}
\lim_{t \to \infty} \frac{\Prob\{\xi>t\}}{h(t)} \E X(t)
~=~ \E \int_{[0,\,1]}(1-y)^{-\beta}\mathrm{d}W_{\alpha}(y)
~=~ \frac{\Gamma(1-\beta)}{\Gamma(1-\alpha)\Gamma(1+\alpha-\beta)}.
\end{equation*}
\end{Lemma}
\begin{proof}
We use the notation of Lemma \ref{Lem:mu=infty EX(t)}, that is,
$Z(t) := t-S_{N(t)-1}$ denotes the undershoot at time $t$ and
$f(t) := h(t)/\Prob\{\xi>t\}$, $t \geq 0$.
Then, the expression under the double limit in \eqref{eq:double limit}
equals $\E f(Z(t)) \1_{\{Z(t) \leq (1-\rho)t\}}/f(t)$.

{\sc Case 1:}
We first consider the case when $\alpha>\beta$ or $\alpha=\beta$ and $c=\infty$.

\noindent
If $\alpha>\beta$ then, by Theorem 1.5.3 in
\cite{Bingham+Goldie+Teugels:1989} there exists an {\it increasing} function $u$ such
that $u(t) \sim f(t)$ as $t \to \infty$.
If $\alpha=\beta$ and $c=\infty$ such a function $u$ exists by assumption.
Now fix $\varepsilon > 0$ and let $t_0 > 0$ be such that $(1-\varepsilon) u(t) \leq f(t) \leq (1+\varepsilon)u(t)$ for all $t \geq t_0$.
Then
\begin{equation*}
\frac{\E f(Z(t)) \1_{\{Z(t)\leq t_0\}}}{f(t)}
~\leq~ \frac{\sup_{0 \leq y \leq \in t_0} f(y)}{f(t)}
~\to~   0   \quad   \text{as } t \to \infty
\end{equation*}
by the local boundedness of $f$. Further, for $t$ such that $(1-\rho)t>t_0$,
\begin{eqnarray*}
\frac{\E f(Z(t))\1_{\{t_0<Z(t)\leq (1-\rho)t\}}}{f(t)}
& \leq &
\frac{1+\varepsilon}{1-\varepsilon} \frac{\E u(Z(t))\1_{\{Z(t) \leq (1-\rho)t\}}}{u(t)}     \\
& \leq &
\frac{1+\varepsilon}{1-\varepsilon}\Prob{\{Z(t)\leq (1-\rho)t\}}.
\end{eqnarray*}
By a well-known result due to Dynkin (see, for instance, Theorem 8.6.3 in \cite{Bingham+Goldie+Teugels:1989})
\begin{equation*}
\lim_{t \to \infty} \Prob{\{Z(t)\leq (1-\rho)t\}}   ~=~ \frac{1}{\Gamma(\alpha)\Gamma(1-\alpha)} \int_{0}^{1-\rho} y^{-\alpha}(1-y)^{\alpha-1} \dy
\end{equation*}
When $\rho \uparrow 1$ the last integral goes to zero, which
proves \eqref{eq:double limit}.

{\sc Case 2:}
Now consider the case when $\alpha=\beta$ and $c<\infty$.
Then $f$ is bounded. Hence $\E f(Z(t)) \1_{\{Z(t)\leq (1-\rho)t\}} \leq \mathrm{const} \cdot \Prob\{Z(t)\leq (1-\rho)t\}$, $t \geq 0$.
The rest of the proof is the same as in the previous case.

Turning to the second assertion of the lemma, we observe that
\begin{align*}
\frac{\Prob\{\xi>t\}}{h(t)} \int_{[0,\rho t]} h(t-y) \, U(\dy)
~=~ \Prob\{\xi > t\} U(t) \int_{[0,\rho]} \!\! \frac{h(t(1-y))}{h(t)} \, U_t(\dy)
\end{align*}
where $U_t([0,x]) = U(tx)/U(t)$, $0 \leq x \leq 1$. Formula
(8.6.4) on p.~361 in \cite{Bingham+Goldie+Teugels:1989} says that
$\lim_{t \to \infty} \Prob\{\xi>t\} U(t) =
\Gamma(1\!-\!\alpha)^{-1} \Gamma(1\!+\!\alpha)^{-1}$. Hence, the
measures $U_t(\dx)$ converge weakly to $\alpha x^{\alpha-1} \dx$
as $t \to \infty$. This in combination with the uniform
convergence theorem \cite[Theorem
1.2.1]{Bingham+Goldie+Teugels:1989} yields
\begin{align*}
\lim_{t \to \infty} & \Prob\{\xi > t\} U(t) \int_{[0,\rho]} \!\! \frac{h(t(1-y))}{h(t)} \, U_t(\dy) \\
& =~
\frac{\alpha}{\Gamma(1\!-\!\alpha)\Gamma(1\!+\!\alpha)} \int_0^{\rho} (1-y)^{-\beta} y^{\alpha-1} \dy   \\
&\underset{\rho \to 1}{\to}~
\frac{\Gamma(1-\beta)}{\Gamma(1+\alpha-\beta) \Gamma(1\!-\!\alpha)}.
\end{align*}
An appeal to \eqref{eq:double limit} proves the second assertion
of the lemma.
\end{proof}

\begin{Lemma}   \label{Lem:convergence of moments}
Under the assumptions of Theorem \ref{Thm:non-endogenous & mu=infty}, \eqref{eq:convergence of moments} holds.
\end{Lemma}
\begin{proof}
{\sc Case 1:} either $\alpha>\beta$ or $\alpha=\beta$ and
$c=\infty$.

\noindent
Set $g(t) = 1/\Prob\{\xi>t\}$ and observe that $\lim_{t \to \infty} g(t)h(t) = \infty$ in the present situation.
We prove the result only for $u=1$.
Define
\begin{equation*}
d_k ~:=~ \frac{k!}{\Gamma(1-\alpha)^k} \prod_{j=1}^k \frac{\Gamma(1-\beta+(j-1)(\alpha-\beta))}{\Gamma(j(\alpha-\beta)+1)},
\quad k \in \N.
\end{equation*}
We then have to show that
\begin{equation}    \label{eq:convergence of jth moment at t=1}
\lim_{t \to \infty} \frac{\E X(t)^j}{g(t)^j h(t)^j} ~=~ d_j
\end{equation}
for all $j \in \N$. We will use induction on $j$.
The case $j=1$ follows from Lemma \ref{Lem:convergence of 1st moment at t=1}.
Assuming that \eqref{eq:convergence of jth moment at t=1} holds for $j=1,\ldots,k-1$,
we will prove \eqref{eq:convergence of jth moment at t=1} for $j=k$.

From the decomposition \eqref{eq:X(t)=h(t)+X_*(t-xi)},
one can derive the following representation for $\E X(t)^k$:
\begin{equation}    \label{eq:EX(t)^k=int r_k dU}
\E X(t)^k   ~=~ \int_{[0,\,t]} r_k(t-y) \, U(\dy)
\end{equation}
where
\begin{equation*}
r_k(t) ~=~ \sum_{j=0}^{k-1}{k \choose j} h(t)^{k-j} \E X_*(t-\xi_1)^j \1_{\{\xi_1\leq t\}}
~=~ \sum_{j=0}^{k-1} v_j h(t)^{k-j} \E X(t)^j
\end{equation*}
for some real constants $v_j$, $0 \leq j \leq k-1$ with $v_{k-1} = k$.

If we can prove that
\begin{equation}    \label{eq:r_k/g^(k-1)h^k}
\lim_{t \to \infty} \frac{r_k(t)}{g(t)^{k-1} h(t)^k} = kd_{k-1},
\end{equation}
which, among other things, means that $r_k(t)$ is regularly
varying at $\infty$ with index $(k-1) \alpha - k \beta$,
then \eqref{eq:EX(t)^k=int r_k dU} in combination with the argument given in the proof of Lemma \ref{Lem:convergence of 1st moment at t=1} shows that
\begin{eqnarray*}
\lim_{t \to \infty} \frac{\E X(t)^k}{g(t)^k h(t)^k}
& = &
k d_{k-1}\lim_{t \to \infty} \frac{U(t)}{g(t)} \frac{\int_{[0,\,t]}r_k(t-y) \, U(\dy)}{r_k(t) U(t)} \\
& = &
\frac{\alpha k d_{k-1}}{\Gamma(1-\alpha)\Gamma(1+\alpha)} \int_0^1 (1-y)^{(k-1)\alpha-k\beta}y^{\alpha-1}\dy    \\
& = &
\frac{\Gamma(1-\beta+(k-1)(\alpha-\beta))}{\Gamma(1-\alpha)\Gamma(k(\alpha-\beta)+1)}kd_{k-1}=d_k.
\end{eqnarray*}
We now verify \eqref{eq:r_k/g^(k-1)h^k}. By the induction
hypothesis, for $j = 0, \ldots, k-1$, $\E X(t)^j$ is regularly
varying with index $j(\alpha-\beta)$. Hence, for such $j$'s,
$h(t)^{k-j} \E X(t)^j$ are regularly varying with indices
$j\alpha-k\beta$. Since $g(t)^{k-1} h(t)^k$ is regularly varying
with index $(k-1)\alpha-k\beta$, we conclude that
\begin{equation*}
\lim_{t \to \infty} \frac{\sum_{j=0}^{k-2} v_j h(t)^{k-j}\E X(t)^j}{g^{k-1}(t)h^k(t)}   ~=~ 0.
\end{equation*}
Hence
\begin{equation*}
\lim_{t \to \infty} \frac{\E X(t)^k}{g(t)^{k-1}h(t)^k}  ~=~ \lim_{t \to \infty} \frac{k \E X(t)^{k-1}}{g(t)^{k-1} h(t)^{k-1}}   ~=~ kd_{k-1},
\end{equation*}
which proves \eqref{eq:r_k/g^(k-1)h^k}.

{\sc Case 2:}
$\alpha=\beta$ and $c<\infty$.
\eqref{eq:convergence of moments} has been proved in Proposition \ref{Prop:exponential limit}
under weaker assumptions.
\end{proof}

\footnotesize
\noindent   {\bf Acknowledgements}  \quad
The research of A.\,M.~was financially supported by a Free Competition Grant of the Netherlands Organisation for Scientific Research (NWO).
The research of M.\,M.~was supported by DFG SFB 878 ``Geometry, Groups and Actions''.
A part of this work was carried out during visits of A.\,I.~and A.\,M.~to M\"{u}nster and a visit of M.\,M.~to Kiev.
Grateful acknowledgment is made for financial support and hospitality.
\normalsize

\begin{appendix}

\section{Appendix: Auxiliary results}   \label{sec:appendix}

\begin{Lemma}   \label{Lem:mean}
Let $W$ be a random variable with characteristic function given by \eqref{eq:stable ch f}.
Then, for $r < \alpha$,
\begin{equation*}
\E |W|^r ~=~ \frac{2\Gamma(r+1)}{\pi r} \sin{(r\pi/2)} \Gamma(1-r/\alpha)|\Gamma(1-\alpha)|^{r/\alpha} \cos{(\pi r/2-\pi r/\alpha)}.
\end{equation*}
In particular,
\begin{equation*}
\E |W|  ~=~ 2\pi^{-1}\Gamma(1-1/\alpha)|\Gamma(1-\alpha)|^{1/\alpha} \sin{(\pi/\alpha)}.
\end{equation*}
\end{Lemma}
\begin{proof}
We use the integral representation for the $r$th absolute moment
(see Lemma 2 in \cite{Bahr+Esseen:1965})
\begin{equation}    \label{eq:integral representation}
m_r ~:=~    \E|W|^r ~=~ \frac{\Gamma(r+1)}{\pi} \sin \Big(\frac{r \pi}{2}\Big) \int_{\R} \frac{1-{\rm Re}\,\E e^{{\rm i} t W}}{|t|^{r+1}} \, \dt.
\end{equation}
Set $A:=\pi^{-1}\Gamma(r\!+\!1)\sin{(r\pi/2)}$, $B:=\Gamma(1\!-\!\alpha) \cos{(\pi\alpha/2)}$ and
$C:=\Gamma(1\!-\!\alpha) \sin{(\pi\alpha/2)}$.
Using Euler's identity $e^{\i x} =\cos x +{\rm i}\sin x$ in \eqref{eq:stable ch f}, we obtain
\begin{equation*}
{\rm Re}\,\E e^{{\rm i}tW}
~=~ \exp{(-B|t|^{\alpha})}\cos{(-C|t|^{\alpha}{\rm sgn}(t))}.
\end{equation*}
Substituting this into formula \eqref{eq:integral representation} yields
\begin{equation*}
m_r
~=~
2A\int_0^{\infty} \frac{1-\exp{(-Bt^{\alpha})}\cos{(Ct^{\alpha})}}{t^{r+1}} \, \dt.
\end{equation*}
A change of variables ($u := t^{\alpha}$) gives
\begin{eqnarray}
m_r
&=&
\frac{2A}{\alpha}\int_0^{\infty}\big(1-\exp{(-Bu)}\cos{(Cu)}\big) u^{-1-r/\alpha} \, \du    \notag  \\
&=&
\frac{2A}{\alpha} \int_0^{\infty}\big(1-\exp{(-Bu)}\big) u^{-1-r/\alpha} \, \du \notag  \\
& &
+ \frac{2A}{\alpha} \int_0^{\infty}\big(\exp{(-Bu)}-\exp{(-Bu)}\cos{(Cu)}\big) u^{-1-r/\alpha} \, \du   \notag  \\
&=:&
I_1+I_2.    \label{eq:I_1+I_2}
\end{eqnarray}
Using \cite[Formula (3.945(2))]{Gradstein+Ryzhik:2000} and integration by parts, respectively, we obtain    
\begin{eqnarray*}
I_2 & = & \frac{2A}{\alpha} \, \Gamma(-r/\alpha)\Big(B^{r/\alpha}-|\Gamma(1-\alpha)|^{r/\alpha} \cos{(\pi r/2-\pi r/\alpha)}\Big);  \\
I_1
& = &
\frac{2A}{\alpha}\int_0^{\infty}\big(1-\exp{(-Bu)}\big) u^{-1-r/\alpha} \, \du  \\
& = &
\frac{2AB}{r}\int_0^{\infty}u^{-r/\alpha}\exp{(-Bu)} \du    \\
& = &
\frac{2AB^{r/\alpha}}{r} \Gamma(1-r/\alpha)
~=~ -\frac{2AB^{r/\alpha}}{\alpha} \Gamma(-r/\alpha).
\end{eqnarray*}
Now plugging in the values of $I_1$ and $I_2$ in \eqref{eq:I_1+I_2} gives
\begin{eqnarray*}
m_r
& = &
-\frac{2A}{\alpha}\Gamma(-r/\alpha)|\Gamma(1-\alpha)|^{r/\alpha} \cos{(\pi r/2-\pi r/\alpha)}   \\
& = &
\frac{2A}{r}\Gamma(1-r/\alpha)|\Gamma(1-\alpha)|^{r/\alpha} \cos{(\pi r/2-\pi r/\alpha)}.
\end{eqnarray*}
\end{proof}

\begin{Lemma}   \label{Lem:int X_t dnu_t->int X dnu}
Let $0 \leq a < b < \infty$.
Assume that $X_t(\cdot) \Rightarrow X(\cdot)$ as $t\to\infty$ in $D[a, b]$ in the $M_1$ topology.
Further, assume that $\nu_t$, $t \geq 0$ are finite measures such that $\nu_t \to \nu$ weakly as $t \to \infty$
where $\nu$ is a finite measure on $[a,b]$, which is continuous w.r.t.\ Lebesgue measure.
Then
\begin{equation*}
\int_{[a,\,b]} X_t(y) \, \nu_t(\dy)
~\stackrel{\mathrm{d}}{\to}~
\int_{[a,\,b]} X(y) \, \nu(\dy) \quad   \text{as } t \to \infty,
\end{equation*}
and if $X$ is a.s. continuous at some $c \in [a,b]$, then
\begin{equation*}
X_t(c) + \int_{[a,\,b]} X_t(y) \, \nu_t(\dy)
~\stackrel{\mathrm{d}}{\to}~    X(c) + \int_{[a,\,b]} X(y) \, \nu(\dy)  \quad   \text{as } t \to \infty.
\end{equation*}
\end{Lemma}
\begin{proof}
Use Lemma 6.5 in \cite{Iksanov:2012} and Skorokhod's representation theorem.
\end{proof}

\begin{Lemma}   \label{Lem:Fubini non product measure}
Let $(X_1,\A_1)$ and $(X_2,\A_2)$ be measurable spaces.
Let $\mu_1(\cdot)$ be a finite measure on $(X_1,\A_1)$.
Assume that  $\mu_2(x,\cdot)$ is a measure on $(X_2,\A_2)$ for every $x \in X_1$,
and that for every $B \in \A_2$ the function $X_1 \ni x \mapsto \mu_2(x,B)$ is measurable with respect to $\A_1$.
Then $\nu(B) := \int_{X_1} \mu_2(x,B) \mathrm{d}\mu_1(x)$ is a measure on $(X_2,\A_2)$.
Furthermore, for every non-negative measurable function $f$ on $(X_2,\A_2)$,
the function $g(x) := \int_{X_2} f(y) \mu_2(x,\mathrm{d}y)$ is measurable with respect to $\A_1$ and
\begin{equation}\label{eq:conditional Fubini}
\int_{X_2} f \, \mathrm{d}\nu ~=~ \int_{X_1} g \, \mathrm{d}\mu_1.
\end{equation}
\end{Lemma}
The proof of this lemma is a standard approximation argument:
check \eqref{eq:conditional Fubini} for indicators of sets in $\A_2$,
then for the finite linear combinations of such indicators with positive coefficients and finally for arbitrary non-negative measurable functions.

\begin{Lemma}   \label{Lem:multi integral}
Let $(X(\omega,u))_{u \in \R}$ be an arbitrary increasing random process
defined on probability space $(\Omega,\A,\Prob)$.
For fixed $k \in \N$ let $h:\R^{k} \to \R_+$ be a positive Borel function. Then
\begin{equation*}
\E \int_{\R^{k}} \!\! h(s_1,\ldots,s_k)\mathrm{d}X(s_1)\ldots\mathrm{d}X(s_k)
= \int_{\R^{k}} \!\! h(s_1,\ldots,s_k)\E \big(\mathrm{d}X(s_1)\ldots\mathrm{d}X(s_k)\big).
\end{equation*}
\end{Lemma}
\begin{proof}
Lemma \ref{Lem:Fubini non product measure} can be applied as follows.
Define $(X_1,\A_1) := (\Omega,\A)$ and $(X_2,\A_2):=(\mathbb{R}^k,\B)$
where $\B$ is the standard Borel $\sigma$-algebra of subsets of $\mathbb{R}^k$.
Set also $\mu_1:=\Prob$. For every $\omega\in\Omega$ put
\begin{equation*}
\mu_2(\omega,(a_1,\,b_1] \times \ldots \times (a_k,\,b_k])
:=  (X(\omega,b_1)-X(\omega,a_1)) \ldots (X(\omega,b_k)-X(\omega,a_k))
\end{equation*}
and extend this to a measure on $(X_2,\A_2)$.
This is possible since $X(\omega,u)$ is increasing for almost all $\omega$.
It is clear that $\omega \mapsto\mu_2(\omega,A)$ is measurable for every $A\in\B$. By Lemma \ref{Lem:Fubini non product measure}, $\nu(A)=\E \mu_2(\omega,A)$ is a measure and for any given positive Borel function $h:\R^{k}\to \R_+$,
the function $g: \Omega \to \R_+$ defined by
$g(\omega) := \int_{\R^k} h(s_1,\ldots,s_k) \mu_2(\omega,\ds_1, \ldots, \ds_k)$ is measurable and
\begin{align*}
\E \int_{\R^k} & h(s_1,\ldots,s_k) \, \mathrm{d}X(s_1)\ldots\mathrm{d}X(s_k)    \\
&=~ \E \int_{\R^k}h(s_1,\ldots,s_k) \, \mu_2(\omega,\ds_1\times\ldots\times \ds_k)  \\
&=~ \int_{\Omega} g(\omega) \, \Prob(\mathrm{d}\omega)
~\overset{\eqref{eq:conditional Fubini}}{=}~    \int_{\R^k}h(s_1,\ldots,s_k) \, \nu(\ds_1\times\ldots\times \ds_k)  \\
&=~ \int_{\R^{k}}h(s_1,\ldots,s_k) \, \E \big(\mathrm{d}X(s_1)\ldots\mathrm{d}X(s_k)\big).
\end{align*}
\end{proof}

\begin{Lemma}   \label{Lem:int<->sum}
Let $f:\R_+ \to \R_+$ be a decreasing function with $\lim_{t \to \infty} f(t) \geq 0$.
Then, for every $\theta>0$,
\begin{equation*}
\int_0^n f(\theta y) \dy ~=~ \sum_{k=0}^n f(\theta k) + \delta_n(\theta),   \quad   n \in \N,
\end{equation*}
where $\delta_n(\theta)$ converges as $n \to \infty$ to some $\delta(\theta) \leq 0$.
\end{Lemma}
\begin{proof}
We assume w.l.o.g.\ that $\theta = 1$.
For each $n \geq 1$,
\begin{eqnarray*}
\sum_{k=0}^n f(k) - \int_0^n f(y) \dy
& = &
\sum_{k=0}^{n-1}\left(f(k)-\int_{k}^{k+1} f(y) \dy \right) + f(n).
\end{eqnarray*}
Since $f$ is decreasing, each summand in the sum is
non-negative. Hence, the sum is increasing in $n$. On the other
hand, it can be bounded from above by
\begin{equation*}
\sum_{k=0}^{n-1} \left(f(k)-\int_k^{k+1} f(y) \dy \right)
~\leq~  \sum_{k=0}^{n-1} (f(k)-f(k+1))
~\leq~  f(0)    ~<~ \infty.
\end{equation*}
Consequently,
the series $\sum_{k\geq 0} \big(f(k)-\int_{k}^{k+1} f(y) \dy \big)$ converges.
Recalling that $\lim_{n \to \infty} f(n)$ exists completes the proof.
\end{proof}

\begin{Lemma}   \label{Lem:reduction of integrals}
Let $f_1, f_2: \R_+\to\R_+$ be bounded decreasing functions such that $f_1(t) \geq f_2(t)$ for all $t \in \R_+$
and such that $\int_0^{t_0} (f_1(y)-f_2(y)) \dy > 0$ for some $t_0 > 0$.
Then
\begin{equation}    \label{eq:reduction of integrals}
\sup_{t \geq t_0}\frac{\E \int_{[0,\,t]}(f_1(t-y)-f_2(t-y)) \, \mathrm{d}N(y)}{\int_{0}^{t}(f_1(y)-f_2(y)) \dy}
~<~\infty.
\end{equation}
\end{Lemma}
\begin{proof}
Decompose the integral in the numerator of the left-hand side of \eqref{eq:reduction of integrals} as follows:
\begin{equation*}
\int_{[0,\lfloor t \rfloor]} \!\! (f_1(t-y)-f_2(t-y)) \mathrm{d}N(y) + \int_{(\lfloor t \rfloor,t]} \!\! (f_1(t-y)-f_2(t-y)) \mathrm{d}N(y) =: I_1(t)+I_2(t).
\end{equation*}
By the distributional subadditivity of $N$, we get for $I_2(t)$:
\begin{eqnarray*}
I_2(t)
& \leq &
\int_{(\lfloor t \rfloor,\,t]}f_1(t-y) \, \mathrm{d}N(y)
~\leq~ f_1(0) (N(t)-N(\lfloor t \rfloor))   \\
& \leq &
f_1(0) (N(t)-N(t-1))
~\stackrel{\mathrm{d}}{\leq}~ f_1(0)N(1),
\end{eqnarray*}
hence $\E I_2(t) < f_1(0) \E N(1) < \infty$ for all $t \geq 0$.
It remains to consider $I_1(t)$:
\begin{eqnarray*}
\E I_1(t)
& = &
f_1(t)-f_2(t) + \E \sum_{j=0}^{\lfloor t \rfloor-1}\int_{(j,\,j+1]}(f_1(t\!-\!y)-f_2(t\!-\!y)) \, \mathrm{d}N(y)    \\
& \leq &
f_1(0) + \sum_{j=0}^{\lfloor t \rfloor-1}(f_1(t\!-\!j\!-\!1)-f_2(t\!-\!j)) \E (N(j\!+\!1)-N(j)) \\
& \leq &
f_1(0) + \sum_{j=0}^{\lfloor t \rfloor-1} (f_1(t\!-\!j\!-\!1)-f_2(t\!-\!j)) \E N(1) \\
& = &
\E N(1) \bigg(\int_0^{\lfloor t \rfloor} (f_1(y)-f_2(y) ) \dy + O(1)\bigg).
\end{eqnarray*}
\end{proof}
\end{appendix}

\end{document}